\documentclass[a4paper,twoside]{amsart}
\usepackage{pdfsync}

\usepackage{amsmath,amsthm,amsfonts,latexsym,amscd,amssymb,enumerate,mathrsfs}
\usepackage{graphicx}
\usepackage{comment}
\usepackage{extarrows}
\usepackage[backrefs,?,numeric]{amsrefs}
\usepackage[pdftex,pdftitle={Numeric rho invariants}, backref,pdftex]{hyperref}

\usepackage[all]{xy}

\swapnumbers
\theoremstyle{plain}
\newtheorem{theorem}{Theorem}[section]
\newtheorem{lemma}[theorem]{Lemma}
\newtheorem{assumption}[theorem]{Assumption}

\newtheorem{proposition}[theorem]{Proposition}

\theoremstyle{definition}
\newtheorem{definition}[theorem]{Definition}

\newtheorem{set-up}[theorem]{Geometric set-up}
\newtheorem{remark}[theorem]{Remark}

\DeclareMathOperator{\cyl}{cyl}

\DeclareMathOperator{\Ind}{Ind}

\newcommand{\forget}[1]{}

\def  \nuint {\raise10pt\hbox{$\nu$}\kern-6pt\int}

\newcommand\Tr{\operatorname{Tr}}

\newcommand\C{\mathcal C}

\def \Sp {{\cal S}}

\newcommand\Di{D\kern-6pt/}
\newcommand\cDi{{\mathcal D}\kern-6pt/}
\newcommand\spi{S\kern-6pt/}
\newcommand \cspi{\Sp\kern-6pt/}

\newcommand\CC{\mathbb C}

\def \cal {\mathcal}
\def \C {{\cal C}}

\newcommand\RR{\mathbb R}
\newcommand\ZZ{\mathbb Z}

\newcommand\Ker{\operatorname{Ker}}

 \usepackage{xcolor}
\definecolor{darkgreen}{cmyk}{1,0,1,.2}
\definecolor{m}{rgb}{1,0.1,1}
\long\def\red#1{\textcolor {red}{#1}} 

\long\def\blue#1{\textcolor {blue}{#1}}

{\catcode`@=11\global\let\c@equation=\c@theorem}

\allowdisplaybreaks[2]

\begin{document}
\pagestyle{myheadings}
\markboth{Paolo Piazza, Hessel Posthuma, Yanli Song and Xiang Tang }{Higher orbital integrals, rho numbers and  index theory}

\title{Delocalized eta invariants of the signature operator on G-proper manifolds}

\author{Paolo Piazza}
\address{Dipartimento di Matematica, Sapienza Universit\`{a} di Roma, I-00185 Roma, Italy}
\email{piazza@mat.uniroma1.it}
\author{Hessel Posthuma}
\address{Korteweg-de Vries Institute for Mathematics, University of Amsterdam, 1098 XG Amsterdam, the Netherlands}
\email{H.B.Posthuma@uva.nl}
\author{Yanli Song}
\address{Department of Mathematics and Statistics, Washington University, St. Louis, MO, 63130, U.S.A.} 
\email{yanlisong@wustl.edu}
\author{Xiang Tang}
\address{Department of Mathematics and Statistics, Washington University, St. Louis, MO, 63130, U.S.A.}
\email{xtang@math.wustl.edu}



\maketitle

\begin{abstract}
Let $G$ be a connected, linear real reductive group and let $X$ be a cocompact $G$-proper manifold without
boundary. 
We define delocalized eta invariants associated to a $L^2$-invertible perturbed  Dirac operator $D_X+A$ 
with $A$ a suitable smoothing perturbation. We also investigate the case in which $D_X$ is not invertible but $0$ is isolated in the $L^2$-spectrum of $D_X$.  We  prove index formulas relating these delocalized eta invariants to Atiyah-Patodi-Singer 
delocalized indices on  $G$-proper manifolds with boundary. In order to achieve this program 
we give a detailed account of both the large and small time behaviour of the heat-kernel of perturbed Dirac operators, as a map from the positive real line to the algebra of Lafforgue integral operators. 
We apply these results to the definition of rho-numbers associated to  $G$-homotopy equivalences 
between closed $G$-proper manifolds and to the study of their bordism properties. We also define delocalized
signatures of manifolds with boundary satisfying an invertibility assumption on the differential form Laplacian of the boundary in middle degree and prove an
Atiyah-Patodi-Singer formula for these delocalized signatures. 
		\end{abstract}

\tableofcontents

\section{Introduction}

Let $G$  be a connected, linear real reductive group. Let $K$ be a maximal compact subgroup of $G$.
This article complements our article {\it Higher orbital integrals, rho numbers and index theory}
\cite{PPST}. There, among other contributions, we gave a detailed study of the convergence problem for the delocalized 
eta invariant associated to an equivariant $L^2$-invertible Dirac operator  on a G-proper manifold without boundary
and also showed that under an $L^2$-invertibility assumption on the boundary operator, such an invariant
enters as a boundary correction term in a delocalized Atiyah-Patodi-Singer index theorem. This allowed us to
introduce in the spin case rho-numbers associated to $G$-equivariant metrics of positive scalar curvature 
and to study their bordism properties.  

More precisely,
we established in \cite{PPST} the following two results:

 \begin{theorem}\label{thm etadefine-intro}
Let $(X,\mathbf{h})$ be a cocompact $G$-proper manifold without boundary 
 and let $D_X$ be a $G$-equivariant Dirac-type operator (associated to a unitary
 Clifford action and a Clifford connection). We assume that $D_X$ is $L^2$-invertible. Let $g\in G$ be a semi-simple element and let $\tau_g$
 the associated orbital integral on Lafforgue's algebra\footnote{We actually worked with the Harish-Chandra Schwartz algebra in \cite{PPST} and the rapid decay algebra in \cite{PP1,PP2}. Our proofs in \cite{PPST, PP1,PP2} generalize directly to Lafforgue's algebra $\mathcal{L}_s(G)$ which is a dense (Banach) subalgebra of $C^*_r(G)$ stable under holomorphic functional calculus, c.f. \cite{Lafforgue}.} $\mathcal{L}_{s}(G)$ (see Definition \ref{defn:laffargue}). Let $\tau^X_g$ be the associated trace 
 on the algebra of Lafforgue integral operators $\mathcal{L}^\infty_{G,s} (X)$, a dense subalgebra of the Roe $C^*$-algebra $C^* (X)^G$.
 Then the  integral 
\begin{equation}\label{intro:eta-large t}
\eta_g (D):=\frac{1}{\sqrt{\pi}} \int_0^\infty \tau^{X}_g (D_X\exp (-tD_{X}^2) \frac{dt}{\sqrt{t}}
\end{equation}
converges. 
\end{theorem}

 \begin{theorem}\label{intro:0-delocalized-aps}
  Let $(Y_0, \mathbf{h}_0)$ be a $G$-proper manifold
 with boundary with compact quotient and let $(Y,\mathbf{h})$ be the associated manifold with cylindrical ends.
 We assume that the boundary operator  $D_{\partial}$ is $L^2$-invertible.  Let $D_0$  be a generalized Dirac operator on $(Y_0, \mathbf{h}_0)$ and denote by $D$ the associated 
 operator on $(Y,\mathbf{h})$.
 Then:
 \begin{itemize}
\item[1)] there exists a dense holomorphically closed
 subalgebra $\mathcal{L}^\infty_{G,s} (Y,E)$ of the Roe algebra $C^* (Y_0\subset Y,E)^G$ and a smooth index class 
 $\Ind_\infty (D)\in K_0(\mathcal{L}^\infty_{G,s} (Y, E)) \cong
 K_0(C^*(Y_0\subset Y,E)^G)$;\\
 \item[2)] for the pairing of the index class $\Ind(D_\infty)$ with the $0$-cocycle  $\tau^Y_g\in HC^0 (\mathcal{L}^\infty_{G,s} (Y,E))$
 defined by the orbital integral $\tau_g$  
 the following delocalized 0-degree index formula holds:
 \begin{equation}
\label{main-0-degree}
 \langle \tau^Y_g,\Ind_\infty (D) \rangle= \int_{(Y_0)^g} c^g {\rm AS}_g (D_0) - \frac{1}{2} \eta_g (D_{\partial })\,,
 \end{equation}
where $c^g$ is a cut-off function for the fixed-point manifold
$(Y_0)^g$ and where the Atiyah-Segal delocalized integrand $ {\rm AS}_g(D_0)$ appears.
\end{itemize}
 
 \end{theorem}
 
 \noindent
This second statement was previously  established by different methods in \cite{HWW2} 
but under  the rather restrictive hypothesis that  $G/Z_G (g)$ is compact. 

\medskip
In this paper our main interest lies in the  signature operator $D^{{\rm sign}}$. Since the signature operator is rarely $L^2$-invertible, Theorem \ref{thm etadefine-intro} and \ref{intro:0-delocalized-aps} do not apply directly to $D^{{\rm sign}}$. So we investigate in this article under which conditions we can lift the invertibility assumption on $D_X$ in Theorem \ref{thm etadefine-intro} and on the boundary operator $D_\partial$ in Theorem \ref{intro:0-delocalized-aps}.
As in previous work on the subject, our position is that 
such an assumption can indeed be lifted, provided we are able to find a suitable perturbation 
of $D_X$ or $D_\partial$ that makes it $L^2$-invertible. The main issue is of course
the following:  what kind of perturbation can one allow? Indeed, we focus on a perturbation for which we are able to
prove a delocalized Atiyah-Patodi-Singer index theorem. 

We  consider the following two situations.
\begin{enumerate}
\item There exists an element $C$  in $\mathcal{L}^c_G (X,E)$, the $G$-equivariant Lafforgue integral  operators of $G$-compact support,  making the operator $D_X+C$ $L^2$-invertible; under this assumption  we show that 
\begin{equation}\label{statement-convergence}
\text{the delocalized eta invariant}\;\; \eta_g (D_X+C)\;\; \text{is well defined.}
\end{equation}
\item The operator $D_X$ is such that $0$ is isolated in the spectrum.
Under this assumption we prove that the delocalized eta invariant 
$\eta_g (D)$ is well defined; moreover, as a preparation for the Atiyah-Patodi-Singer we shall establish later
in the article,
we consider 
$\eta_g (\Theta)$, $\Theta$ a positive real number, defined  in terms of a global 0th-order 
perturbation $D_\Theta$ of the operator; 
we study the limit of  $\eta_g (\Theta)$ as $\Theta\downarrow 0$ and prove that it is equal to
the sum of $\eta_g (D)$ and the term $\tau^X_g (\Pi_{\Ker D})$, which we prove to be finite;\\
\end{enumerate}

Crucial for all these results  is a very detailed study of the large  and short time behaviour of the heat kernel of a perturbed Dirac operator {\em as a map from $\RR^+$ to a suitable Lafforgue algebra of integral operators}. As we shall see, results that are well known in the compact case often require
alternative and not so obvious arguments.

Once all these results are established on a $G$-proper manifold without boundary, we investigate   how  these two invariants fit into suitable 
Atiyah-Patodi-Singer index theorems for Dirac operator with {\it non-invertible}
boundary operator. We now state these results. 
%

First, we have the following APS index theorem, connected with the delocalized eta invariant introduced above, $\eta_g (D_X+C)$.

\begin{theorem}\label{intro:smoothing-perturbation}
Let $C_\partial $ be  a Lafforgue integral  operator of $G$-compact support
such that $D_\partial + C_\partial$ is $L^2$-invertible.
Then there exists 
a well defined 
index class $\Ind (D,C_\partial)\in K_0 (\mathcal{L}^\infty_{G, s} (Y))=K_0 (C^* (Y_0\subset Y, E)^G$. Moreover,
for the delocalized index $\langle [\tau^Y_g], \Ind (D,C_\partial) \rangle$ the following formula holds
\begin{equation*}
\langle [\tau^Y_g], \Ind (D,C_\partial) \rangle= \int_{(Y_0)^g} c^g {\rm AS}_g (D_0) - \frac{1}{2} \eta_g (D_{\partial}+C_\partial).
\end{equation*}
\end{theorem}

Using the  results about $\eta_g (D_X+C)$ 
and the above index theorem 
we are able to give an exhaustive treatment of the signature operator $D^{{\rm sign}}$ on $G$-proper manifolds. As explained above, this is our main motivation for establishing the above general results.
In particular, we introduce  rho-invariants $\rho_g ({\bf f})$ associated to  a
$G$-homotopy equivalence  ${\bf f}: X\to Y$  between two cocompact $G$-proper manifolds without boundary
and we study their
bordism properties;
we 
employ crucially the Hilsum-Skandalis perturbation and \eqref{statement-convergence}
for the definition of  $\rho_g ({\bf f})$ and Theorem \ref{intro:smoothing-perturbation}
for its bordism properties. Next, under an invertibility assumption on the differential form Laplacian in middle degree,
we introduce a delocalized eta invariant $\eta_g (D^{{\rm sign}}_{\mathbf{h}})$ on a cocompact $G$-proper manifold
without boundary $(X,\mathbf{h})$; we use crucially {\it symmetric} perturbations as in \cite{LP-AGAG}
 \cite{Wahl-product}.
We then pass to a cocompact $G$-proper manifold with boundary $(Y_0,\mathbf{h}_0)$;
we show that under the same assumption on the boundary Laplacian we have a well-defined delocalized
signature $\sigma_g (Y_0,\partial Y_0, \mathbf{h}_0)$ and we give a Atiyah-Patodi-Singer signature formula for it, in terms
of a local integral on fixed point sets and the delocalized eta invariant $\eta_g (D^{{\rm sign}}_{\partial \mathbf{h}_0})$.

As far as delocalized Atiyah-Patodi-Singer index theory is concerned we also analyze the case in which 0 is isolated in the spectrum  of the boundary operator $D_\partial$.  Following the notation  in Theorem  \ref{intro:0-delocalized-aps},  we prove the following result:  

\begin{theorem}\label{intro:gap-theo}
Assume that the boundary 
operator on $\partial Y$ satisfies  
$$
\exists\; \delta >0\;\;\text{such that}\;\; {\rm spec}_{L^2} (D_\partial)\cap (-\delta,\delta)=\{0\}.
$$
Then for $\Theta\in (0,\delta)$ there exists a well defined smooth index class $\Ind_\infty (D_\Theta)\in K_0 (\mathcal{L}^\infty_{G, s}(Y,E))=K_0 (C^* (Y_0\subset Y, E)^G)$, independent of  $\Theta\in (0,\delta)$.
Moreover, for its pairing with the 0-degree cyclic cocycle $\tau_g^Y$ the following formula holds:
\[
\langle \Ind_\infty (D_\Theta),\tau^{Y}_{g}\rangle =
\int_{(Y_0)^g} c^g {\rm AS}_g (D_0)-\frac{1}{2}(\eta_g (D_\partial) + \tau^{\partial Y}_g (\Pi_{\ker D_\partial})).
\]
\end{theorem}

\bigskip
\noindent
 {\bf The paper is organized as follows.} In Section \ref{sect:preliminaries} we recall the results in \cite{PPST}.
In Section \ref{section:heat}  we state results about the short and large time behaviour of the heat kernel of
a perturbed operator, both on closed manifolds and on manifolds with boundary. These results, necessarily of a technical
nature, play a crucial role in our analysis; the proofs, that are somewhat long,  are given in a separate section, Section \ref{sect:proofs-heat}.
 In Section \ref{sect:modified} we introduce and study the delocalized eta invariant $\eta_g (D_X+C)$;
 Section \ref{sect:aps-perturbed} is devoted to the proof of Theorem \ref{intro:smoothing-perturbation}, the delocalized APS index theorem corresponding to $\eta_g (D_{\partial}+C)$.
In Section \ref{sect:numeric} we apply these results to the signature operator:
we introduce rho-numbers associated to $G$-homotopy equivalences
and study their bordism-invariance properties. We also discuss delocalized signatures of 
$G$-proper manifolds with boundary such that the differential forms Laplacian on the boundary is invertible
in middle degree. Finally, 
 Section \ref{sect:further} is devoted to the gap case.

\bigskip
\noindent
{\bf Acknowledgements.} We would like to thank Pierre Albin, Jean-Michel Bismut,
Peter Hochs, Yuri Kordyukov, Xiaonan Ma, Shu Shen, and 
Weiping Zhang for  inspiring discussions. Piazza was partially supported by 
the PRIN {\it Moduli spaces and Lie Theory} of MIUR ({\it Ministero Istruzione Universit\`a Ricerca})
and by {\it Istituto Nazionale di Alta Matematica Francesco Severi}; Song was partially supported by NSF Grant DMS-1800667 and DMS-1952557; Tang was partially supported by NSF grants DMS-1952551, DMS-2350181, and Simons Foundation grant MPS-TSM-00007714.

\section{Index classes and cyclic cocycles}\label{sect:preliminaries}

In this section we shall briefly recall the main characters of our paper  \cite{PPST}.
 We shall be brief and refer the reader
to that paper for details.

\subsection{Definition of some algebras}
 We consider first a co-compact $G$-proper manifold {\it without} boundary $X$.
We know \cite{Abels} that 
there exists a compact submanifold
$S\subset X$ on which the $G$-action restricts to an action of a maximal compact subgroup $K\subset G$, so that 
the natural map
\[
G\times_K S\to X,\quad [g,x]\mapsto g\cdot x,
\]
is a diffeomorphism. Let $\mathcal{L}_G^c (X)$ be the algebra of $G$-equivariant smoothing operators
of $G$-compact support;  this decomposition of $X$ induces an isomorphism
\begin{equation}
\label{algebra-slice}
\mathcal{L}_G^c (X)\cong \left(C^\infty_c(G)\hat{\otimes}\Psi^{-\infty}(S)\right)^{K\times K}
\end{equation}
and more generally
\begin{equation}
\label{algebra-slice-bis}
\mathcal{L}_G^c (X,E)\cong \left(C^\infty_c(G)\hat{\otimes}\Psi^{-\infty}(S,E|_S)\right)^{K\times K},
\end{equation}
in the presence of a $G$-equivariant vector bundle $E$.
(We shall often expunge the vector bundle $E$ from the notation.)

\begin{definition}\label{defn:laffargue}
For $s \in [0, \infty)$, define the Lafforgue algebra $\mathcal{L}_s(G)$ to be the completion of $C_c(G)$ with respect to the norm $\nu_s$ defined as follows
\[
\nu_s (f) \colon = \sum_{g\in G} \left\{(1 + \|g\|)^s \cdot \Xi^{-1}(g)\cdot |f(g)|\right\},
\]	
where $\Xi(g)$ denotes the Harish-Chandras' spherical function. Moreover, we define
\[
	\mathcal{L}^\infty_{G,s} (X):= \left(\mathcal{L}_s(G)\hat{\otimes}\Psi^{-\infty}(S)\right)^{K\times K}.
	\]    
\end{definition}

The completion of the Lafforgue algebra $\mathcal{L}^\infty_{G,s}(X)$ under the operator norm $\|\cdot \|_{L^2(X, E)}$ is the Roe $C^*$-algebra $C^*(X, E)^G$. Let $e$ be a $K\times K$-invariant rank 1 projection in $\Psi^{-\infty}(S)$ associated with the $K$-invariant constant function 1 on $S$. Then $\mathcal{L}^\infty_{G,s}(X)(Id\otimes e):=\{T(Id\otimes e)|\ T\in \mathcal{L}^{\infty}_{G,s}(X)\}$  has a natural $\mathcal{L}^\infty_{G,s}(X)$-$\mathcal{L}_s(G)$ bimodule structure and carries a $C^*_r(G)$-valued inner product $\langle\ ,\ \rangle_{C^*_r(G)}$ defined by the trace $\Tr$ on $\Psi^{-\infty}(S)$. The completion of $\mathcal{L}^\infty_{G,s}(X)(Id\otimes e)$ under the inner product $\langle\ ,\ \rangle_{C^*_r(G)}$ defines an element $[e]$ in $KK(C^*(X, E), C^*_r(G))$. Product with the element $[e]$ gives a group homomorphism 
\[
[e]: K_\bullet\big( C^*(X, E)^G\big)\longrightarrow K_\bullet \big(C^*_r(G)\big).
\]
A Lie group version of \cite[Chapter 3, 4.$\gamma$, Lemma 8]{Connes94} and \cite{Roe02} relates the index $\Ind(D)\in K_0(\mathcal{L}^\infty_{G.s}(X, E))\cong K_0(C^*(X, E)^G)$ to the $C^*_r(G)$-index $\Ind_{C^*_r(G)}(D)\in K_0(C^*_r(G))$, 
\[
[e](\Ind(D))=\Ind_{C^*_r(G)}(D).
\]
We remark that in general $C^*(X, E)^G$ is not Morita equivalent to the (reduced) group $C^*$-algebra $C^*_r(G)$. And the map $[e]: K_\bullet\big( C^*(X, E)^G\big)\longrightarrow K_\bullet \big(C^*_r(G)\big)$ is not an isomorphism of abelian groups. For example, let $K$ be a compact group acting trivially on a compact manifold $M$. Let $E$ be a $K$-equivariant vector bundle $E$ on $M$. Since the $K$-action is trivial and $M$ is compact, $C^*(M, E)^G$ is the algebra $\mathcal{K}(M, E)$ of compact operators on $L^2(M, E)$, which is not Morita equivalent to $C^*(K)$. And the map $[e]$ on $K$-theory groups fails to be an isomorphism. On the other hand, let $H$ be a $G$-equivariant admissible module on $X$; under this assumption Guo, Hochs, and Mathai  \cite{GuoHochsMathai} showed that $C^*(X, E)^G$ is Morita equivalent to $C^*_r(G)$. \\

Throughout the article we will fix a sufficiently large $s\gg0$ and work with $\mathcal{L}_{s}(G)$ and $\mathcal{L}_{G, s}^\infty(X)$.  
 \subsection{Orbital integrals}
 \label{oiclosed}
 Let us fix a semisimple element $g\in G$. Under this  assumption we know that $G/Z_g$, with $Z_g$ denoting
the centralizer of
$g$, has a $G$-invariant Haar measure $d(hZ_g)$. Recall that the orbital integral of a function $f\in C_c (G)$
is, by definition,
\[
\tau_g (f):=\int_{G/Z_g} f(hgh^{-1}) d(h Z_g),
\]
which {\it extends}
to a continuous trace homomorphism 
\begin{equation}\label{orbital-trace}
\tau_g: \mathcal{L}_{s}(G)\to\CC\,.
\end{equation} 
Following  \cite{Hochs-Wang-KT}, we define 
\begin{equation}\label{tr-g-closed}
\tau_g^X (T):= \int_{G/Z_g}\int_X c(hgh^{-1}x) {\rm tr} (hgh^{-1}\kappa (hg^{-1}h^{-1}x,x))dx \,d(hZ)
\end{equation}
with $c$ a cut-off function for the action of $G$ on $X$ and ${\rm tr}$ denoting the vector-bundle fiberwise
trace. Then $\tau_g^{X}$ defines a continuous 
 trace
 \begin{equation}\label{delocalized-trace-on-closed}
\tau_g^{X}: \mathcal{L}^\infty_{G,s} (X,E)\to\CC\,.
\end{equation}

\subsection{$G$-proper manifolds with boundary}
Let now $Y_0$ be a  manifold with boundary; we assume again that  $G$ 
acts properly
and cocompactly on $Y_0$. We denote by $X$ the boundary of $Y_0$. 
We endow $Y_0$ with a $G$-invariant metric $\mathbf{h}_0$ which is of product type near the boundary. We let
$(Y_0,\mathbf{h}_0)$ be the resulting Riemannian  manifold with boundary. 
We denote by $c_0$ a cut-off function for the action of $G$ on $Y_0$; since the action is cocompact, this can be chosen to be a compactly supported smooth function.
We consider the associated manifold with cylindrical ends
$\widehat{Y}:= Y_0\cup_{\partial Y_0} \left(   (-\infty,0] \times \partial Y_0 \right)$,
endowed with the extended metric $\widehat{\mathbf{h}}$ and the extended $G$-action.
We denote by  $(Y,\mathbf{h})$ the $b$-manifold associated to  $(\widehat{Y},\widehat{\mathbf{h}})$. We shall often treat 
$(\widehat{Y},\widehat{\mathbf{h}})$ and $(Y,\mathbf{h})$ as the same object. We denote by $c$ the  obvious extension of the cut-off function
$c_0$ for the action of $G$ on $Y_0$ (constant along the cylindrical end); this is a cut-off function of the extended action of $G$ on $Y$. Let us  fix a slice $Z_0$ for the $G$ action on $Y_0$; thus
$$Y_0\cong G\times_K Z_0$$
with $K$ a maximal compact subgroup of $G$ and $Z_0$ a smooth compact manifold  with boundary, denoted by $S = \partial Z_0$, endowed with a $K$-action. 
Clearly  for the boundary $X = \partial Y_0$ we have
\[
X  \cong G \times_K S. 
\] 
Notice that $Y\cong G\times_K Z$ with $Z$ the $b$-manifold associated to $Z_0$. As in \cite{PPST}
 we assume that 
\begin{enumerate}
	\item the symmetric space $G/K$,
	\item the $G$-manifold $Y_0$,
	\item the $K$-slice $Z_0$
\end{enumerate}
are all even dimensional. 

\subsection{Manifolds with boundary: the Lafforgue algebras and index classes}
Let $(Y_0,\mathbf{h}_0)$  be  a cocompact $G$ proper manifold with boundary $X$ as in the previous subsection
and let $(Y,\mathbf{h})$   be the associated $b$-manifold. We denote by $Z_0$ a slice for $Y_0$
and by $Z$ the associated $b$-manifold.

Using the $b$-calculus with $\epsilon$-bounds on $Z$ and the algebra ${\mathcal{L}}^\infty_{G,s}(Y)$ one can employ Abels' theorem \cite{Abels}
in order
to define algebras of  operators on $Y$
fitting  into a short exact sequence of operators:
\begin{equation}\label{b-short-c}0\to  \mathcal{L}^c_G (Y)\to 
{}^b \mathcal{L}^c_G (Y)\xrightarrow{I} {}^b \mathcal{L}^c_{G,\RR} ({\rm cyl}(\partial Y))\to 0, 
\end{equation}
with $\cyl (\partial Y)=\RR\times \partial Y$. 
Here $I$ denotes the indicial operator. See \cite{PP2} and references therein. Notice that we have and we shall omit the $\epsilon$ from the notation of the calculus with bounds.

Employing {Lafforgue's algebra  $\mathcal{L}^\infty_{s}(G)$ instead of $C^\infty_c (G)$ 
we can define as in \cite{PP2} the three algebras  $\mathcal{L}^\infty_{G,s} (Y)$, ${}^b \mathcal{L}^\infty_{G,s} (Y)$ and ${}^b \mathcal{L}^\infty_{G,s, \RR} ({\rm cyl}(\partial Y))$
fitting into the short exact sequence
\begin{equation}\label{b-short-HC}0\to  \mathcal{L}^\infty_{G,s} (Y)\to {}^b \mathcal{L}^\infty_{G,s} (Y)\xrightarrow{I} {}^b \mathcal{L}^\infty_{G,s,\RR} ({\rm cyl}(\partial Y))\to 0.\end{equation}
This extends \eqref{b-short-c}. Recall at this point, see \cite{PP2}, that $ \mathcal{L}^\infty_{G,s} (Y)$, also called the residual algebra, is dense and holomorphically closed in the Roe algebra $C^*(Y_0\subset Y)^G$.




\medskip
Let now $D$ be an equivariant Dirac operator, of product type near the boundary.
We shall make for the time being the following assumption:
\begin{equation}\label{assumption}
 \text{the boundary operator}\;\;
D_{\partial Y}\;\; \text{is}\;\;L^2\text{-invertible}.
\end{equation}

Let $D$ be as above and let $Q^\sigma$ be a symbolic $b$-parametrix for $D$.
Following Melrose, see \cite[Section 5.13]{Melrose-Book}, we can use a {\it true} $b$-parametrix $Q^b=Q^\sigma-Q^\prime$ for $D$ in order to define a Connes-Skandalis index class
\begin{equation}\label{CS-class-bis}
\Ind_\infty (D):=[P^b_{Q}] - [e_1]\in K_0(\mathcal{L}^\infty_{G,s} (Y))\equiv 
K_0 (C^*(Y_0\subset Y)^G)
\;\;\;\text{with}\;\;\;e_1:=\left( \begin{array}{cc} 0 & 0 \\ 0&1
\end{array} \right).
\end{equation}
This class can also be obtained \`a la Connes-Moscovici, by improving the parametrix $ Q
:= \frac{I-\exp(-\frac{1}{2} D^- D^+)}{D^- D^+} D^+$ to a true $b$-parametrix. This defines a projector
$V^b (D)$, a $2\times 2$ matrix with entries
in  (the unitalization of) $\mathcal{L}^\infty_{G,s} (Y)$ and one can prove that  $[V^b (D)]  - [e_1]= [P^b_{Q}] - [e_1]$ in $K_0(\mathcal{L}^\infty_{G,s} (Y))$.
We can also consider  the plain Connes-Moscovici projector $V(D)$,  
\[
V(D):=\left( \begin{array}{cc} e^{-D^- D^+} & e^{-\frac{1}{2}D^- D^+}
\left( \frac{I- e^{-D^- D^+}}{D^- D^+} \right) D^-\\
e^{-\frac{1}{2}D^+ D^-}D^+& I- e^{-D^+ D^-}
\end{array} \right).
\]
It is a $2\times 2$ matrix with entries in ${}^b\mathcal{L}^\infty_{G,s} (Y_0)$; similarly, 
the Connes-Moscovici projector $V(D^{\cyl})$ is a $2\times 2$ matrix with entries in ${}^b\mathcal{L}^\infty_{G,s,\RR} ({\rm cyl}(\partial Y))$. 
The first two authors proved in \cite{PP2}, following \cite{moriyoshi-piazza}, that these two projectors define a smooth relative index class  $\Ind_\infty (D,D_\partial )\in 
K_0({}^b \mathcal{L}^\infty_{G,s} (Y_0),{}^b \mathcal{L}^\infty_{G,s,\RR} ({\rm cyl}(\partial Y)))$; indeed
$\Ind_\infty (D,D_\partial ): =[V(D), e_1,p_t]$ 
with $p_t $, $t\in [1,+\infty]$, defined as follows

\begin{equation}\label{pre-wassermann-triple}
p_t:= \begin{cases} V(t (D_{\cyl}))
\;\;\quad\text{if}
\;\;\;t\in [1,+\infty)\\
e_1 \;\;\;\;\;\;\;\;\;\;\;\;\;\,\text{ if }
\;\;t=\infty
 \end{cases}
\end{equation}
 Moreover, the class $\Ind_\infty (D)$ is sent to the class $\Ind_\infty (D,D_\partial )$ through the excision isomorphism
$\alpha_{{\rm exc}}$. Notice that in \eqref{pre-wassermann-triple} we use implicitly the large time behaviour of the heat kernel
and, more generally, of the Connes-Moscovici projector, as a map
with values in $2\times 2$ matrices with entries in ${}^b\mathcal{L}^\infty_{G,s,\RR} ({\rm cyl}(\partial Y))$; this was discussed in detail in \cite{PPST}. (This point was treated very briefly in \cite{PP2}.) 

%
%

\subsection{Cyclic cocycles associated to orbital integrals on manifolds with boundary}
\label{section:0cocycle}
As explained in \cite{PPST}, the orbital integral $\tau^g$ defines a trace-homomorphism 
on the residual algebra \begin{equation}\label{delocalized-trace-on-M}
\tau_g^Y: \mathcal{L}^\infty_{G,s} (Y)\to\CC\,
\end{equation} 
exactly as in the closed case
\begin{equation}\label{tr-g-b}
\tau_g^Y (T):= \int_{G/Z_g}\int_Y c(hgh^{-1}x) {\rm tr} (hgh^{-1}\kappa (hg^{-1}h^{-1}x,x))dx \,d(hZ)
\end{equation}
with $dx$ denoting now the $b$-volume form associated to the $b$-metric $\mathbf{h}$. The trace $\tau^Y_g$ defines a cyclic 0-cocycle
 on the algebra $ \mathcal{L}^\infty_{G,s} (Y)$. Using the pairing between $K$-theory and cyclic cohomology, denoted by 
 $\langle\,,\,\rangle$, we have in our case
 \begin{equation}\label{pairing}
 \langle\cdot\,,\cdot\,\rangle: HC^0 (\mathcal{L}^\infty_{G,s} (Y))\times K_0(\mathcal{L}^\infty_{G,s} (Y)) \to \CC
 \end{equation}
 and thus a homomorphism
 \begin{equation}\label{pairing-hom}
 \langle\tau^Y_g,\cdot\rangle: K_0(\mathcal{L}^\infty_{G,s} (Y))\to \CC.
 \end{equation}

 Following the relative cyclic cohomology approach in \cite{moriyoshi-piazza,GMPi,PP2} we defined in \cite{PPST}
 a relative cyclic 0-cocycle $(\tau^{Y,r}_g,\sigma_g)$.
(The $r$ on the right hand side stands for {\it regularized}, given that the $b$ is the differential in cyclic cohomology.)

The 1-cocycle $\sigma^{\partial Y}_g$ can in fact be defined on any  proper $G$-manifold $X$ without boundary;
it is the following  1-cochain on ${}^b \mathcal{L}^c_{G,\RR} ({\rm cyl}(X))$
\begin{equation}\label{1-eta}
\sigma^X_g (A_0,A_1)=\frac{i}{2\pi}\int_\RR\tau_g^X ( \partial_\lambda I(A_0,\lambda)\circ I(A_1, \lambda)) d\lambda\,,
\end{equation}
where the indicial family of $A\in {}^b \mathcal{L}^c_{G,\RR} ({\rm cyl}(X))$, denoted $I(A,\lambda)$, appears. 
Then $\sigma^X_g (\,,\,)$ is well-defined  and  a cyclic 1-cocycle.

Going back now to a $b$-manifold $Y$ with  boundary   $\partial Y$, we have also defined in \cite{PPST}
 the functional  $\tau^{Y,r}_g (\cdot): {}^b \mathcal{L}^c_{G} (Y)\to \CC$:
$$\tau^{Y,r}_g (T):= 
 \int_{G/Z_g}\int^b_Y c(hgh^{-1}y) {\rm tr} (hgh^{-1}\kappa (hg^{-1}h^{-1}y,y))dy \,d(hZ)$$
where Melrose's $b$-integral has been used, $dy$ denotes the $b$-density associated to the $b$-metric $\mathbf{h}$ and where we recall
that  the cut-off function $c_0$ on $Y_0$ is extended constantly along
the cylinder to define $c$. We proved in \cite{PPST} that
\begin{itemize}
\item 
the pair 
$(\tau^{Y,r}_g,\sigma^{\partial Y}_g)$ defines  a relative 0-cocycle for ${}^b \mathcal{L}^c_G (Y)\xrightarrow{I} {}^b \mathcal{L}^c_{G,\RR} ({\rm cyl}(\partial Y))$;
\item 
the pair 
$(\tau^{Y,r}_g,\sigma^{\partial Y}_g)$ extends continuously to  a relative 0-cocycle for 
${}^b \mathcal{L}^\infty_{G,s} (Y)\xrightarrow{I} {}^b \mathcal{L}^\infty_{G,s,\RR} ({\rm cyl}(\partial Y));$
\item the following formula holds:
\begin{equation}\label{basic-formula-3} 
 \langle\tau^Y_g,\Ind_\infty (D)\rangle=\langle (\tau^{Y,r}_g,\sigma^{\partial Y}_g), \Ind_\infty (D,D_\partial )\rangle\,.
 \end{equation}
 \end{itemize}

  \section{The heat kernel}\label{section:heat}


In this section we wish to extend our study  of the large and short time behaviour 
of the heat kernel in \cite{PPST} to the case of  {\it perturbed} Dirac operators, first
 on cocompact $G$-proper manifolds without boundary
and  then on cocompact $G$-proper manifolds with boundary. As we shall see, properties of the heat kernel
of a perturbed Dirac operator that are analogue of well-known results on compact manifolds require a rather delicate analysis in order to be established in the $G$-proper case. In this section we 
state the main results; the (long) proofs will be given in a separate section at the end of the paper, see Section \ref{sect:proofs-heat}.

In this and the next section, we will freely use properties of the Bochner integrals with values in Fr\'echet algebras, c.f.  \cite{Thomas, Wiciak}.

\subsection{The heat kernel of a perturbed Dirac operator}
Instead of assuming that $D$ is $L^2$-invertible, as in \cite{PPST}, we assume now that there is a self-adjoint element $C\in \mathcal{L}_{G}^c(X)$ (c.f. Equation (\ref{algebra-slice})) such that $D+C$ is $L^2$-invertible. 
Thus there exists $a>0$ such that
\begin{equation}\label{spec-hypothesis}
{\rm Spec}_{L^2} (D+C)\cap (-2a,2a) = \emptyset\,.
\end{equation}

First of all, we need to make sense of  $\exp (-t(D+C)^2)$ as an element $\mathcal{L}^\infty_{G, s} (X)$. Observe that
$(D+C)^2 = D^2 + A$, with $A=C^2 + DC +  C D$.  
By its definition, $C^2$ is an element in $\mathcal{L}^c_{G}(X)$.  For the operators $DC$ and $C D$, we
proceed as follows. First we consider the operator
\[
D_{{\rm split}}:=D_{G, K} \hat{\otimes} 1+1\hat{\otimes} D_S.
\]
We know that elements of $\mathcal{L}^\infty_{G, s}(X)$ can be viewed as  $K\times K$-invariant Lafforgue's functions on $G$ with value in smoothing operators $\Psi^{-\infty}(S)$. Notice  that $D_{G,K}$ is an elliptic first order $G$-invariant differential operator on $G/K$, and $D_S$ is an elliptic operator on the slice $S$. By the very definition of $\mathcal{L}^c_G(X)$, $D_{G, K}C$ and $C D_{G, K}$ are in $\mathcal{L}^c_{G}(X)$.  And it follows from the definition of $\Psi^{-\infty}(S)$ 
that $D_S C$ and $C D_S$ are both in $\mathcal{L}^c_{G}(X)$. In summary, both $D_{{\rm split}} C$ and $C D_{{\rm split}}$ are in $\mathcal{L}^c_{G}(X)$.
On the other hand we know that
$D-D_{{\rm split}}=R$, with $R$  a 0th-order $G$-equivariant differential operator. Using the arguments in
 Lemma 4.17 in
\cite{PPST} we see that $R C$ and $C R$ are both in $\mathcal{L}^c_{G}(X)$ and thus we conclude that
$DC$ and $C D$ are in $\mathcal{L}^c_{G}(X)$. Summarizing, $(D+C)^2=D^2+A$ with $A\in \mathcal{L}^c_{G}(X)$.

Following \cite[Ch. 9, Appendix]{BGV} we consider the Volterra series
\begin{align*}
Q_t &:=\sum_{k=0}^\infty (-t)^k\int_{\Delta_k} e^{-\sigma_0 t D^2} A  e^{-\sigma_1 t D^2} \cdots A e^{-\sigma_k t D^2} d\sigma\\
&= e^{- t D^2} + \sum_{k=1}^\infty (-t)^k\int_{\Delta_k} e^{-\sigma_0 t D^2} A  e^{-\sigma_1 t D^2} \cdots A e^{-\sigma_k t D^2} d\sigma.
\end{align*}
Let $p$ be any seminorm on $ \mathcal{L}^\infty_{G, s} (N)$. We know that $e^{- t D^2} A\in \mathcal{L}^\infty_{G, s} (N)$; clearly we have  $p (e^{- t D^2} A)\leq C(p) p(A)$ so that for any $k\geq 1$ we 
have
$$ p \left( \sum_{k=1}^\infty (-t)^k\int_{\Delta_k} e^{-\sigma_0 t D^2} A  e^{-\sigma_1 t D^2} \cdots A e^{-\sigma_k t D^2} d\sigma
\right) \leq \frac{(C(p))^{k+1} p(A)^k}{k !}\,.$$
Thus the series converges in the Fr\'echet space  $ \mathcal{L}^\infty_{G, s} (N)$.

\subsection{Large time behaviour of the Connes-Moscovici projector $V(t(D+C))$ on $G$-proper manifolds without boundary.}
We now tackle the large time behaviour of the heat operator $\exp (-t(D+C)^2)$ and of the Connes-Moscovici projector $V(t(D+C))$. Let $\gamma$ be a suitable path in the complex plane, missing the spectrum of $(D+C)^2$. For example, we can take $\gamma$ to be the path given by the union of the two straight half-lines $$ {\rm Im}z=\pm m {\rm Re}z + b$$
in the half-plane ${\rm Re}z>0$; here $m={\rm tg} \phi>0$, $m$ small, and $b>0$ is smaller than the bottom of the spectrum of $(D+C)^2$. We set 
\begin{equation}\label{half-lines}
\ell^\pm (m,b):= \{z\in\CC\,|\, {\rm Im}z=\pm m {\rm Re}z + b\}.
\end{equation}

\medskip

\noindent
We shall be interested in real functions satisfying the following

\begin{assumption}\label{function-as-CM}
$f$ is an analytic function on $[0, \infty)$ that is of rapid decay, i.e. $\forall m,n$, there is an constant $M_{m,n}>0$ such that $$\operatorname{max}_{x\in [0, \infty)}|x^m \frac{\partial^n f(x)}{\partial x^n}|< M_{m,n}$$ for all $x\in [0, \infty)\,.$
\end{assumption}
%

\noindent
For a fixed function $f$ satisfying Assumption \ref{function-as-CM}, we choose $\gamma$ a suitable path in the domain of $f$ missing the spectrum of $-(D+C)^2$. Note that  both $e^{-z^2}$ and $e^{-\frac{z}{2}}\frac{1-e^{-z}}{z}$ satisfy 
Assumptions \ref{function-as-CM} and for these we can fix the $\gamma$ as above.

\noindent
The following Proposition, to be proved in  Section \ref{sect:proofs-heat}, is the main result of this subsection:

\begin{proposition}\label{prop:large-no-boundary-unperturbed}
If  $C\in \mathcal{L}^c_{G} (X)$, $D+C$ is $L^2$-invertible and $f$ satisfies Assumption , then
\begin{equation}\label{large-no-boundary-unperturbed}
 f (-t(D+C)^2)\rightarrow 0 \quad\text{in}\quad   \mathcal{L}^\infty_{G, s} (X)\quad\text{as}\quad t\to +\infty.
 \end{equation}
In particular, the Connes-Moscovici projector $V(t(D+C))-e_1$ converges to 0 in $M_{2\times 2}\Big(\mathcal{L}^\infty_{G, s} (X)\Big)$ as $t\to +\infty$.
\end{proposition}

\subsection{Small time behaviour of the heat kernel}
In this subsection, we state results on the small time behavior of the heat kernel for perturbed Dirac operators. Following the previous section, we consider a self-adjoint element $C\in \mathcal{L}_G^c(X)$ such that $D+C$ is $L^2$-invertible. 

In Lemma \ref{L2-behaviour-heat}, we showed that for any $K\in \mathcal{L}_G^c(X)$, the operator norm of $\exp(-\delta (D+C)^2)K-K$ converges to zero as $\delta$ goes to $0$. 
 In this section, we strengthen this result to the following one, to be proved in 
 Section \ref{sect:proofs-heat}. 

\begin{proposition}\label{prop:small-time-heat}
For any $K\in \mathcal{L}_G^c(X)$ we have that $K\exp(-\delta D^2)$ converges to $K$ in $\mathcal{L}_{G, s}^\infty(X)$ as $\delta\to 0$. Consequently, it directly follows from the Volterra series expansion of $\exp(-\delta(D+C)^2)$ that for any $C\in \mathcal{L}_G^c(X)$, 
\[
K\exp(-\delta(D+C)^2)\to K,\ \text{in}\ \mathcal{L}_{G, s}^\infty(X) \text{ as } \delta\downarrow 0.
\]
\end{proposition}

\smallskip
\noindent
The proof will be given in Section \ref{sect:proofs-heat}. 

\subsection{Large time behaviour on manifolds with boundary}
We now tackle the case of invertible perturbed  Dirac operators on manifolds with boundary (the {\it unperturbed}
 invertible case is treated in \cite{PPST}). 
We consider  $(Y_0,\mathbf{h}_0)$, a cocompact $G$ proper manifold with boundary $X$ and denote 
by  $(Y,\mathbf{h})$  the associated $b$-manifold. We denote by $Z_0$ a slice for $Y_0$
and by $Z$ the associated $b$-manifold.
Let $D$ be a Dirac operator on $Y$. 
Let $C\in {}^b \mathcal{L}^c_{G} (Y)$ and assume that $C$ is self-adjoint as a bounded operator on $L^2(Y)$. 
We assume that $D+C$ is $L^2$-invertible.
We also make the natural assumption that the indicial family $I(C,\lambda)$ is rapidly decreasing as a function with values in $\mathcal{L}^c_{G} (\partial Y)$
as $| {\rm Re} (\lambda)|\to +\infty$ whenever $|{\rm Im}(\lambda)|$ ranges in a finite interval.
In fact, we shall apply what follows to 
a particular situation that we describe now and for simplicity we only treat this special situation, given that is enough
for our geometric applications.
Recall, from \cite[Sec. 5.2]{PPST}, that we can construct a section $$\varphi: {}^b \mathcal{L}^\infty_{G, s, \RR} (\cyl (\partial Y)) \to {}^b \mathcal{L}^\infty_{G,s} (Y)$$  of the indicial homomorphism 
$$I : {}^b \mathcal{L}^\infty_{G,s} (Y)\to {}^b \mathcal{L}^\infty_{G, s, \RR} (\cyl (\partial Y));$$ $\varphi $ is defined by using a suitable cut-off function equal to 1 
on the boundary. The simple situation we will be dealing with, is when $C=\varphi (C_{\cyl})$, with
$C_{\cyl}$ obtained by inverse Mellin transform of $\alpha (\lambda)C_{\partial}$,
 with
$C_{\partial}\in \mathcal{L}^c_{G} (\partial Y)$ and 
$\alpha(\lambda)$ an entire function which is rapidly decreasing in ${\rm Re}\lambda$
on any horizontal strip $|{\rm Im} \lambda|<N$. From now on, we put ourselves in this particular situation.

We thus consider the perturbed operator $D+C$. We want first of all to make sense
of the heat kernel $\exp (-(D+C)^2)$ as an element in ${}^b \mathcal{L}^\infty_{G, s} (Y)$. 
Arguing as we have done  in the closed case 
we have that $(D+C)^2=D^2 + A$ with $A\in {}^b \mathcal{L}^\infty_{G, s} (Y)$.
Notice that $A$ is {\it not} a residual operator.
We  consider again the Volterra series
\begin{align*}
Q_t &:=\sum_{k=0}^\infty (-t)^k\int_{\Delta_k} e^{-\sigma_0 t D^2} A  e^{-\sigma_1 t D^2} \cdots A e^{-\sigma_k t D^2} d\sigma\\
&= e^{- t D^2} + \sum_{k=1}^\infty (-t)^k\int_{\Delta_k} e^{-\sigma_0 t D^2} A  e^{-\sigma_1 t D^2} \cdots A e^{-\sigma_k t D^2} d\sigma.
\end{align*}
Let $p$ be any seminorm on $ {}^b \mathcal{L}^\infty_{G, s} (Y)$. We know that 
$e^{- t D^2}\in  {}^b \mathcal{L}^\infty_{G, s} (Y)$, see \cite{PP2}; thus 
$e^{- t D^2} A\in {}^b \mathcal{L}^\infty_{G, s} (Y)$; since  $p (e^{- t D^2} A)\leq C(p) p(A)$ we 
have
$$ p \left( \sum_{k=1}^\infty (-t)^k\int_{\Delta_k} e^{-\sigma_0 t D^2} A  e^{-\sigma_1 t D^2} \cdots A e^{-\sigma_k t D^2} d\sigma
\right) \leq \frac{(C(p))^{k+1} p(A)^k}{k !}$$
and thus $\exp (-(D+C)^2)\in {}^b \mathcal{L}^\infty_{G, s} (Y)$. 

Regarding the large time behaviour of the heat kernel under the assumption that there exists $a>0$  
such that 
\begin{equation}\label{specrtum-perturbed}
{\rm Spec}_{L^2} (D+C)\cap [-2a,2a]=\emptyset,
\end{equation}
 the following proposition holds:

\begin{proposition}\label{prop:large-yes-boundary-perturbed}
Let $C\in {}^b\mathcal{L}^{c}_{G} (Y)$ be self-adjoint,
$C=\varphi (C_{\cyl})$, with
$C_{\cyl}$ obtained by inverse Mellin transform of $\alpha (\lambda)C_{\partial}$,
 with
$C_{\partial}\in \mathcal{L}^c_{G} (\partial Y)$ and 
$\alpha(\lambda)$ an entire function which is rapidly decreasing in ${\rm Re}\lambda$
on any horizontal strip $|{\rm Im} \lambda|<N$. 
 Let $D+C$ be $L^2$-invertible. Then
\begin{equation}\label{large-yes-boundary-perturbed}
 \exp (-t(D+C)^2)\rightarrow 0 \quad\text{in}\quad   {}^b \mathcal{L}^\infty_{G, s} (Y)\quad\text{as}\quad t\to +\infty.
 \end{equation}
 \end{proposition}

\noindent
Proceeding as in the closed case, considering thus a function $f$ 
satisfying Assumption \ref{function-as-CM}, we can establish more generally the following. 

\begin{proposition}\label{prop:large-yes-boundary-perturbed-bis}
Let $C\in {}^b\mathcal{L}^{c}_{G} (Y)$ be self-adjoint
and 
of the type  explained in the statement of the previous 
Proposition. If  $D+C$ is $L^2$-invertible then
for the Connes-Moscovici projector it holds that
 \begin{equation}\label{large-yes-boundary-perturbed-bis}
 V(t(D+C))-e_1\to 0\quad\mbox{in}~M_{2\times 2}\Big({}^b\mathcal{L}^\infty_{G, s} (Y)\Big)\quad \mbox{as} ~t\to +\infty.
 \end{equation} 
\end{proposition}


\noindent
Proofs will be given in Section \ref{sect:proofs-heat}

\section{Delocalized eta invariants associated to smoothing perturbations}\label{sect:modified}
Let $(X,\mathbf{h})$ be a cocompact $G$-proper manifold of odd dimension, without boundary. As we have assumed $G/K$ is of even dimension in Sec. \ref{sect:preliminaries}, the slice $S$ is of odd dimension.
In this section we shall prove  results on delocalized eta invariants associated to {\it perturbed}
Dirac operators. More precisely we shall consider
 a perturbed Dirac operator $D+C$ with $C$ an element in $\mathcal{L}^\infty_{G, s} (X,E)$ making the operator $D+C$ invertible.

 \begin{proposition}\label{prop:eta-conv-perturbed}
 Let $(X,\mathbf{h})$ be an odd dimensional  cocompact $G$-proper manifold without boundary and 
 $\mathbf{h}$ a $G$-invariant Riemannian metric. Let $D$ be a Dirac-type operator associated to a Clifford connection. 
If $C\in \mathcal{L}^c_G (X)$ 
  is such that $D+C$ is $L^2$-invertible then 
$$\eta_g (D+C):=\frac{1}{\sqrt{\pi}} \int_0^\infty \tau^{X}_{g} ((D+C) \exp (-t(D+C)^2) \frac{dt}{\sqrt{t}}$$
is well defined.
\end{proposition}

\begin{proof}
We made  sense of  $\exp (-t(D+C)^2)$ as an element $\mathcal{L}^\infty_{G, s} (X)$
and proved that  $\exp (-t(D+C)^2))$
is converging to $0$ in  $ \mathcal{L}^\infty_{G, s} (X)$ as $t\to +\infty$. A similar estimates show that
$(D+C)\exp (-t(D+C)^2)$ is of $o(t)$ as $t\to \infty$. (c.f. Proposition \ref{prop:large-no-boundary-unperturbed}).
This implies that 
$$\frac{1}{\sqrt{\pi}} \int_1^\infty \tau^{X}_{g} ((D+C) \exp (-t(D+C)^2) \frac{dt}{\sqrt{t}}$$
converges at $t=+\infty$. 

On the other hand, by construction, $\exp (-t (D^2+A))- \exp (-t D^2)$ is $O(t)$ in $ \mathcal{L}^\infty_{G, s} (X)$. Since  $C \exp (-t D^2)\to C$ as $t\to 0$ in $\mathcal{L}^\infty_{G, s}(X)$ by Proposition \ref{prop:small-time-heat}, we conclude that
$$\frac{1}{\sqrt{\pi}} \int_0^1 \tau^{X}_{g} ((D+C) \exp (-t(D+C)^2) \frac{dt}{\sqrt{t}}$$
is also convergent as the integral 
\[
\frac{1}{\sqrt{\pi}} \int_0^1 \tau^{X}_{g} (D \exp (-tD^2) \frac{dt}{\sqrt{t}}
\]
converges (\cite[Theorem 5.1]{PPST}). The proof of the proposition is complete.
\end{proof} 
 
\section{The  delocalized index theorem corresponding to smoothing perturbations}\label{sect:aps-perturbed}
In this section we shall prove an index formula for a Dirac operator that admits a smoothing perturbation
 of the boundary operator that makes it invertible. 

As in the previous section, we do  not assume that $D_\partial$ is $L^2$-invertible; instead 
we assume the existence of an operator
 $C_\partial \in \mathcal{L}^c_{G} (X)\equiv \mathcal{L}^c_{G} (\partial Y)$ such that  $D_{\partial} +C_\partial$ is $L^2$-invertible. 
 We know  then that  
$$\eta_g (D_{\partial}+C_\partial):=\frac{1}{\sqrt{\pi}} \int_0^\infty \tau^{X}_{g} ((D_{\partial}+C_\partial) \exp (-t(D_{\partial}+C_\partial)^2) \frac{dt}{\sqrt{t}}$$
is well defined. Our goal in this section is to show that this delocalized eta invariant is the boundary correction term in a delocalized APS index theorem.
Geometric applications will be given in the next section.

\medskip

\medskip
We show first of all that we can lift the perturbation $C_\partial\in \mathcal{L}^c_G (X)$ to a perturbation $C_Y\in {}^b \mathcal{L}^c_G (Y)$.
We follow closely \cite{Mel-P1}.
Recall the short exact sequence (\ref{b-short-c}) of algebras
\[
0\to  \mathcal{L}^c_G (Y)\to 
{}^b \mathcal{L}^c_G (Y)\xrightarrow{I} {}^b \mathcal{L}^c_{G,\RR} ({\rm cyl}(\partial Y))\to 0.
\]
Using a suitable cut-off function near the boundary, we define as before a  section $\varphi:  {}^b \mathcal{L}^c_{G,\RR} ({\rm cyl}(\partial Y))\rightarrow {}^b \mathcal{L}^c_{G} (Y)$ of the indicial homomorphism.  
 Thus we can define an operator on $Y$ by defining a translation invariant  operator on ${\rm cyl}(\partial Y)$ and then applying
$\varphi$ to it. Starting with $C\in \mathcal{L}^c_G (X)$, with $X=\partial Y_0\equiv \partial Y$,
we thus want to define a translation invariant  operator on ${\rm cyl}(\partial Y)$;  by inverse Mellin transform, we can  equivalently define 
the indicial family of such an operator. 

Let $\rho\in C^\infty_c (\RR)$  be non negative,  even and have integral 1, so that
$\rho_\epsilon (t):= \epsilon^{-1} \rho (t/\epsilon)$ approximates $\delta_0$ as $\epsilon\downarrow 0$.
The Fourier-Laplace transform $$\widehat{\rho}_\epsilon (z)=\int_\RR e^{-it z} \rho_\epsilon (t) dt, \quad z\in\CC,$$ 
is therefore entire,  even,  real for real $z$
and has value at $0$ equal to 1.  Moreover, it is rapidly decreasing in ${\rm Re}z$
on any horizontal strip $|{\rm Im z}|<N$ and  as $\epsilon\downarrow 0$ it approximates the constant function 1 
uniformly on any compact subset of $\CC$.

We consider the holomorphic family of operators on $X=\partial Y$ given by 
\begin{equation}\label{essepsilon}
 S(\epsilon,z):=\widehat{\rho}_\epsilon (z)C_\partial
 \end{equation}
and consider the inverse Mellin transform for $z=\lambda\in\RR$. By applying suitable
identifications near the boundary, this defines an element $C^+_{\rm cyl} (\epsilon)$
in 
${}^b \mathcal{L}^\infty_{G,s_0, \RR} ({\rm cyl}(\partial Y))$ where we do not write the bundles
as usual; see
\cite[Equation (8.5)]{Mel-P1}. By applying the section
$\varphi:  {}^b \mathcal{L}^\infty_{G,s_0,\RR} ({\rm cyl}(\partial Y))\rightarrow {}^b \mathcal{L}^\infty_{G, s} (Y)$
of the indicial homomorphism 
to this element we finally define 
$$C^{+}_Y (\epsilon)\,,\quad C^{-}_Y (\epsilon):= (C^{+}_Y (\epsilon))^*\,.$$
We have now defined an operator 

\[
D + C_Y (\epsilon):= \left( \begin{array}{cc} 0 & D^- + C^{-}_Y (\epsilon)\\
D^+ + C^{+}_Y (\epsilon) & 0
\end{array} \right).
\]
The indicial family $I(D^+ + C^{+}_Y (\epsilon), \lambda)$, $\lambda\in\RR$ is $L^2$-invertible
for $\epsilon$ small;
indeed, up to identifications near the boundary,
$$I(D^+ + C^{+}_Y (\epsilon), \lambda)=i\lambda + S(\epsilon,\lambda) + D_\partial$$
which is invertible by construction at $\lambda=0$ and it is invertible if $\lambda\in\RR\setminus \{0\}$ given that 
$D_\partial$ is self-adjoint and $\epsilon$ is small, see \cite[Lemma 8]{Mel-P1}. 
From now on we fix $\epsilon$ small as to ensure that $I(D^+ + C^{+}_Y (\epsilon), \lambda)=i\lambda + S(\epsilon,\lambda) + D_\partial$
is $L^2$-invertible for all $\lambda\in\RR$. We then set
\begin{equation}\label{R}
R(\lambda):= S(\epsilon,\lambda):=\widehat{\rho}_{\epsilon} (\lambda)C_\partial
\end{equation}
and remark that 
\begin{equation}\label{at-zero}
R(0)= C_\partial\quad\text{and}\quad R(\lambda)=R(-\lambda)
\end{equation}
We write $C_Y$ instead of $C_Y (\epsilon)$.\\
Consequently we have a well-defined index class
$$\Ind (D+C_Y)\in K_0 (\mathcal{L}^\infty_{G, s} (Y))=K_0 (C^* (Y)^G)$$
that can be defined \`a la Connes-Skandalis through a true parametrix or, equivalently, through
the heat kernel and  the {\it improved} Connes-Moscovici projector $V^b (D+C_Y)$. This index class
does not depend on $\epsilon$, for $\epsilon$ small, or on the particular choice of the function $\rho\in C^\infty_c (\RR)$. It is also independent of the choices involved in
the definition of the section\footnote{Two different choices 
differ by a residual operator and we know that the index class does not change by the addition of a residual operator.} $\varphi$
and for this reason the following notation is well-posed:

\medskip
\noindent
{\bf Notation.} Following \cite{PiazzaSchick_BCrho} we shall also denote this index class by
$\Ind (D,C_\partial)$.

\begin{remark}\label{comparison-wahl}
We have lifted the perturbation $C_\partial$ to a global perturbation $C_Y$ on $Y$ and we have 
insisted that this perturbation $C_Y$ were a $b$-pseudodifferential operator. This is necessary in order to
use the $b$-calculus and  prove the APS index formula.
However, as far as index classes are concerned, we can simply extend the operator $C_\partial$ using a cut-off function;
the resulting operator is {\em not} in the $b$-calculus but   it can be proved nevertheless that we still get an index class, called the {\it cylindrical} index class in \cite{LLP}.
It can be established that the $b$-index class and  the {\it cylindrical} index class are in fact equal; the detailed argument given in 
the proof of Theorem 10.1 in \cite{LLP} can be easily adapted to the present case.
\end{remark}

\medskip
We also have the Connes-Moscovici projectors $V(D+C_Y)$ and $V(D_{{\rm cyl}} + C_{\rm cyl})$
that define together a relative class
$$ [V(D+C_Y),e_1,r_t]\in  
K_0({}^b \mathcal{L}^\infty_{G, s} (Y),{}^b \mathcal{L}^\infty_{G,s,\RR} ({\rm cyl}(\partial Y)))$$
with $e_1:=\begin{pmatrix} 0&0\\0&1 \end{pmatrix}$ and $r_t $, $t\in [1,+\infty]$, defined as follows

\begin{equation}\label{pre-wassermann-triple-rt}
r_t:= \begin{cases} V(t (D_{\cyl}+C_{{\rm cyl}})),
\;\;\quad\text{if}
\;\;\;t\in [1,+\infty),\\
e_1, \;\;\;\;\;\;\;\;\;\;\;\;\;\,\text{ if }
\;\;t=\infty.
 \end{cases}
\end{equation}
Here we have used Proposition \ref{prop:large-yes-boundary-perturbed} to see that this is well-defined.\\
Proceeding as in \cite{PPST} we then have that

\begin{equation}\label{abs=rel-perturbed}
\langle [\tau^Y_g], \Ind (D+C_Y) \rangle=  \langle [\tau^{Y,r}_g,\sigma^{\partial Y}_g], [V(D+C_Y ),e_1,r_t]\rangle\end{equation}

We shall use this equality in order to prove the following theorem.

\begin{theorem}
\label{theorem-aps-ind-perturbed}
Let $C_\partial\in  \mathcal{L}^c_{G} (\partial Y)$ be  such that $D_\partial + C_\partial$ is $L^2$-invertible.
Then there exists 
a well defined 
index class $\Ind (D,C_\partial)\in K_0 (\mathcal{L}^\infty_{G, s} (Y))=K_0 (C^* (Y_0\subset Y, E)^G)$. Moreover,
for the delocalized index $\langle [\tau^Y_g], \Ind (D,C_\partial) \rangle$ the following formula holds
\begin{equation}\label{main-perturbed}
\langle [\tau^Y_g], \Ind (D,C_\partial) \rangle= \int_{(Y_0)^g} c^g {\rm AS}_g (D_0) - \frac{1}{2} \eta_g (D_{\partial}+C_\partial)
\end{equation}
with 
$$\eta_g (D_{\partial}+C_\partial)=\frac{1}{\sqrt{\pi}} \int_0^\infty \tau^{X}_{g} ((D_{\partial}+C_\partial) \exp (-t(D_{\partial}+C_\partial)^2) \frac{dt}{\sqrt{t}}$$
\end{theorem}

\begin{proof}
We have already discussed the first part of the statement. 
We thus concentrate on the proof of \eqref{main-perturbed}.
Let us consider the right hand-side. 

\smallskip
\noindent
{\bf Notation.}
To simplify the notation we set
$$A^\pm  := D^\pm + C^\pm_Y \,\quad 
A^\pm_{{\rm cyl}}   := D^\pm_{{\rm cyl}} + C^\pm_{{\rm cyl}}\,,\quad
B (\lambda) = D_\partial + R(\lambda)$$
where we recall that $R(\lambda)=\widehat{\rho}_{\epsilon} (\lambda)C_\partial$, so that  $R(0)=C_\partial$ and $R(\lambda)=R(-\lambda)$.

\smallskip
\noindent
By definition of 
relative pairing we have:
\begin{equation}\label{relative-pairing}
\langle [\tau^{Y,r}_g,\sigma^{\partial Y}_g], [V(D+C_Y ), e_1, r_t ]\rangle=
 \tau^{Y,r}_g (e^{-A^- A^+})- \tau^{Y,r}_g (e^{-A^+ A^-})
+ \int_1^\infty \sigma^{\partial Y}_g ([\dot{r}_t  ,r_t ],r_t )dt\,.
\end{equation}
We now rescale the operator on $Y$ and consider $s(D+C_Y)$. This defines the same index class $\Ind (D,C_\partial)$
and so
$$\langle [\tau^Y_g], \Ind (D,C_\partial) \rangle= 
\tau^{Y,r}_g (e^{-s^2 A^- A^+})- \tau^{Y,r}_g (e^{-s^2 A^+ A^-})
+ \int_s^\infty \sigma^{\partial Y}_g ([\dot{r}_t ,r_t ],r_t)dt\,.$$
Observe that
$$s^2  A^- A^+ = s^2 D^- D^+  + s^2 F^{-,+} \quad\text{with}\quad F^{-,+}\in {}^b \mathcal{L}^\infty_{G, s} (Y)$$
and similarly for $s^2 A^+ A^-$.
By a classic Volterra series argument we obtain that
$$\lim_{s\downarrow 0} \tau^{Y,r}_g (e^{-s^2 A^- A^+})- \tau^{Y,r}_g (e^{-s^2 A^+ A^-})= \lim_{s\downarrow 0}  \tau^{Y,r}_g (e^{-s^2 D^- D^+})- \tau^{Y,r}_g (e^{-s^2 D^+ D^-})$$
and by \cite[Proposition 4.47]{PPST} we thus have 
$$\lim_{s\downarrow 0} \tau^{Y,r}_g (e^{-s^2 A^- A^+})- \tau^{Y,r}_g (e^{-s^2 A^+ A^-})= \int_{(Y_0)^g} c^g {\rm AS}_g (D_0)\,.$$
Consequently, from the $s$-independence of the index class we obtain the existence of the
integral
$$  \int_0^\infty \sigma^{\partial Y}_g ([\dot{r}_t  ,r_t ],r_t )dt$$
and the formula 
$$\langle [\tau^Y_g], \Ind (D,C_\partial) \rangle= 
 \int_{(Y_0)^g} c^g {\rm AS}_g (D_0)
+ \int_0^\infty \sigma^{\partial Y}_g ([\dot{r}_t ,r_t ],r_t )dt\,.$$
We consider 
 $ \int_0^\infty \sigma^{\partial Y}_g ([\dot{r}_t  ,r_t ],r_t )dt$
where we recall that 
$$\int_0^\infty\sigma^{\partial Y}_g ([\dot{r}_t ,r_t ],r_t )dt=  
\frac{i}{2\pi} \int_\RR  \tau^{\partial Y}_g (\partial_\lambda (I([\dot{r}_t ,r_t ],\lambda))\circ I(r_t,\lambda))d\lambda.$$
We can integrate by parts in $\lambda$ on the right hand side, given that the two terms are rapidly decreasing
in $\lambda$. Thus  we consider 
\begin{equation}\label{reduction-to-eta}
 \frac{i}{2\pi}\int_0^\infty \left( 
 \int_\RR  \tau^{\partial Y}_g ( (I([\dot{r}_t ,r_t ],\lambda)\circ \partial_\lambda(I(r_t ,\lambda))d\lambda
 \right) dt.
\end{equation}

\begin{lemma}\label{lemma-key}
 The integral \eqref{reduction-to-eta} is equal to 
\[
-\frac{1}{2\sqrt{\pi}} \int_0^\infty \tau^{\partial Y}_g (D_\partial + C_{\partial}) \exp (-t(D_{\partial}+C_\partial)^2 )\frac{dt}{\sqrt{t}}.
\]
\end{lemma}

\smallskip
\noindent
\begin{proof}
We are  considering the integral (\ref{reduction-to-eta})
\begin{equation*}
  \int_0^\infty \left( \int_\RR  \tau^{\partial Y}_g ( (I([\dot{r}_t ,r_t ],\lambda)\circ \partial_\lambda(I(r_t ,\lambda))d\lambda
 \right)dt.
\end{equation*}
We write down the indicial family of $r_t $. By definition
$$r_t = \left( \begin{array}{cc} e^{-t^2 A_{\cyl}^- A_{\cyl}^+} & e^{-\frac{t^2}{2}A_{\cyl}^- A_{\cyl}^+}
\left( \frac{I- e^{-t^2A_{\cyl}^- A_{\cyl}^+}}{t^2 A_{\cyl}^- A_{\cyl}^+} \right) t A_{\cyl}^-\\
e^{-\frac{t^2}{2}A_{\cyl}^+  A_{\cyl}^-} t A^+& I- e^{- t^2 A_{\cyl}^+ A_{\cyl}^-}
\end{array} \right)$$
where we recall that $A^\pm_{{\rm cyl}}  := D^\pm_{{\rm cyl}} + C^\pm_{{\rm cyl}}$.
Recall, also, that we are using the notation $B (\lambda) =D_\partial + R(\lambda)$
where, once again,  we do not write the $\epsilon$ in $R$ because we have fixed it.
The indicial family of $r_t$ is given by 
$$I(r_t ,\lambda)= \left( \begin{array}{cc} e^{-t^2 (\lambda^2 +B (\lambda)^2)} & e^{-\frac{t^2}{2}(\lambda^2 + B (\lambda)^2)}
\left( \frac{I- e^{-t^2 (\lambda^2 + B (\lambda)^2)}}{t^2 ( \lambda^2 + B (\lambda)^2)}\right) t (-i\lambda+B (\lambda))\\
e^{-\frac{t^2}{2}(\lambda^2 +B (\lambda)^2)} t (i\lambda +B (\lambda))& I- e^{- t^2 (\lambda^2 +B (\lambda)^2)}
\end{array} \right).$$
We set 
$$r(t,\lambda):= I(r_t ,\lambda)\,.$$
We now consider the integrand   in \eqref{reduction-to-eta}:

\smallskip
\noindent
{\bf Claim.} We have the equality
\begin{equation}
\label{eq-fe}
\tau_g^{\partial Y} \left(  [\partial_t r(t,\lambda), r(t,\lambda)]\circ \partial_\lambda r(t,\lambda)\right)=-2it \tau_g\left(e^{-t^2 (\lambda^2 +B (\lambda)^2)}(B(\lambda)-\lambda\partial_\lambda B(\lambda))\right).
\end{equation}
\smallskip
\noindent
{\bf Proof of the claim.}\\
Computing the $t$-derivative of $r(t,\lambda)$ and we see that the $2\times 2$-matrix $Q(t,\lambda):=\left[\partial_t r(t,\lambda), r(t,\lambda)\right]$
has entries that consist of products of operators of the form $B(\lambda)\pm i\lambda$, $(\lambda^2+B(\lambda)^2)^{\pm 1}$ and heat-kernels
$e^{-\frac{t^2}{2}(\lambda^2 +B (\lambda)^2)}$. In particular, the operators all commute with each other since they are obtained by applying 
functional calculus to the operator $B(\lambda)\pm i\lambda$. 

The $\lambda$-derivative of $r(\lambda,t)$ is more involved and introduces the operator $\partial_\lambda B(\lambda)$ which does not necessarily commute with other terms: to compute the
derivative of the heat kernel we use Duhamel's formula
\[
\partial_\lambda r_{11}(t,\lambda)=\partial_\lambda e^{-t^2 (\lambda^2 +B (\lambda)^2)}=\int_0^te^{(t-s)^2 (\lambda^2 +B (\lambda)^2)}
\partial_\lambda(\lambda^2 +B (\lambda)^2)e^{s^2 (\lambda^2 +B (\lambda)^2)}ds.
\]
Also to compute the $\lambda$-derivative of $( \lambda^2 + B (\lambda)^2)^{-1}$ we use the identity
\[
( \lambda^2 + B (\lambda)^2)^{-1}=\int_0^\infty e^{-s( \lambda^2 + B (\lambda)^2)} ds,
\]
which holds true because $\lambda^2 + B (\lambda)^2$ is invertible, and subsequently use Duhamel's formula once again. If we now remark that in the computation the only source of noncommutativity is $\partial_\lambda B(\lambda)$ and since there is 
at most one of them in each term, we can use cyclicity of the trace $\tau_g$ to obtain the identities
\begin{align*}
\tau_g\left(\ldots\partial_\lambda e^{-t^2 (\lambda^2 +B (\lambda)^2)}\ldots\right)&=-2t^2\tau_g\left(\ldots e^{-t^2 (\lambda^2 +B (\lambda)^2)}\ldots(\lambda+B(\lambda)\partial_\lambda B(\lambda))\right),\\
\tau_g\left(\ldots\partial_\lambda (\lambda^2+B(\lambda^2))^k\ldots\right)&=2k\tau_g\left(\ldots (\lambda^2+B(\lambda^2))^{k-1}\ldots(\lambda+B(\lambda)\partial_\lambda B(\lambda))\right),\quad k\in\mathbb{Z},
\end{align*}
where the $\ldots$ denotes terms commuting with $\partial_\lambda B(\lambda)$. We therefore see
that in the computation we can always put the terms with $\partial_\lambda B(\lambda)$ on the right 
hand side in the trace. With this convention, the computation reduces to a straightforward, but tedious, 
computation in which we can assume that all operators commute with each other. This can even be done using e.g. {\em Mathematica} and the final outcome is:
\[
\tau_g\left(\left[\partial_t r(t,\lambda),r(t,\lambda)\right]\partial_\lambda(t,\lambda)\right)=-2it \tau_g\left(e^{-t^2 (\lambda^2 +B (\lambda)^2)}(B(\lambda)-\lambda\partial_\lambda B(\lambda))\right).
\]
This finishes the proof of the claim. 

To finish the proof of the Theorem, we adapt  an argument from \cite[\S 3]{loya}, in turn inspired by 
\cite{Mel-P1}. We have the following analogue of \cite[Lemma 3.10]{loya}:
\begin{lemma}
\label{lemma-lail}
Let $A(r,\lambda)$ be a continuous and rapidly decreasing function on $[0,1]\times\RR$  with values $\mathcal{L}^\infty_{G,s}(X)$. The integrals
\[
\int_\RR\tau_g^X\left(A(r,\lambda)e^{-t^2(\lambda^2+B(r\lambda)^2)}\right)d\lambda,
\qquad 
\int_\RR\tau_g^X\left(B(r\lambda)e^{-t^2(\lambda^2+B(r\lambda)^2)}\right)d\lambda,
\]
converge absolutely and are $\mathcal{O}(t^{-1})$ as $t\downarrow 0$ and vanish exponentially as $t\to \infty$, both uniformly in $r\in [0,1]$.
\end{lemma}
\begin{proof}
The arguments around \eqref{structure}, together with the Volterra expansion for the perturbed heat kernel allow us to apply the same arguments as in \cite{loya} 
once we take into account that $\tau_g^X$ is a continuous trace on $\mathcal{L}^\infty_{G,s}(X)$.
\end{proof}
 Next, we define
\begin{align*}
\zeta(r):=\frac{1}{\pi}\int_0^\infty\zeta(r,t) dt=\frac{1}{\pi}\int_0^\infty\int_{\mathbb{R}}\tau_g\left(e^{-t^2(\lambda^2+B(r\lambda)^2)}t(B(r\lambda)-\lambda\partial_\lambda B(r\lambda))\right)d\lambda dt.
\end{align*}
Remark that $\lambda\partial_\lambda B(r\lambda)$ is rapidly decreasing function of $(r,\lambda)$ with values in 
$\mathcal{L}^\infty_{G,s}(X)$, so the integrals above converge by Lemma \ref{lemma-lail}.
We can write the $t$-integrand as
\[
\zeta(r,t)=\int_{\mathbb{R}}\tau_g\left( L(\tilde{B})e^{-t^2\lambda^2-\tilde{B}^2}\right)d\lambda,
\]
where $\tilde{B}(r,t,\lambda):=tB(r\lambda)$ and $L:=t\partial_t-\lambda\partial_\lambda$. We then compute the $r$-derivative of $\zeta(r,t)$ as
\[
\frac{d\zeta(r,t)}{dr}=\underbrace{\int_{\mathbb{R}}\tau_g\left(L\left(\frac{d\tilde{B}}{dr}\right)e^{-t^2\lambda^2-\tilde{B}^2}\right)d\lambda}_{{\rm (I)}}+
\underbrace{\int_{\mathbb{R}}\tau_g\Bigg( L\left(\tilde{B}\right)\frac{d}{dr}e^{-t^2\lambda^2-\tilde{B}^2}\Bigg)d\lambda}_{\rm (II)}.
\]
In the above formula, we can move the differential inside by dominated convergence: Lemma \ref{lemma-lail} tells us that (I) converges absolutely. For (II) we use Eq. \eqref{eq-II} below together with the fact that $\frac{d\tilde{B}}{dr}\tilde{B}$ and $\tilde{B}\frac{d\tilde{B}}{dr}$ are rapidly decreasing in $\lambda$ in $\mathcal{L}^\infty_{G,s}(X)$ to obtain absolute convergence by the same Lemma.

The first term (I) we work out using an integration by parts: 
\begin{align*}
{\rm (I)}&=\frac{d}{dt}\left[\int_{\mathbb{R}}\tau_g\left(t\frac{d\tilde{B}}{dr}e^{-t^2\lambda^2-\tilde{B}^2}\right)d\lambda\right]-\int_{\mathbb{R}}\tau_g\left(\frac{d\tilde{B}}{dr}Le^{-t^2\lambda^2-\tilde{B}^2}\right)d\lambda\\
&=\frac{d}{dt}\left[\int_{\mathbb{R}}\tau_g\left(t\frac{d\tilde{B}}{dr}e^{-t^2\lambda^2-\tilde{B}^2}\right)d\lambda\right]
+\int_{\mathbb{R}}e^{-t^2\lambda^2}\int_0^1\tau_g\left(\frac{d\tilde{B}}{dr}e^{-(t-s)^2B^2}\left(L(\tilde{B}) \tilde{B}+\tilde{B}L(\tilde{B})\right)e^{s^2B^2}\right)dsd\lambda,
\end{align*}
by Duhamel's formula combined with the fact that $L(e^{-t^2\lambda^2})=0$. We can also write out the second term (II) using Duhamel's formula:
\begin{equation}
\label{eq-II}
{\rm (II)}=-\int_{\mathbb{R}}e^{-t^2\lambda^2}\int_0^1\tau_g\left( L(\tilde{B})e^{-(t-s)^2B^2}\left(\frac{d\tilde{B}}{dr}\tilde{B}+\tilde{B}\frac{d\tilde{B}}{dr}\right)e^{s^2B^2}\right)dsd\lambda.
\end{equation}
After this, putting the two terms (I) and (II) back together, we find that the terms invoking Duhamel's formula combine to a commutator:
\begin{align*}
\int_{\mathbb{R}}e^{-t^2\lambda^2}\int_0^1\tau_g\left[\frac{d\tilde{B}}{dr},e^{-(t-s)^2B^2}\left(L(\tilde{B}) \tilde{B}+\tilde{B}L(\tilde{B})\right)e^{s^2B^2}\right]dsd\lambda=0,
\end{align*}
by the trace property of $\tau_g$. This results in the following formula for the derivative of $\zeta(r,t)$:
\[
\frac{d\zeta(r,t)}{dr}=\frac{d}{dt}\left[\int_{\mathbb{R}}\tau_g\left(t^2\frac{d B(r\lambda)}{dr}e^{-t^2(\lambda^2-B(r\lambda)^2)}\right)d\lambda\right].
\]
With this formula we see that
\begin{align*}
\frac{d\zeta(r)}{dr}&=\frac{1}\pi\int_0^\infty\frac{d\zeta(r,t)}{dr}dt\\
&=\frac{1}{\pi}\int_0^\infty\frac{d}{dt}\left[\int_{\mathbb{R}}\tau_g\left( t^2\frac{d B(r\lambda)}{dr}e^{-t^2(\lambda^2-B(r\lambda)^2)}\right)d\lambda\right]dt=0,
\end{align*}
where we have used Lemma \ref{lemma-lail} in the final identity to evaluate the boundary terms.
It follows that $\zeta(r)$ is constant and therefore $\zeta(0)=\zeta(1)$. But $\zeta(1)$ is exactly $(2\pi i)^{-1}$ times the right hand side of Eq. \eqref{eq-fe}.
On the other hand we find, after a Gaussian integration,
\[
\zeta(0)=\frac{1}{2\sqrt{\pi}}\int_0^\infty e^{-t(D_\partial+C_\partial)^2}(D_\partial+C_\partial)\frac{dt}{\sqrt{t}}.
\]
This finishes the proof of Lemma \ref{lemma-key}.
\end{proof}

\noindent
With the Lemma, the proof of Theorem \ref{theorem-aps-ind-perturbed} is now also complete.
\end{proof}

\section{Secondary invariants of the signature operator on $G$-proper manifolds}\label{sect:numeric}
In our previous work \cite{PPST} we have defined (higher) rho numbers associated to a $G$-invariant positive scalar curvature
metric on a $G$-proper cocompact spin manifold.
In this section we shall introduce  rho numbers associated 
to the signature operator. More precisely, 
using the signature operator and the Hilsum-Skandalis perturbation,  we shall define rho numbers associated  to $G$-equivariant homotopy equivalences and we shall study their properties. Next, 
under an invertibility assumption on the differential form laplacian in middle degree we shall introduce the delocalized 
eta invariant of the signature operator on a G-proper manifold without boundary and we shall prove that it enters into a APS index theorem. These results on the signature operator
are in fact the main motivation for developing the general theory we have illustrated  in the previous sections.
 

\subsection{Non-elliptic elements in $G$}
We begin by recalling a definition from 
\cite{PPST}.
Consider a cocompact proper $G$-manifold $(W,\mathbf{h})$, possibly with boundary, with $G$
a connected linear real reductive group. 
Let $D_W$ be a Dirac operator on $W$, defined by a unitary and Clifford connection.
Let $g\in G$ be semisimple. We shall say that $g$ is {\it geometrically-simple
on $W$} 
if 
$$\int_{W^g} c^g {\rm AS}_g (W,S)=0\, ,$$
with ${\rm AS}_g (W,S)$ the usual Atiyah-Segal integrand appearing as  the local term in a 0-degree delocalized APS-index theorem for $D_W$.

\noindent
As explained in \cite{PPST}, the following proposition provides many examples of geometrically-simple elements $g$ on an arbitrary
$G$-proper manifold $W$.

\begin{proposition}\label{prop:fixed point}If $g$ is non-elliptic, that is, does not conjugate to a compact element, then every element of the conjugacy class $\{hgh^{-1}|h\in G\}$ in $G$ does not have any fixed point on $W$.
\end{proposition}

%
\subsection{Stability results for rho numbers}
We consider a closed $G$-proper manifold $X$ {\bf without} boundary, $G$ a connected linear real reductive, $g\in G$ a semisimple element,
$D_X$ a $G$-equivariant Dirac operator. We assume the existence of a perturbation $B \in \mathcal{L}^c_G (X,E)$
with the property that there exists $a>0$ such that
$${\rm spec}_{L^2} (D_X+B)\cap [-2a,2a]=\emptyset\,.$$ 
We consider
$$\eta_g (D_X+B):=\frac{1}{\sqrt{\pi}} \int_0^\infty \tau^{X}_{g} ((D_{X}+B) \exp (-t(D_{X}+B)^2) \frac{dt}{\sqrt{t}}\;;$$
  we have proved  in Proposition \ref{prop:eta-conv-perturbed}  that 
this integral
is well defined.

\noindent
We shall need stability results regarding this number. 

\begin{proposition}\label{prop:stability}
Let $[0,1]\ni r \to B(r) \in \mathcal{L}^c_G (X,E)$ a $C^1$-map with the property that there exists $a>0$ such that
$${\rm spec}_{L^2} (D_X+B (r))\cap [-2a,2a]=\emptyset\quad\forall r\in[0,1]$$ 
We assume that $B(r)$ is constant for $r\in [0,1/4] \cup [3/4,1]$.
If $g$ is non-elliptic, then 
$$\eta_g (D_X+B (0))=\eta_g (D_X+B (1))\,.$$
\end{proposition}

\begin{proof}
Consider the cylinder $\RR_r\times X$ endowed with the lifted action and the product metric.
Alternatively, we compactify the cylinder to a compact $b$-cylinder $[0,1]_x\times X$ with the lifted action
and the product $b$-metric.
The family 
$\{D_X+B (r)\}_{r\in [0,1]}$ defines an operator  
$$P:= \partial_r + D_X+B (r)\quad \text{on} \quad \RR_r\times X,$$ where we extend the perturbations
to be constant for $r<0$ and $r>1$. Notice that $P$ is a perturbation of $D:= \partial_r + D_X$.
Similarly, the family 
$\{D_X+B (r)\}_{r\in [0,1]}$ defines a $b$-operator $P_0$ on $[0,1]\times X$
and  $P_0$ is a perturbation of the $b$-operator
associated to $D$, call it $D_0$, given explicitly by $D_0=x(1-x) \partial_x +D_X$. Using the proof of  [Lemma 3.13 (4)] in \cite{Wahl-Ramanujan} we have that the index class of the operator $P_0$ is equal to zero; indeed there exists an homotopy between $P_0$ and an invertible operator.

\noindent
We conclude the proof by applying  our main theorem:
$$0= \eta_g (D_X+B (0))- \eta_g (D_X+B (1)) + \int_{([0,1]\times X)^g}  AS_g (D_0) $$
where on the left hand side we have the pairing of the $b$-index class on $I\times X$, which is 0, with the
0-cocycle $\tau^X_g$; indeed, as $g$ is non-elliptic the 
integral on the right hand side vanishes and the proof is complete. 
\end{proof}

\subsection{Rho numbers associated to a $G$-homotopy equivalence}

\medskip
We now come to the signature operator and to the rho-numbers associated to a G-homotopy equivalence.\\
Consider $\mathbf{f}:X_1\to X_2$,  an oriented $G$-homotopy equivalence between two cocompact oriented
$G$-proper manifolds. It is not difficult to show, see \cite{Fukumoto}, that $\mathbf{f}$ is automatically a proper map. Consider $X:= X_1\sqcup (-X_2)$ (where $-X_2$ is the same manifold as $X_2$ with the opposite orientation). Note that $X$ is equipped with a proper $G$ action with the compact quotient $X_1/G\sqcup (-X_2/G)$. We follow the construction in \cite[Section 2]{HigsonPedersenRoe} for a coarse structure on $X$ by taking the entourages of the form
\[
E=\bigcup_{g\in G} gC\times gC
\]
where $C$ runs over all compact subsets of $X$. The Roe algebra $C^*(X, \Lambda^* X)^G$ can be defined \cite[Definition 5.3]{HigsonPedersenRoe} as the completion of locally compact finite propagation operators on $X$. With the definitions, it is not hard to verify that $\mathcal{L}^c_G(X, \Lambda^* X)$ is dense in $C^* (X,\Lambda^* X)^G$.


We go back to the definition of the rho invariant associated to $\mathbf{f}$.
 Fukumoto, see \cite{Fukumoto}, has extended the Hilsum-Skandalis proof \cite{HS} of the homotopy invariance of the signature index class
 from the case of  free actions to the case of proper actions. 
Recall that the Hilsum-Skandalis result was sharpened in \cite{PiazzaSchick_BCrho}
 where it was proved, building heavily on  the work of Hilsum-Skandalis, that there exists a  perturbation 
  $A_\mathbf{f}$ of the signature operator on $X= X_1\sqcup (-X_2)$  with the property that
$D^{{\rm sign}}_X+A_\mathbf{f}$ has  the same domain of $D^{{\rm sign}}_X$
 and it is $L^2$-invertible. Further contributions on the structure of this perturbation were then given by Wahl
 \cite{Wahl}, Zenobi \cite{zenobi-JTA} and, above all,
Spessato \cite{spessato}. For the present article it is important to remark that the work of Spessato is 
done entirely in the framework
of Roe algebras; moreover his arguments can be easily extended to the proper case. 
We obtain in this way a perturbation $A_{\mathbf{f}}\in C^* (X,\Lambda^* X)^G$ such that $D^{{\rm sign}}_X+A_\mathbf{f}$
is $L^2$-invertible.
By density, we conclude that there exists a perturbation 
$B_{\mathbf{f}}\in \mathcal{L}^c_G (X,\Lambda^* X)$
such that $D^{{\rm sign}}_X+B_{\mathbf {f}}$ is $L^2$-invertible.
We can thus define the rho-number of the homotopy equivalence $\mathbf {f}$ as
\begin{equation}\label{rho-f}
\rho_g (\mathbf{f}):= \eta_g  (D^{{\rm sign}}_X+B_\mathbf{f})\,.
\end{equation}
%
%
%
The number $\rho_g(\mathbf{f})$ does not depend on the choices that enter 
into the definition of the perturbation. Indeed, proceeding as in \cite[Section 4, page 165]{Wahl}, we know that the two perturbations,
 $A_\mathbf{f}$ and $\widetilde{A}_\mathbf{f}$, associated to  two different sets of choices for the Hilsum-Skandalis 
 perturbation for $\mathbf{f}$
can be joined by a path of invertible perturbations.  This family can be uniformly approximated by a family with values 
in $\mathcal{L}^c_G (X,\Lambda^* X)$.
By applying Proposition \ref{prop:stability} we then obtain that $\rho_g (\mathbf{f})$ is independent of the choices that enter into the definition of the Hilsum-Skandalis perturbation and also on the choice of the element
$B_{\mathbf{f}}\in \mathcal{L}^c_G (X,\Lambda^* X)$ that approximates $A_{\mathbf{f}}\in
C^* (X,\Lambda^* X)^G$.\\
However,  the signature operator also depends on the choice of the metric; if $h_r$, $r\in [0,1]$, are metrics 
on  $X:= X_1\sqcup (-X_2)$, with $h_r$ constant in $[0,1/4]$ and $[3/4,1]$,  then we can consider the family 
$$D^{{\rm sign}}_{h_r} + A_\mathbf{f} (r).$$
(The Hilsum-Skandalis perturbation depends on the operator $D^{{\rm sign}}_{h_r}$ and for this
reason it will change with $r$, even if the homotopy equivalence is fixed.)
Given two approximations of $A_\mathbf{f} (0)$ and $A_\mathbf{f} (1)$, call them 
$B_\mathbf{f} (0)$ and $B_\mathbf{f} (1)$, we can consider a uniform family of approximations 
$B_\mathbf{f} (r)$ for $A_\mathbf{f} (r).$
Proceeding similarly to Proposition \ref{prop:stability} we can consider the operator
$P:= \partial_r + D^{{\rm sign}}_{h_r} + B_\mathbf{f} (r)$ on the cylinder $\RR\times X$
endowed with a metric $H$ induced by $h_r$ on $[0,1]\times X\subset \RR\times X$ and extended constantly 
on $r<0$ and $r>1$. Similarly, we can consider the associated $b$-operator $P_0$. \\
Since the family $\{D^{{\rm sign}}_{h_r} + B_\mathbf{f} (r)\}_{r\in [0,1]}$ is invertible, it follows again
from the proof of  Lemma 3.13 (4) in \cite{Wahl-Ramanujan}
that  the  index class of $P_0$ is equal to 0.\\
We now apply again our index theorem and get, 
$$0= \eta_g (D^{{\rm sign}}_{g_0} +B_\mathbf{f} (0))- 
\eta_g (D^{{\rm sign}}_{g_1} +B_\mathbf{f} (1))
 + \int_{([0,1]\times X)^g}  AS_g (H). $$
 Choosing  $g\in G$ to be a non-elliptic element we obtain, summarizing, the following:

 \begin{proposition}\label{prop:independence-rho} Let $g\in G$ be a non-elliptic element and let $\mathbf{f}:X_1\to X_2$ be  an oriented $G$-homotopy equivalence between two cocompact oriented
$G$-proper manifolds. Let $X=X_1\sqcup -X_2$, endowed with a metric $h:=h_1\sqcup h_2$. Then 
 the rho-number $\rho_g (\mathbf{f}):= \eta_g  (D^{{\rm sign}}_h+B_\mathbf{f})$ is well defined,
 independent of the choice of the metric and of the choices  that enter into the definition of the Hilsum-Skandalis perturbation $A_\mathbf{f}$ and its approximation $B_\mathbf{f}$.
 \end{proposition}
 

\subsection{Bordism properties}\label{subsect:h-cobordism}
The delocalized APS index theorem proved in this article can be used in order to study the bordism properties 
of these rho invariants $\rho_g (\mathbf{f})$, as we shall now explain. 

\medskip
Let $Y$ be a cocompact  orientable $G$-proper manifold. A $G$-structure over $Y$ is a cocompact orientable $G$-proper manifold $M$ together with an oriented
$G$-homotopy equivalence $M\xrightarrow{\mathbf{f}} Y$. Two $G$-structures over $Y$,  $(M_1\xrightarrow{\mathbf{f}_1} Y)$ and    $(M_2\xrightarrow{\mathbf{f}_2}  Y)$,
are $G$-h-cobordant  if there exists an orientable  $G$-proper manifold with boundary $W$ and an oriented  $G$-homotopy equivalence $\mathbf{F}: W\to Y\times [0,1]$ 
restricting to $(M_1\xrightarrow{\mathbf{f}_1} Y)$ and  $(M_2\xrightarrow{\mathbf{f}_2}  Y)$
over  the boundary. We can and we shall assume that $F$ is equal to ${\mathbf{f}_j} \otimes {\rm Id}$ on the collar neighbourhoods of the two boundary components.

Consider such a $G$-h-cobordism $\mathbf{F}: W\to Y\times [0,1]$.
By combining the work of Fukumoto \cite{Fukumoto}, Wahl \cite{Wahl}  and Spessato \cite{spessato} it is possible to prove that  there exists a well-defined signature index class
associated to $X:=W\sqcup - (Y\times [0,1])$ and $\mathbf{F}$ (for this it would suffice that $\mathbf{F}|_{\partial}$ is a $G$-homotopy equivalence)
and that, moreover, this index class is equal to zero (this is because $\mathbf{F}$ is a global  $G$-homotopy equivalence). This vanishing when $\mathbf{F}$ is a global  $G$-homotopy equivalence
is the consequence of the existence of a global perturbation $A(\mathbf{F})$ making the signature operator on $X$
invertible and giving a 
 a suitable extension of the perturbations $A_{\mathbf {f}_1}\otimes {\rm Id}$ and  $A_{\mathbf {f}_2}\otimes {\rm Id}$
 on  the collar neighbourhood of $\partial X$. Let $B_{\mathbf {f}_1}$ and  $B_{\mathbf {f}_2}$ be two approximations  of $A_{\mathbf {f}_1}$ and  $A_{\mathbf {f}_2}$ respectively. Set
 $Y_j:= M_j \sqcup (-Y)$.
 An explicit homotopy as in Wahl, together  with Remark \ref{comparison-wahl}, allows to prove that
 $$\Ind (D^{{\rm sign},+}_X+A(\mathbf{F}))= \Ind ((D^{{\rm sign},+}_X, B_{\mathbf {f}_1}\sqcup B_{\mathbf {f}_2})\,.$$
 Since the left hand side  is 0 we can apply our index theorem to the right hand side and obtain
 $$0=\eta_g  (D^{{\rm sign}}_{Y_1} +B_{\mathbf{f}_1}) - \eta_g  (D^{{\rm sign}}_{Y_2} +B_{\mathbf{f_2}}) + \int_{W^g} c^g AS_g (
 D^{{\rm sign}}_W).$$
 By choosing $g$ non-elliptic we conclude that the following Proposition holds

\begin{proposition}
If two $G$-structures over $Y$,  $(M_1\xrightarrow{\mathbf{f}_1} Y)$ and    $(M_2\xrightarrow{\mathbf{f}_2}  Y)$,
are $G$-h-cobordant and $g\in G$ is a non-elliptic element, then 
  \begin{equation}\label{rho-bord}
  \rho_g (\mathbf{f}_1)=\rho_g (\mathbf{f}_2).
  \end{equation}
  \end{proposition}

\begin{remark} It is unclear whether such a result, the $G$-h-cobordism invariance of the  rho number in \eqref{rho-f},
 can have immediate   geometric applications, as in the case of free proper $\Gamma$-manifolds,
with $\Gamma$-discrete. 
See for example \cite{ChangWeinberger}, \cite{HigsonRoe3}, \cite{PiazzaSchick_sig}, \cite{WXY}.
 In that case such a result
can be used in order to define invariants on the structure sets $\mathcal{S}(M)$ ($M$ an orientable smooth compact manifold) and give sufficient conditions in order to show that $| \mathcal{S}(M) | = \infty$.
In order 
to have such geometric applications in the present context would seem to require the development 
of a {\em proper $G$-equivariant} surgery sequence for a {\it non-compact} Lie group $G$, certainly a non trivial task.
See \cite{Do-S} for the case in which $G$ is compact. 
We leave this fundamental question to future research.\\
\end{remark}

\subsection{Delocalized eta invariants of the signature operator}

In the previous subsection, {\em using the signature operator},
 we have defined the rho invariant associated to a non-elliptic element $g\in G$ and to an oriented  $G$-homotopy
 equivalence $\mathbf{f}: M\to Y$ between {\em two} cocompact $G$-proper manifolds; moreover, we were able to study the  bordism properties of such an invariant. 
 
 Under additional assumptions we can in fact define $\eta_g (D^{{\rm sign}}_{\mathbf{h}})$
 for a {\it single} cocompact $G$-proper manifold $X$ of dimension  $2m-1$ endowed with a $G$-invariant 
 riemannian metric $\mathbf{h}$. Moreover we can show that $\eta_g (D^{{\rm sign}}_{\mathbf{h}})$ enters as a boundary correction term
 in a suitable APS index theorem on $Y$ with $\partial Y=X$.  The case of Galois coverings has been treated in  
 \cite{LP-AGAG}, \cite{LLP},
 \cite{Wahl-product}. We now proceed to give the details of this program.\\
  
  Recall that $\dim X=2m-1$. We shall make the following assumption.

\begin{assumption}\label{assume-middle}  the differential form Laplacian in degree $m$, $\Delta^{[m]}_h$, is $L^2$-invertible. Put it differently, the space of $L^2$-harmonic forms of degree $m$, $\mathcal{H}^m_{(2)} (X,\mathbf{h})$ is trivial.
\end{assumption}
  
\begin{remark} In the case of Galois coverings we 
know, by Dodziuk theorem \cite{dodziuk}, that this assumption is metric independent (in fact, it is a homotopy-invariant condition because of the well-known results of Gromov and Shubin on the Novikov-Shubin invariants 
\cite{gromov-shubin}). In this more general context we
do not know whether this metric independence  is true. 
\end{remark}

\noindent
Under our assumption we have an $L^2$-orthogonal decomposition
$$\Omega_{L^2} (X)= V_{X} \oplus W_{X}$$
(thus invariant under the induced action of $G$) with the property that 
\begin{itemize}
\item
$D_{\mathbf{h}}\equiv D^{{\rm sign}}_{X,\mathbf{h}}$  and the Hodge involution $\tau^\mathbf{h}_{X}$ are both diagonal with respect to this decomposition (we of course take the intersection of the domain of $D_{\mathbf{h}}$
with $V_{X}$ and $ W_{X}$ respectively);
\item $D_{\mathbf{h}}$ restricted to $V_{X}$ admits a bounded $G$-equivariant inverse;
\item there exists a bounded
$G$-equivariant operator 
$\mathcal{I}$ on $\Omega_{L^2} (X)$ which vanishes on $V_{X}$, is an involution on 
$W_{X}$ and anticommutes there with  (the restrictions to $W_X$ of) $\tau^h_{X}$ and $D_{\mathbf{h}}$.\end{itemize}
Indeed
\begin{itemize}
\item we can  choose $V_{X}=\overline{d\Omega^{m-1}_{L^2} (X)}\oplus \overline{d^*\Omega^{m}_{L^2} (X)}$ where  $d$ and $d^*$ are the closure of $d$ and $d^*$ respectively and they are  defined on their
respective domains (as we are  in the complete case, these operators are essentially closed);
\item  $W_{X} = (V_{X})^\perp$ and we define $\mathcal{I}$  
on $W_X$ by considering  $\Omega^{<}$ and $\Omega^{>}$
as the
subspaces of $W_{X} $ made of forms of degree $\leq k-1$ and $\geq k$ respectively
and declaring $\mathcal{I}$ to be $- {\rm Id} $ on $\Omega^{<}$  and ${\rm Id}$ on 
$\Omega^{>}$; we define $\mathcal{I}$ to be zero on $V_X$.
\end{itemize}
%
%
Consider the 0-th order differential operator defined by Wahl in \cite{Wahl-product}:
$$A^{\mathbf{h}}_{{\rm 0}}:= i\mathcal{I}\tau^{\mathbf{h}}_{X}.$$ This is a 0-th order $G$-equivariant trivializing perturbation 
for $D_{\mathbf{h}}$, which is diagonal with respect to  $V_{X} \oplus W_{X}$, equal to 0 on $V_{X}$
 and $\ZZ_2$-graded odd with respect to the grading defined by  $\mathcal{I}$ on $W_{X}$. 
 Wahl calls the operator $A^{\mathbf{h}}_{{\rm 0}}$ the {\it canonical} perturbation associated to $\mathbf{h}$ 
 under Assumption \ref{assume-middle}.
 
 \noindent
 For $R>>0$ we fix a smooth function $\phi_R\in \C^\infty (\mathbb{R}, [0,1])$  such that
 $\phi_R (x) =1 $ for $|x|<R$ and $\phi_R (x) =0 $ for $|x|> 2R$ 
 
 \smallskip
  \noindent
 {\bf Claim.} {\it The operator $A^{\mathbf{h}}:= \phi_R (D_{\mathbf{h}}) \circ A^{\mathbf{h}}_{{\rm 0}} \circ \phi_R (D_{\mathbf{h}})$ is in $C^* (X,E)^G$, 
  is diagonal with respect to  $V_{X} \oplus W_{X}$, is equal to 0 on $V_{X}$,
 is $\ZZ_2$-graded odd with respect to the grading defined by  $\mathcal{I}$ on $W_{X}$
 and  is a trivializing operator for $D$.}

 \smallskip
  \noindent
  {\em Proof.} Notice that $\phi_R (D_{\mathbf{h}})$ is diagonal with respect to $V_{X} \oplus W_{X}$ since $D_{\mathbf{h}}$ is. 
  Moreover, it is a classic result that $\phi_R (D_{\mathbf{h}})$  is an element in $C^* (X,E)^G$. Since $A^{\mathbf{h}}_{{\rm 0}}$
  is  in $D^* (X,E)^G$ and $C^* (X,E)^G$ is an ideal in $D^* (X,E)^G$, it follows that 
  $A^{\mathbf{h}}$ is in $C^* (X,E)^G$. The other properties follow from the properties of $A^{\mathbf{h}}_{{\rm 0}}$ once we
  choose $R$ large enough. The Claim is proved.
%
  
  \medskip
  \noindent
  We call a trivializing perturbation $C$ such as $A^{\mathbf{h}}$ a {\em symmetric} trivializing perturbation, by definition. If 
$C'$ is in $C^* (X,E)^G$, 
  is diagonal with respect to  $V_{Y} \oplus W_{Y}$, is equal to 0 on $V_{Y}$,
 is $\ZZ_2$-graded odd with respect to the grading defined by  $\mathcal{I}$ on $W_{Y}$
 and  is a trivializing perturbation for $D$, then  we obtain a symmetric trivializing perturbation $C$ 
by  compressing $C'$  
 using $\phi_R$, obtaining in this way a symmetric trivializing perturbation. We can and we shall 
approximate  $C$ with an element   $B \in \mathcal{L}^c_G (X,\Lambda^* X)$.
 
%
%
%
%
%
%
%
%
%

 Let $(X,\mathbf{h})$ be an odd dimensional $G$-proper cocompact riemannian manifold
 satisfying Assumption \ref{assume-middle}.
Consider  $D^{{\rm sign}}_{\mathbf{h}}$ and let $A^{\mathbf{h}}$ be the associated  trivializing perturbation 
defined by smoothing  $A^{\mathbf{h}}_0$, the canonical 0-th order trivializing operator of Wahl.
By density we can  and we shall choose  $A^{\mathbf{h}}\in \mathcal{L}^c_G (X,\Lambda^* X)$, obtaining in this way
a {\em generalized} symmetric trivializing perturbation. 
%
%
\medskip
\noindent
For a semisimple element $g\in G$ we consider
 $
  \eta_g (D^{{\rm sign}}_{\mathbf{h}}+A^{\mathbf{h}}).$

\begin{proposition}\label{independence-sym}
Consider the delocalized eta invariant $ \eta_g (D^{{\rm sign}}_{\mathbf{h}}+A^{\mathbf{h}})$. 
\begin{enumerate}
\item If $B\in \mathcal{L}^c_G (X,\Lambda^* X)$ is another
 {\it  symmetric} trivializing perturbation then 
  \begin{equation}\label{inde-symmetric}
   \eta_g (D^{{\rm sign}}_{\mathbf{h}}+A^{\mathbf{h}})=  \eta_g (D^{{\rm sign}}_{\mathbf{h}}+B)
   \end{equation}
\item  Let $g$ be a non-elliptic element. Let $h_0$ and $h_1$  be  riemannian metrics
  satisfying Assumption \ref{assume-middle} and such that there exists a path of riemannian metrics
  $\{h_r\}_{r\in [0,1]}$ joining them and satisfying Assumption \ref{assume-middle}. Then
 $$\eta_g (D^{{\rm sign}}_{\mathbf{h}_0}+A^{\mathbf{h}_0})= \eta_g (D^{{\rm sign}}_{\mathbf{h}_1}+A^{\mathbf{h}_1}).$$
 \end{enumerate}
  \end{proposition}

\begin{proof}
Let us consider the first property.
We will prove that  
$ \eta_g (D^{{\rm sign}}_{\mathbf{h}}+B)=  \eta_g (D^{{\rm sign}}_{\mathbf{h}}+C) 
$
for any pair of  symmetric trivializing perturbations.\\

 To this end we recall a classic property of eta invariants of Dirac operators on even dimensional manifolds, namely that they vanish. 
 Assume that $N$ is an even dimensional compact manifold, $E$ is a $\mathbb{Z}_2$-graded bundle of Clifford modules
 and $D$ is a $\mathbb{Z}_2$-graded odd Dirac operator acting on the sections of $E$. For simplicity we assume that $D$ is $L^2$-invertible. If we denote by
 $\mathcal{I}$ the grading for $E$ this means that $D\circ \mathcal{I} + \mathcal{I}\circ D=0$. 
 With respect to the decomposition $E=E^+\oplus E^-$ induced by $\mathcal{I}$ we can write the Dirac operator 
 in the block-form
$$ \left( \begin{array}{cc}0 & D^-\\
D^+&  0
\end{array} \right).
$$
The operator entering into the definition of the eta invariant, namely $De^{-t D^2}$ can be written in block form as
$$ \left( \begin{array}{cc}0 & D^-\\
D^+&  0
\end{array} \right) \circ  \left( \begin{array}{cc}e^{-t D^- D^+} & 0\\
0&  e^{-tD^+ D^-}
\end{array} \right).
$$
and this has trace equal to zero. Thus the eta invariant of $D$, $\eta(D)$,  vanishes. The same result
holds for an $L^2$-invertible perturbed operator $D+B$, with $B$ a smoothing operator, provided 
that $B$ is  $\mathbb{Z}_2$-graded odd, that is  $B\circ \mathcal{I} + \mathcal{I}\circ B=0$. Under these assumptions
$D+B$ is $\mathbb{Z}_2$-graded odd and exactly the same argument shows that $\eta (D+B)=0$.

\smallskip

Let us go back to $G$-proper manifolds and to our task, proving
that $ \eta_g (D^{{\rm sign}}_{\mathbf{h}}+B)=  \eta_g (D^{{\rm sign}}_{\mathbf{h}}+C) 
$ for two symmetric trivializing perturbations $B$ and $C$.  We write $D$ for $D^{{\rm sign}}_{\mathbf{h}}$. Recall that
$D$ is diagonal
 with respect to the decomposition $\Omega_{L^2} (X)= V_{X} \oplus W_{X}$; $D$ is invertible on $V_{X}$
 and its restriction to $W_{X}$ is $\mathbb{Z}_2$-graded odd with respect to the grading $\mathcal{I}$ of $W_{X}$. Similarly
 $D+B$ is block diagonal, equal to $D$ on $V_{X}$ and invertible there;  moreover its restriction to $W_{X}$
 is  $\mathbb{Z}_2$-graded odd with respect to the grading $\mathcal{I}$ of  $W_{X}$.
 Now we write explicitly the operator in the eta integrand, that is $(D+B)\exp (-t (D+B)^2)$ with respect to the
 decomposition $\Omega_{L^2} (X)= V_{X} \oplus (W^+_{X}\oplus W^-_{X})$; for simplicity we denote
 temporarily this decomposition as $V\oplus (W^+\oplus W^-)$. We have:
 $$ \left( \begin{array}{ccc}D_V & 0 & 0\\
0&  0 & D_{W^-}\\
0&D_{W^+}&0
\end{array} \right) \circ  \left( \begin{array}{ccc}e^{-t D^2_V} & 0 & 0\\
0 &  e^{-t D_{W^+} D_{W^-}} & 0\\
0&0& e^{-t D_{W^-} D_{W^+}} 
\end{array} \right).
$$
Similarly we can write $(D+B) e^{-t (D+B)^2}$ as
 $$ \left( \begin{array}{ccc}D_V & 0 & 0\\
0&  0 & D_{W^-}+B_-\\
0&D_{W^+}+B_+&0
\end{array} \right) \circ  \left( \begin{array}{ccc}e^{-t D^2_V} & 0 & 0\\
0 &  e^{-t (D_{W^+} + B_+) (D_{W^-}+B_-)} & 0\\
0&0& e^{-t (D_{W^-}+B_-)( D_{W^+} + B_+)} 
\end{array} \right),
$$
that we rewrite as
$$ \left( \begin{array}{ccc}D_V e^{-t D^2_V}& 0 & 0\\
0&  0 & (D_{W^-}+B_-) e^{-t (D_{W^+} + B_+) (D_{W^-}+B_-)} \\
0&(D_{W^+}+B_+) e^{-t (D_{W^-}+B_-)( D_{W^+} + B_+)} &0
\end{array} \right) .
$$
It is clear that when we apply the trace $\tau^{X}_{g}$ we get a contribution in the eta integrand that only involves
$D_V e^{-t D^2_V}$. This proves that $ \eta_g (D^{{\rm sign}}_{\mathbf{h}}+B)=  \eta_g (D^{{\rm sign}}_{\mathbf{h}}+C) 
$ for two symmetric trivializing perturbations $B$ and $C$.\\

Let us prove the second property:  if  $\mathbf{h}_r$, $r\in [0,1]$, is a family of  metrics satisfying Assumption \ref{assume-middle}  
then from the family of 0-th order perturbations of Wahl,
$$A^{\mathbf{h}_r}_0 := i\mathcal{I}\tau^{\mathbf{h}_r}_{X},$$ 
we can construct  a 1-parameter family of 
trivializing perturbations $A^{\mathbf{h}_r}$ for $D^{{\rm sign}}_{\mathbf{h}_r}$ and then we have, as before,
 $$\eta_g (D^{{\rm sign}}_{\mathbf{h}_0}+A^{\mathbf{h}_0})=\eta_g (D^{{\rm sign}}_{\mathbf{h}_1}+A^{\mathbf{h}_1})$$
 for a non-elliptic element $g\in G$. Indeed, 
 as the two operators $D^{{\rm sign}}_{\mathbf{h}_0}+A^{\mathbf{h}_0}$ and 
 $D^{{\rm sign}}_{\mathbf{h}_1}+A^{\mathbf{h}_1}$
can be joined by the family of invertible operators $D^{{\rm sign}}_{\mathbf{h}_r}+A^{\mathbf{h}_r}$, the argument given in the proof of Proposition \ref{prop:independence-rho} implies also in this case  that 
$0=\eta_g (D^{{\rm sign}}_{\mathbf{h}_0}+A^{\mathbf{h}_0})-\eta_g (D^{{\rm sign}}_{\mathbf{h}_1}+A^{\mathbf{h}_1})+ 0$,
where the last summand on the right hand side is a zero because $g$ is non-elliptic.
\end{proof}

  \begin{definition}\label{def:eta-symmetric}
  Let $(X,\mathbf{h})$ be an odd dimensional $G$-proper cocompact riemannian manifold
 satisfying Assumption \ref{assume-middle}.
 We set
 \begin{equation}\label{def-eta-sign}
 \eta_g (D^{{\rm sign}}_{\mathbf{h}}): =  \eta_g (D^{{\rm sign}}_{\mathbf{h}}+B)
 \end{equation} 
  for any {\it  symmetric} trivializing perturbation $B\in \mathcal{L}^c_G (X,\Lambda^* X)$; this is well-defined, independent on the choice 
  of the symmetric trivializing perturbation $B$, because of \eqref{inde-symmetric}.
 \end{definition}


\subsection{A  APS-index theorem for the signature operator on $G$-proper manifolds} $\;$\\
We are now ready to state and prove a delocalized  APS index theorem for the signature operator 
on $G$-proper manifolds with boundary. 

\begin{theorem}
Let $(Y_0,\mathbf{h}_0)$ be
an even-dimensional  $G$-proper manifold with boundary, with $G$-invariant metric $\mathbf{h}_0$  that is product-like near the boundary. Let
$D^{{\rm sign}}_{\mathbf{h}_0}$ be the associated signature operator.
Let $(Y,\mathbf{h})$ be the associated $b$-manifold with associated $b$-differential operator $D^{{\rm sign}}_{\mathbf{h}}$. 
Let $\mathbf{h}_{\partial}$ be the boundary metric and let $X=\partial Y_0$, endowed with the metric $\mathbf{h}_{\partial}$.
Assume that $\mathbf{h}_{\partial}$ satisfies Assumption \ref{assume-middle}.  Let $A_\partial$ be a {\it symmetric} trivializing perturbation
for the boundary operator $D^{{\rm sign}}_{\mathbf{h}_\partial}$;  we can and we shall assume that  $A_\partial\in \mathcal{L}^c_G (X,\Lambda^* X)$. We have:

\begin{enumerate}
\item there exists a well-defined signature class 
$\sigma_{C^*} (Y,\partial Y, \mathbf{h}, A_\partial)\in K_0 ( C^* (Y_0\subset Y)^G)$
and  thus a well defined delocalized signature $\sigma_g (Y,\partial Y, \mathbf{h},A_\partial):= \langle \tau^Y_g, \sigma_{C^*} (Y,\partial Y, \mathbf{h},A_\partial) \rangle$ ;
\item 
the following APS delocalized signature formula holds
$$\sigma_g (Y,\partial Y, \mathbf{h},A_\partial)= \int_{(Y_0)^g} c^g {\rm AS}_g (D^{{\rm sign}}_0) - 
\frac{1}{2} \eta_g (D^{{\rm sign}}_{\mathbf{h}_{\partial}} + A_{\partial})= \int_{(Y_0)^g} c^g {\rm AS}_g (D^{{\rm sign}}_0) - 
\frac{1}{2} \eta_g (D^{{\rm sign}}_{\mathbf{h}_{\partial}})
$$
with $\eta_g (D^{{\rm sign}}_{\mathbf{h}_{\partial}})$ as in \eqref{def-eta-sign}. Here Proposition 
\ref{independence-sym} and Definition \ref{def:eta-symmetric} have been used in the second equality.\\ Notice that, consequently, the left hand side does not depend on the choice of $A_\partial$ and can hence be denoted $\sigma_g (Y,\partial Y, \mathbf{h})$.  Let $N^g$ be the normal bundle of the fixed point submanifold $(Y_0)^g$ in $Y_0$. The explicit expression $ {\rm AS}_g (D^{{\rm sign}}_0)$ is a product of $L^g(N^g)$ and $L((Y_0)^g)$, where $L^g(N^g)$ is the $g$-twisted $L$-class of Hirezruch associated to the $g$ action on $N^g$ and $L((Y_0)^g)$ is the $L$-class of $(Y_0)^g$, c.f. \cite[Theorem 4.7]{BGV}, \cite[Theorem 14.5]{LaMi}. 
More explicitly,  if   $R_{N^g}$ and $R_{(Y_0)^g}$ are the curvatures forms on $N^g$ and $(Y_0)^g$ associated to the Levi-Civita connection on $TY_0$, 
\[
L^g(N^g)=\frac{\det^{\frac{1}{2}}\Big(g\frac{R_{N^g}}{\tanh R_{N^g}}\Big)}{\det^{\frac{1}{2}}(1-g\exp(-R_{N^g}))},\ \ \ 
L((Y_0)^g)= {\det}^{\frac{1}{2}}  \Big(\frac{R_{(Y_0)^g/2}}{\tanh (R_{(Y_0)^g}/2)}\Big).
\]

\item  if $\mathbf{h}(r)$, $r\in [0,1]$ is a 1-parameter family of metrics with $ \mathbf{h}_{\partial} (r)$ satisfying Assumption \ref{assume-middle}, then $\sigma_{g} (Y,\partial Y, \mathbf{h}_0)= \sigma_{g} (Y,\partial Y, \mathbf{h}_1)$.
\end{enumerate}
\end{theorem}

\begin{proof}
The signature class $\sigma_{C^*} (Y,\partial Y, \mathbf{h}, A_\partial)\in K_0 ( C^* (Y_0\subset Y)^G)$ is equal, by definition,
to $\Ind (D^{{\rm sign}}_{Y,\mathbf{h}}, A_{\partial})$. 
The signature formula is a direct consequence of our delocalized APS index theorem
and definition \eqref{def-eta-sign}.\\
 The independence on the metric can be proved as follows:
if $\mathbf{h}(r)$, $r\in [0,1]$ is a 1-parameter family of metrics with $ \mathbf{h}_{\partial} (r)$ satisfying Assumption \ref{assume-middle}
and if 
$A^{\mathbf{h}_\partial (r)}$ is the associated  1-parameter family of trivializing perturbation 
built out of Wahl's canonical 0-th order perturbation, then
$\Ind (D^{{\rm sign}}_{Y,h(0)}, A^{h_\partial (0)})=\Ind (D^{{\rm sign}}_{Y,h(1)}, A^{h_\partial (1)})$; indeed, 
through $A^{\mathbf{h}_\partial (r)}$  we can define a continuous family of  operators,
$D^{{\rm sign}}_{Y,\mathbf{h}(r)} + A^{\mathbf{h}_\partial (r)}$
to which we can apply
well-know results.
Thus 
$\sigma_{C^*} (Y,\partial Y, \mathbf{h}_0, A^{\mathbf{h}_\partial (0)})= \sigma_{C^*} (Y,\partial Y, \mathbf{h}_1, A^{\mathbf{h}_\partial (1)})$ in $K_0 ( C^*(Y_0\subset Y)^G)$ and the assertion 
$\sigma_{g} (Y,\partial Y, \mathbf{h}_0)= \sigma_{g} (Y,\partial Y, \mathbf{h}_1)$ follows.

\end{proof}

\section{Further results: the gap case}\label{sect:further}
In this section we shall treat a different non-invertible situation. That of an operator that has $0$ isolated in its $L^2$-spectrum. We shall introduce and study the delocalized eta invariant and 
prove a corresponding APS index theorem. 

We use the same notations as in the previous subsections.

\subsection{Delocalized eta invariants associated to the gap case}
\begin{definition}
\label{def:perturbed Dirac}
For any $\Theta>0$, let us introduce the perturbed Dirac operator 
\[
D_\Theta =D + \Theta. 
\]
Let $\kappa^\Theta_t$ be the smoothing kernel of the following operator
\[
D_\Theta \cdot \exp(-tD^2_\Theta). 
\]
\end{definition}
For any small number $s >0$, we split the integral
\begin{equation}
\label{eq regulareta} 
\begin{split}
\int_s^\infty \tau^{X}_{g} (D_\Theta\exp (-tD_\Theta^2) \frac{dt}{\sqrt{t}} 
=&\int_s^\infty \int_X\int_G c_G(h)c(hy) {\rm tr}(\kappa^\Theta_t(y, gy)g)\; dh dy \frac{dt}{\sqrt{t}}\\
=&\int_s^\infty \int_{N(r)}\int_G c_G(h)c(hy) {\rm tr}(\kappa^\Theta_t(y, gy)g)\; dh dy \frac{dt}{\sqrt{t}}\\
+&\int_s^\infty \int_{N^c(r)}\int_G c_G(h)c(hy) {\rm tr}(\kappa^\Theta_t(y, gy)g)\; dh dy \frac{dt}{\sqrt{t}}, 
\end{split}
\end{equation}
where $N(r)$ is a properly chosen $r$-neighborhood of the fixed point submanifold $X^g$ in $X$, c.f. \cite[Proposition 5.27 and Lemma 5.30]{PPST}. 

Note that the integral
\[
\int_s^\infty \tau^{X}_{g} (D_\Theta\exp (-tD_\Theta^2) \frac{dt}{\sqrt{t}}
\]
might be divergent as $s \to 0$. The reason is that $D_\Theta$ is no longer associated to a Clifford connection and thus 
\[
\int_s^\infty \int_{N(r)}\int_G c_G(h)c(hy) {\rm tr}(\kappa^\Theta_t(y, gy)g)\; dh dy \frac{dt}{\sqrt{t}}
\]
will in general tend to infinity as $s \to 0$. 

\begin{proposition}
For the terms in the righthand side of (\ref{eq regulareta}),  the following holds
\begin{enumerate}
	\item The integrand of the first summand has the asymptotic expansion \cite[Section 2]{Zhangwp}:
\[
\int_{N(r)}\int_G c_G(h)c(hy) {\rm tr}(\kappa^\Theta_t(y, gy)g)\; dh dy \sim  t^{-\frac{\dim X+1}{2}} \left(\sum_{i=0}^\infty a_i(r) \cdot t^i\right)
\]
We define
\[
A(r) \colon ={\rm LIM}_{s\downarrow 0}\frac{1}{\sqrt{\pi}} \int_s^\infty \int_{N(r)}\int_G c_G(h)c(hy) {\rm tr}(\kappa^\Theta_t(y, gy)g)\; dh dy \frac{dt}{\sqrt{t}}
\]
as the coefficient of $s^0$ in this asymptotic expansion.  
\item
The integral 
\[
\int_{N^c(r)}\int_G c_G(h)c(hy) {\rm tr}(\kappa^\Theta_t(y, gy)g)\; dh dy
\]
is of exponential decay as $t \to 0$. And thus
\[
\begin{split}
B(r) \colon =& \lim_{s \to 0} \int_s^\infty\int_{N^c(r)}\int_G c_G(h)c(hy) {\rm tr}(\kappa^\Theta_t(y, gy)g)\; dh dy\frac{dt}{\sqrt{t}}\\
=&\int_0^\infty\int_{N^c(r)}\int_G c_G(h)c(hy) {\rm tr}(\kappa^\Theta_t(y, gy)g)\; dh dy\frac{dt}{\sqrt{t}}
\end{split}
\]
is well defined.
\end{enumerate}	
\end{proposition}

\begin{proof}Part (i) follows from the similar arguments as in \cite[Lemma 5.28]{PPST} via asymptotic expansion \cite[Section 2]{Zhangwp}. And Part (ii) follows from the same arguments as in \cite[Proposition 5.25]{PPST}. We leave the details to the reader.
\end{proof}

\begin{definition}\label{def:modified}
For the perturbed Dirac operator $D_\Theta$, we define the \emph{regularized eta invariant} as
\begin{equation}\label{regularized}
\begin{split}
\eta_g(\Theta) =& {\rm LIM}_{s\downarrow 0}\frac{1}{\sqrt{\pi}} \int_s^\infty \tau^{X}_{g}(D_\Theta \exp (-tD_\Theta^2)) \frac{dt}{\sqrt{t}}\\
=&A(r) + B(r). 
\end{split}
\end{equation}
\end{definition}

\begin{lemma}
The regularized eta invariant $\eta_g(\Theta)$ is independent of the choice of $r$. 
\begin{proof}
Suppose that $r'> r >0$ are two small constants. We want to show that 
\[
A(r) + B(r)  = A(r') + B(r'). 
\]	
By definition, we have 
\[
\begin{split}
A(r') =&{\rm LIM}_{s\downarrow 0}\frac{1}{\sqrt{\pi}} \int_s^\infty\int_{N(r')}\int_G c_G(h)c(hy) {\rm tr}(\kappa^\Theta_t(y, gy)g)\; dh dy \frac{dt}{\sqrt{t}}\\
=&{\rm LIM}_{s\downarrow 0}\frac{1}{\sqrt{\pi}} \int_s^\infty\int_{N(r)}\int_G c_G(h)c(hy) {\rm tr}(\kappa^\Theta_t(y, gy)g)\; dh dy \frac{dt}{\sqrt{t}}\\
+&{\rm LIM}_{s\downarrow 0}\frac{1}{\sqrt{\pi}} \int_s^\infty \int_{N(r')\setminus N(r) }\int_G c_G(h)c(hy) {\rm tr}(\kappa^\Theta_t(y, gy)g)\; dh dy \frac{dt}{\sqrt{t}}.
\end{split}
\]
Because 
\[
\int_{N(r') \setminus N(r)}\int_G c_G(h)c(hy) {\rm tr}(\kappa^\Theta_t(y, gy)g)\; dh dy
\]
is of exponential decay as $t \to 0$, the following limit is finite,
\[
\begin{split}
&{\rm LIM}_{s\downarrow 0}\int_s^\infty \int_{N(r')\setminus N(r) }\int_G c_G(h)c(hy) {\rm tr}(\kappa^\Theta_t(y, gy)g)\; dh dy \frac{dt}{\sqrt{t}}\\
=&\int_0^\infty\int_{N(r') \setminus N(r) }\int_G c_G(h)c(hy) {\rm tr}(\kappa^\Theta_t(y, gy)g)\; dh dy \frac{dt}{\sqrt{t}}< +\infty. 
\end{split}
\]
Thus, 
\[
A(r') = A(r) + \int_0^\infty\int_{N(r') \setminus N(r) }\int_G c_G(h)c(hy) {\rm tr}(\kappa^\Theta_t(y, gy)g)\; dh dy \frac{dt}{\sqrt{t}}. 
\]
On the other hand, we compute
\[
\begin{split}
B(r)= &\int_0^\infty\int_{N^c(r)}\int_G c_G(h)c(hy) {\rm tr}(\kappa^\Theta_t(y, gy)g)\; dh dy\frac{dt}{\sqrt{t}}\\
=&B(r')+ \int_0^\infty\int_{N(r') \setminus N(r) }\int_G c_G(h)c(hy) {\rm tr}(\kappa^\Theta_t(y, gy)g)\; dh dy \frac{dt}{\sqrt{t}}.
\end{split}
\]
This shows that 
\[
A(r) + B(r) = A(r') + B(r'). 
\]
\end{proof}

\end{lemma}
\begin{lemma}\label{lemma:eta-gap}
If  $0$ is isolated in the $L^2$-spectrum of $D$, then 
\begin{enumerate} 
\item the delocalized eta invariant 
$$\eta_g (D)= \frac{1}{\sqrt{\pi}} \int_0^\infty \tau^{X}_{g}(D \exp (-t D^2)) \frac{dt}{\sqrt{t}}$$
is well defined;
\item 	the projection operator $\Pi_{\ker D}\in  \mathcal{L}^\infty_{G, s} (X)$ and $\tau^{X}_g (\Pi_{\ker D})$ is finite.
\end{enumerate}

\begin{proof}
We are assuming  that $D$, and thus $D^2$, has a gap at $0$. Then there exists $\sigma>0$ such that ${\rm spec}_{L^2} (D^2)\cap [-2\sigma,2\sigma]=\{0\}$.
Let $\gamma$ be the path given by the union of the two straight lines ${\rm Im}z=\pm m {\rm Re}z$, $m={\rm tg} \phi>0$, $m$ small, for $|z|>\sigma$ and of
the portion of the circle of radius $\sigma$ joining the  points $$\{\sigma\cos\phi  + i \sigma\sin\phi,\sigma\cos\phi-i\sigma\sin\phi\}$$ and passing through the 
point $-\sigma\equiv -\sigma+i 0$ in the negative real axis. We know that
$$\exp(-tD^2)=\frac{1}{2\pi i}\int_\gamma e^{-t\lambda} (D^2-\lambda)^2 d\lambda\,.$$
Consider now $\gamma_\sigma$, the path obtained by taking  the union of the two straight lines ${\rm Im}z=\pm m {\rm Re}z$, $m={\rm tg} \phi>0$ for $|z|>\sigma$ and 
of the portion of the circle of radius $\sigma$ joining the  points $\{\sigma\cos\phi  + i \sigma\sin\phi,\sigma\cos\phi-i\sigma\sin\phi\}$ and passing through the 
point $\sigma\equiv \sigma+i0$ in the positive real axis. If we denote by $C_\sigma$ the circle of radius $\sigma$ centred at the origin, then
$$\exp(-tD^2)=\frac{1}{2\pi i}\int_{C_\sigma}  e^{-t\lambda} (D^2-\lambda)^2 d\lambda + \frac{1}{2\pi i}\int_{\gamma_\sigma}  e^{-t\lambda} (D^2-\lambda)^2 d\lambda.$$
It is well known that
$$\frac{1}{2\pi i}\int_{C_\sigma}  e^{-t\lambda} (D^2-\lambda)^2 d\lambda= \Pi_{\ker D^2}\equiv\Pi_{\ker D}\,.$$
See \cite{Shubin-Book}. Thus 
\begin{equation}\label{gap-decomposition-heat}
\exp(-tD^2)= \Pi_{\ker D}+ \frac{1}{2\pi i}\int_{\gamma_\sigma}  e^{-t\lambda} (D^2-\lambda)^2 d\lambda.
\end{equation}
This establishes that $$\Pi_{\ker D}\,,\text{ which is equal to }\frac{1}{2\pi i} \int_{C_\sigma}  e^{-t\lambda} (D^2-\lambda)^2 d\lambda \,,
\text{ is in }  \mathcal{L}^\infty_{G, s} (M)$$ (proof as in \cite[Proposition 2.9]{PP2}) and that $\exp(-tD^2)\to \Pi_{\ker D}$ in 
$\mathcal{L}^\infty_{G, s} (M)$ and 
exponentially, as $t\to +\infty$, where the techniques introduced in the proof of Proposition \ref{prop:large-no-boundary-unperturbed} must be used.
 In particular $\tau^{X}_g (\Pi_{\ker D})$
is finite. From \eqref{gap-decomposition-heat} we also see that 
$$ \int_1^\infty \tau^{X}_{g}(D \exp (-t D^2)) \frac{dt}{\sqrt{t}}$$
converges. Since we know that
$$ \int_0^1 \tau^{X}_{g}(D \exp (-t D^2)) \frac{dt}{\sqrt{t}}$$
does converge, see \cite[Proposition 5.29]{PPST}, we conclude that $\eta_g (D)$ is indeed well defined.
\end{proof}

\end{lemma}

\begin{proposition}\label{prop:eta-conv-gap-bis}
 Let $(X,\mathbf{h})$ be an odd dimensional $G$-proper manifold and let 
 $D$ be a $G$-equivariant Dirac operator. 
If  $0$ is isolated in the $L^2$-spectrum of $D$ and $\Theta>0$ is not in the $L^2$-spectrum of $D$ then 
\begin{equation}\label{eta-theta}
\eta_g (\Theta):={\rm LIM}_{s\downarrow 0} \frac{1}{\sqrt{\pi}} \int_s^\infty \tau^{N}_{g}((D_\Theta) \exp (-t(D +\Theta)^2)) \frac{dt}{\sqrt{t}}
\end{equation}
is well-defined. Moreover  if $\Theta>0$ belongs to the gap around $0$, then 
$\lim_{\Theta\downarrow 0} \eta_g (\Theta)$ exists and we have
\begin{equation}\label{limit-eta}
\lim_{\Theta\downarrow 0} \eta_g (\Theta)=\eta_g (D) + \tau^{X}_g (\Pi_{\ker D}). 
\end{equation}

\begin{proof}
We shall adapt to the present context an argument due to Melrose, see \cite[Prop. 8.38]{Melrose-Book}. Let us break 
\[
\eta_g(\Theta)  = {\rm LIM}_{s\downarrow 0}\frac{1}{\sqrt{\pi}} \int_s^1\tau^{X}_{g}((D_\Theta ) \exp (-sD_\Theta^2)) \frac{dt}{\sqrt{t}} + \frac{1}{\sqrt{\pi}} \int_1^\infty \tau^{X}_{g}((D_\Theta ) \exp (-sD_\Theta^2)) \frac{dt}{\sqrt{t}}. 
\]
Because $D_\Theta$ is $L^2$-invertible, we know that the second summand is finite. For the first summand, we need to take the regularized limit,  ${\rm LIM}_{s\downarrow 0} $,  because $D_\Theta$ is a Dirac-type operator but it is not associated to a Clifford connection. In particular, $\eta_g (\Theta):=\eta_g (D_\Theta)$. 

Recall that the heat kernel and its asymptotic expansion for small time, are $C^1$ in $\Theta$, for finite times. Thus,  we have that.
\begin{equation}
\label{s1 equ}
\lim_{\Theta\downarrow 0} {\rm LIM}_{s\downarrow 0} \frac{1}{\sqrt{\pi}} \int_s^1 \tau^{X}_{g}(D_\Theta \exp (-tD_\Theta^2)) \frac{dt}{\sqrt{t}} =  {\rm LIM}_{s\downarrow 0} \frac{1}{\sqrt{\pi}} \int_s^1 \tau^{X}_{g}(D \exp (-tD^2)) \frac{dt}{\sqrt{t}}. 
\end{equation}
Again by \cite[Proposition 5.29]{PPST}, the righthand side is equal to
\[
\frac{1}{\sqrt{\pi}} \int_0^1 \tau^{X}_{g}(D \exp (-tD^2)) \frac{dt}{\sqrt{t}} < +\infty. 
\]
On the other hand, by the tracial property of $\tau_g^X$ and Duhamel's principle, we compute that 
\[
\begin{aligned}
&\frac{d}{d\Theta} \left( \int_1^T  \tau^{X}_{g}\left(D_\Theta \exp(-tD_\Theta^2)\right) \frac{dt}{\sqrt{t}} \right)\\
=&\int_1^T  \tau^{X}_{g}\left[ \left(1 - 2t \cdot D_\Theta^2 \right)\exp(-tD_\Theta^2)\right] \frac{dt}{\sqrt{t}} \\
=& 2 T^{1/2}  \tau^{X}_{g}\left(  \exp (-TD_\Theta^2) \right)
- 2 \tau^{X}_{g}\left(\exp (-D_\Theta^2) \right). 
\end{aligned}
\]
Changing $\Theta$ in $\phi$ and integrating in $\phi$, from $\phi=0$ to $\phi=\Theta$, we obtain:
\begin{equation}
\label{int for}	
\begin{aligned}
&\int_1^T  \tau^{X}_{g}\left(D_\Theta \exp(-tD_\Theta^2)\right) \frac{dt}{\sqrt{t}}\\
 =&\int_1^T  \tau^{X}_{g}\left(D\exp(-tD^2)\right) \frac{dt}{\sqrt{t}} +2 \int_0^\Theta T^{1/2}  \tau^{X}_{g}\left( \exp (-TD_\phi^2)\right) d\phi 
- 2 \int_0^\Theta  \tau^{X}_{g}\left(\exp (-D_\phi^2)\right) d \phi.
\end{aligned}
\end{equation}
We want to take the limit as $T \to +\infty$ and then $\Theta \downarrow 0$. By Lemma \ref{lemma:eta-gap} the first summand on the righthand side converges to
\[
\int_1^\infty \tau^{X}_{g}\left(D\exp(-tD^2)\right) \frac{dt}{\sqrt{t}} < +\infty, 
\]
and it is obviously  independent of $\Theta$. 
 It is clear that 
\[
\lim_{\Theta \downarrow 0} \int_0^\Theta  \tau^{X}_{g}\left(\exp (-D_\phi^2)\right) d \phi = 0. 
\]
The difficult term is 
\[
\lim_{\Theta \downarrow 0} \lim_{T \to +\infty}  \int_0^\Theta T^{1/2}  \tau^{X}_{g}\left( \exp (-TD_\phi^2)\right) d\phi. 
\]
Note that $D_\phi$ is a lower order invertible perturbation of $D$. For $\phi>0$ small, 
 write 
\[
\exp (-TD_\phi^2)= \exp \left(-T (D^2 + 2\phi D)\right) \exp (-T\phi^2).
\]
Here the operator $D^2 + 2\phi D$ is a generalized laplacian. Because $0$ is isolated in the spectrum of $D$, we can choose $\Theta$ small enough such that for each $\phi\in [0,\Theta]$, $0$ is also isolated in the spectrum of $D^2 + 2\phi D$. In particular,  
\[
\ker (D^2 + 2\phi D)= \ker (D).
\] 
By the decomposition of the heat kernel,  
\begin{equation}
\label{R term}
 \exp \left(-T (D^2 + 2\phi D)\right)= \Pi_{\ker D}+ R(T,\phi)
\end{equation}
with $R(T,\phi)$ going to 0 exponentially in $\Psi^{-\infty}_G (X)$.  We conclude that 
\begin{equation}
\label{limit Dirac}
\begin{aligned}
&\lim_{T\to +\infty}\int_0^\Theta T^{1/2}  \tau^{X}_{g}\left( \exp (-TD_\phi^2)\right) d\phi\\
=&\lim_{T\to +\infty} \int_0^\Theta T^{1/2}  e^{-T\phi^2} \cdot \tau^{X}_{g}\left( \exp\left(-T(D^2+2\phi D)\right)\right) d\phi\\
=&\lim_{T\to +\infty} \int_0^\Theta T^{1/2}  e^{-T\phi^2} \cdot \tau^{X}_{g}\left(  \Pi_{\ker D}\right) d\phi\\
 =& \frac{\sqrt{\pi}}{2} \cdot \tau^{X}_{g}\left(  \Pi_{\ker D}\right). 
\end{aligned}
\end{equation}
Combining the above computation, we conclude that 
\[
\lim_{\Theta\downarrow 0} \eta_g(\Theta) =  \frac{1}{\sqrt{\pi}}  \int_0^\infty \tau^{X}_{g}((D) \exp (-sD^2)) \frac{dt}{\sqrt{t}}+\tau^{X}_{g}\left(  \Pi_{\ker D}\right) = \eta_g (D) + \tau^{X}_g (\Pi_{\ker D}). 
\]
\end{proof}
\end{proposition}

\subsection{The delocalized APS index theorem in the gap case}\label{sect:non-invertible-aps}
We  discuss an extension of the 0-degree delocalized APS index theorem to the case in which the 
boundary operator has a spectral gap.

We assume that
\begin{equation}\label{gap}
\exists\; a >0\;\;\text{such that}\;\; {\rm spec}_{L^2} (D_\partial)\cap (-a,a)=\{0\}
\end{equation}
Let $D$ be  the $b$-Dirac operator on $Y$ associated to $D_0$ on $Y_0$, a  $\ZZ_2$-graded odd  Dirac operator, product-type near the
boundary.
Fix $\Theta\in (0,a)$ and consider the $\ZZ_2$-graded odd $b$-operator given by 
$$D^+ _\Theta:=x^{-\Theta} D^+ x^\Theta\,,\quad D^- _\Theta:=(x^{-\Theta} D^+ x^\Theta)^*= x^{\Theta} D^- x^{-\Theta}\,.$$
This is a $0$-th order perturbation of $D$; indeed 
\begin{equation}\label{total-perturb-gap}
D^+ _\Theta\equiv  D^+  + x^{-\Theta} [D,x^{\Theta}]= D^+ +x^{-\Theta}{\rm cl} (d( x^{\Theta}))\,,
\end{equation}
and similarly for $D^- _\Theta$.
It is known that 
\begin{equation}\label{total-perturb-gap-bdry}
(D^\pm_\Theta)_{\partial}= D_{\partial}+ \Theta\,;
\end{equation}
as we shall now explain. By \cite[Prop. 5.8]{Melrose-Book}, 
$$(x^{-\Theta} D^+  x^{\Theta})_{\partial}= I (D^+, -i  \Theta)$$
and it is well known  that  $I (D^+, -i \Theta)= i (-i \Theta)+D_\partial$ which is $ \Theta+ D_\partial $;
a similar computation holds for $D^- _\Theta$:
$$(x^{\Theta} D^-  x^{-\Theta})_{\partial}= I (D^-, i \Theta);$$
moreover   $I (D^-, i \Theta)= -i (i \Theta)+D_\partial$ 
which is again equal to $D_{\partial}+ \Theta$ as claimed.
This means that  $D_\Theta$ has an invertible boundary operator, provided $\Theta\in (0,a)$.
We {\bf define} the index class in the gap case, as the index class associated to  $D_\Theta$. 
This does not depend on $\Theta$, for $\Theta\in (0,a)$.There is a well defined delocalized APS numeric index given by $\langle \Ind_\infty (D_\Theta),\tau^{Y}_{g}\rangle$.
Proceeding 
as before, using in particular  the relative index class, excision and the relative cyclic cocycle 
$(\tau^{Y,r}_{g},\sigma^{\partial Y}_{g})$ we find that for each $s>0$
\begin{align*}\langle \Ind_\infty (D_\Theta),\tau^{Y}_{g}\rangle &=  \tau^{Y,r}_{g} (e^{- s^2 (D^- _\Theta D^+ _\Theta)})-
 \tau^{Y,r}_{g} (e^{- s^2 (D^+ _\Theta D^- _\Theta))})\\&-\frac{1}{2\sqrt{\pi}} \int_s^\infty \tau^{\partial Y}_{g}((D_{\partial }+\Theta) \exp (-s(D_\partial + \Theta)^2)) \frac{dt}{\sqrt{t}}.
\end{align*}
Proceeding as for the eta integrand for  $D_X+\Theta$, c.f.
Definition \ref{def:perturbed Dirac} and Proposition \ref{prop:eta-conv-gap-bis}, we can use 
the short time 
analysis of the heat kernel on the $b$-manifold $(Y,\mathbf{h})$
in order to prove that 
$$ {\rm LIM}_{s\downarrow 0}\left(  \tau^{Y,r}_{g} (e^{- s^2 (D^- _\Theta D^+ _\Theta)})-
 \tau^{Y,r}_{g} (e^{- s^2 (D^+ _\Theta D^- _\Theta)})\right)
 $$ is well defined.
Thus,  adapting  to our context the arguments in Melrose \cite[Chapter 9]{Melrose-Book}, we have:
\begin{align*}\langle \Ind_\infty (D_\Theta),\tau^{Y}_{g}\rangle &= {\rm LIM}_{s\downarrow 0}\left(  \tau^{Y,r}_{g} (e^{- s^2 (D^- _\Theta D^+ _\Theta)})-
 \tau^{Y,r}_{g} (e^{- s^2 (D^+ _\Theta D^- _\Theta)})\right)\\&-
{\rm LIM}_{s\downarrow 0} \frac{1}{2\sqrt{\pi}} \int_s^\infty \tau^{\partial Y}_{g}((D_{\partial }+\Theta) \exp (-s(D_\partial +\Theta)^2)) \frac{dt}{\sqrt{t}}\\
&= {\rm LIM}_{s\downarrow 0}\left(  \tau^{Y,r}_{g} (e^{- s^2 (D^- _\Theta D^+ _\Theta)})-
\tau^{Y,r}_{g} (e^{- s^2 (D^+ _\Theta D^- _\Theta)})\right)
- \eta_g (\Theta)
\end{align*}
 where the definition of $\eta_g (\Theta)$, the modified eta invariant
associated to $D_{\partial}+\Theta$, has been used. See Definition \ref{def:modified}. Now, the first summand in the right hand side of the above equation depends 
continuously on $\Theta\in (0,a)$
and has a limit as $\Theta\downarrow 0$ equal to
$$\int_{Y_0^g} c_0^g {\rm AS}_g (D_0)\,.$$
See \eqref{main-0-degree}. 
For the second summand we recall Proposition \ref{prop:eta-conv-gap-bis}:
\[
\lim_{\Theta\downarrow 0} \eta (\Theta)=\eta_g (D_{\partial}) + \tau^{\partial Y}_g (\Pi_{\ker D_{\partial}}).
\]

\noindent
We can thus conclude that the following theorem holds:

\begin{theorem}\label{gap-theo}
Let $G$ be a connected, linear real reductive group and let $g\in G$ be a semisimple element.
Let $(Y_0,\mathbf{h}_0)$ be an even dimensional  cocompact $G$-proper manifold with boundary, with metric which is product
type near the boundary.
Let $D_0$ be a $G$-equivariant Dirac operator, product type near the boundary. We assume that the associated
operator on $\partial Y$ satisfies  
$$
\exists\; a >0\;\;\text{such that}\;\; {\rm spec}_{L^2} (D_\partial)\cap (-a,a)=\{0\}.
$$
Let  $(Y,\mathbf{h})$ be the $b$-manifold associated to $(Y_0,\mathbf{h}_0)$ and let $D$ be the  $b$-operator associated to $D_0$. Then for $\Theta\in (0,a)$  there exists a well defined smooth index class 
$\Ind_\infty (D_\Theta)\in K_0 (\mathcal{L}^\infty_{G, s}(Y,E))=K_0 (C^* (Y_0\subset Y, E)^G)$ which is independent of 
$\Theta\in (0,a)$ and denoted briefly $\Ind_\infty (D)$.
For the pairing of this index class  with the 0-degree cyclic cocycle $\tau_g^Y$ the following formula holds:
\[
\langle \Ind_\infty (D),\tau^{Y}_{g}\rangle =
\int_{(Y_0)^g} c^g {\rm AS}_g (D_0)-\frac{1}{2}(\eta_g (D_\partial) + \tau^{\partial Y}_g (\Pi_{\ker D_\partial})).
\]

\end{theorem}

\section{
Large and short time behaviour 
of the heat kernel: proofs }\label{sect:proofs-heat}
In this section we collect technical proofs on the large and short time behaviour 
of the heat kernel of a $L^2$-invertible perturbed Dirac operator $D+C$. While all the results are expected, the proofs turn out to be rather intricate, which is why we have collected separately.

\subsection{Large time behaviour on $G$-proper manifolds without boundary.}
We are interested in  the large time behaviour of the heat operator $\exp (-t(D+C)^2)$ and of the Connes-Moscovici projector $V(t(D+C))$ as elements in the Fr\'echet algebra  $\mathcal{L}^\infty_{G, s} (X)$ with $X$ a cocompact $G$-proper manifold without boundary.

Let $\gamma$ be a suitable path in the complex plane, missing the spectrum of $(D+C)^2$. For example, we can take $\gamma$ to be the path given by the union of the two straight half-lines $$ {\rm Im}z=\pm m {\rm Re}z + b$$
in the half-plane ${\rm Re}z>0$; here $m={\rm tg} \phi>0$, $m$ small, and $b>0$ is smaller than the bottom of the spectrum of $(D+C)^2$. Recall (\ref{half-lines})
\[
\ell^\pm (m,b):= \{z\in\CC\,|\, {\rm Im}z=\pm m {\rm Re}z + b\}.
\]

\medskip
\noindent
We shall consider the functional calculus associated to certain functions. On these functions $f$  we
make the following assumption (Assumption \ref{function-as-CM}):

\begin{assumption}
$f$ is an analytic function on $[0, \infty)$ that is of rapid decay, i.e. $\forall m,n$, there is an constant $M_{m,n}>0$ such that $$\operatorname{max}_{x\in [0, \infty)}|x^m \frac{\partial^n f(x)}{\partial x^n}|< M_{m,n}$$ for all $x\in [0, \infty)\,.$
\end{assumption}

\noindent
We start with the following technical but elementary  lemma.
\begin{lemma}\label{lem:schwartzfunction} Let $f$ satisfy Assumption \ref{function-as-CM}. There is an analytic function $g$ on $[0, \infty)$ of rapid decay such that $f=g'$. 
\end{lemma}

\begin{proof} Consider $g_0(x)=\int_0^x f(t){\rm d}t$. As $f(t)$ is of rapid decay, the integral $\int_0^\infty f(t){\rm d}t$ is finite.  Set $g(x)=g_0(x)-\int_0^\infty f(t){\rm d}t$.  Notice that $\frac{\partial^n g(x)}{\partial x^n}=\frac{\partial ^{n-1} f(x)}{\partial x^{n-1}}$ for $n\geq 1$. To prove the rapid property of $g$, it suffices to check $\operatorname{max}_{x\in [0, \infty]}|x^m g(x)|< C_m$ for some $C_m$. We prove this by observing that there exists $N>0$ such that when $x>N$, $|f(x)|<\frac{M_m}{x^{m+1}}$, and 
\[
x^mg(x)=x^m \int_x^\infty f(t){\rm d}t. 
\]  
When $t>N$, $f(t)$ is bounded by $\frac{M_m}{x^{m+1}}$ and 
\[
|x^m g(x)|\leq x^m \int_x^\infty |f(t)|{\rm d}t \leq x^m \int_x ^\infty \frac{M_m}{t^{m+1}}{\rm dt}\leq \frac{M_m}{m}. 
\]
\end{proof}

\noindent
For a fixed function $f$ satisfying Assumption \ref{function-as-CM}, we choose $\gamma$ a suitable path in the domain of $f$ missing the spectrum of $-(D+C)^2$. As already remarked  both $e^{-z^2}$ and $e^{-\frac{z}{2}}\frac{1-e^{-z}}{z}$ satisfy the 
assumptions of the Lemma and for these we can fix the $\gamma$ as above.

\noindent
Write now 
$$
f(-tD^2)=\frac{1}{2\pi i}\int_\gamma f(-t\lambda^2) (D^2-\lambda)^{-1}{\rm d}\lambda. 
$$
Applying Lemma \ref{lem:schwartzfunction} and integration by parts, we
choose $k$ very large, larger than, say,  $2(\ell+j)$ and obtain the following expression for $f(-tD^2)$,
\[
f(-tD^2)=\frac{(k-1)!}{2\pi i}t^{1-k}  \int_{\gamma} g(-t\lambda^2) (D^2-\lambda)^{-k}{\rm d}\lambda.
\]  

\noindent
To study the large $t$ behaviour of $f(-tD^2)$, we start with the following proposition.  
\begin{proposition}\label{prop:inverse-perturbed}
If the operator $D+C$, $C\in \mathcal{L}^c_{G} (X)$,
is $L^2$-invertible, then 
 $$(D+C)^{-1}=B +  A,\quad\text{with}\quad B\in \Psi^{-1}_{G,c}(X), \;\;A \in \mathcal{L}^\infty_{G, s} (X)$$
 \end{proposition}
 
\begin{proof}
We know that there exists $B\in\Psi^{-1}_{G,c}$ such that 
$$(D+C)B=1-K,\quad K\in \mathcal{L}^c_{G} (X)$$
where we denote the identity by 1. 
Indeed, we simply take $B\in\Psi^{-1}_{G,c}$ such that $DB=1-R$, $R\in\Psi^{-\infty}_{G,c}$ and then we have that $(D+C)B= 1-R-CB$
and we know that $CB\in \mathcal{L}^c_{G} (X)$ by \cite[Lemma 4.17]{PPST}; it now suffices to take $K:= R+CB$.

\smallskip

\noindent
{\bf Claim:} there exists $B'\in \mathcal{L}^\infty_{G,s} (X)$ such that 
$$(D+C)(B-B')=1-K', \quad K'\in \mathcal{L}^\infty_{G, s} (X), \quad \| K' \|<1/2
$$
where the norm is the $L^2$-operator-norm, that is, the norm in the algebra of bounded operators $\mathcal{B}(L^2)$.

\begin{lemma}\label{L2-behaviour-heat}
 In $\mathcal{B}(L^2)$, for $f$ satisfying the hypothesis of Lemma \ref{lem:schwartzfunction}, we have that $\| f(-T(D+C)^2)\| \to 0$ as $T\to +\infty$ and $\| f(-\delta(D+C)^2)\| \to ||f||_{C([0, \infty))}$ as $\delta\downarrow 0$. Furthermore, for any $K\in \mathcal{L}^c_{G}(X)$, $\| \exp(-\delta(D+C)^2)K-K\|\to 0$, as $\delta\downarrow 0$. 
\end{lemma}

\begin{proof}
For the first result, we use the fact that the function $f(T\lambda)$ converges to zero in the sup norm on ${\rm spec}_{L^2}\big( -(D+C)^2\big)$. It follows from spectral calculus  that $f(-T(D+C)^2)$ converges to zero in operator norm. For the second result, by spectral calculus, the norm $||f(-\delta(D+C)^2)||$ is equal to the sup norm of the function $|f(-\delta \lambda^2)|$ on ${\rm spec}_{L^2}(D+C)^2$, which converges to $||f||_{C([0, \infty))}$ as $\delta\to 0$. 

\noindent
We next follow the arguments in \cite[Lemma 1.7.5 (c)]{gilkey-book} to prove $\| \exp(-\delta(D+C)^2)K-K\|\to 0$.  For a properly chosen contour $\gamma$, we compute 
\begin{equation}\label{eq:heat-est}
\begin{split}
e^{-\delta (D+C)^2}K-e^{-\delta}K=&\frac{1}{2\pi i}\int_\gamma e^{-t\lambda}\Big( \big((D+C)^2-\lambda \big)^{-1}+\lambda ^{-1}\Big)K {\rm d}\lambda\\
=&\frac{1}{2\pi i}\int_\gamma e^{-t\lambda}\lambda^{-1} \big((D+C)^2-\lambda \big)^{-1}(D+C)^2 K {\rm d}\lambda. 
\end{split}
\end{equation}
For $t\in (0, 1)$, replace $\lambda$ by $t^{-1}\mu$ and shift the contour $t\gamma$ to $\gamma$. We have the above integral equal to the following expression
\begin{equation}\label{eq:mu}
\frac{1}{2\pi i} \int_\gamma e^{-\mu} \mu^{-1} \big((D+C)^2-t^{-1}\mu\big)^{-1} (D+C)^2K {\rm d}\mu. 
\end{equation}
As $K\in \mathcal{L}^{c}_G(X)$, $(D+C)^2K \in \mathcal{L}^{c}_G(X)$. By the spectral calculus, the operator norm of $\big((D+C)^2-\lambda\big)^{-1}$ has the following bound,
\[
\|\big((D+C)^2-\lambda\big)^{-1}\| \leq M \frac{1}{|\lambda|}, 
\]
for sufficiently large $\lambda$, and
\begin{equation}\label{eq:norm}
\|\big((D+C)^2-t^{-1}\mu\big)^{-1}\| \leq M \frac{t}{|\mu|}. 
\end{equation}
Combining the estimate (\ref{eq:norm}) and Equations (\ref{eq:heat-est}) and (\ref{eq:mu}), we conclude with the following bound
\[
\| \exp(-\delta (D+C)^2)K-\exp^{-\delta}K\|\leq \frac{1}{2\pi}t \int_\gamma |\mu^{-2}\exp^{-\mu}{\rm d}\mu| \| (C+D)^2K\|
\leq \delta M' \| (C+D)^2K\|, 
\]
and then
\[
\begin{split}
\|\exp(-\delta (D+C)^2)K-K\|&\leq \|\exp(-\delta (D+C)^2)K-\exp^{-\delta}K\|+(1-e^{-\delta})\|K\|\\
&\leq 
\delta M' \|(D+C)^2K\|+(1-e^{-\delta})\|K\|\to 0,\ \delta\to 0. 
\end{split}
\]
as required.
\end{proof}

\noindent
We now want to prove the claim.
For $N>>0$ and $\delta>0$ small consider
\begin{equation}\label{bprime}
B':=- \int_\delta^N  (D+C)\exp (-t (D+C)^2)dt\circ K\,.
\end{equation}
This operator is in $\mathcal{L}^\infty_{G, s} (X)$.
Then
$$(D+C)(B-B')=1-K +((D+C)^2\circ \int_\delta ^N  \exp (-t (D+C)^2)dt\circ K)$$
But
$$-(D+C)^2\circ \int_\delta^N  \exp (-t (D+C)^2)dt\circ K= (\exp (-N (D+C)^2)\circ K-\exp (-\delta (D+C)^2))\circ K$$
Thus
$$(D+C)(B-B')= 1- (K - \exp (-\delta (D+C)^2)\circ K)-\exp (-N (D+C)^2)\circ K $$
On the right hand side we have the identity minus an element in $\mathcal{L}^\infty_{G, s} (X)$; moreover, using Lemma \ref{L2-behaviour-heat}, we see that
the $L^2$-operator  norm 
of this term, viz.
$$\| -(K - \exp (-\delta (D+C)^2)\circ K)-  \exp (-N (D+C)^2)\circ K \| $$  
is small if $N$ is big enough and $\delta$ is small enough. We have completed the proof of the claim. 

\medskip
\noindent
With the above established claim,  we are led to consider
$$\left( 1- ((K - \exp (-\delta (D+C)^2))\circ K)- \exp (-N (D+C)^2)\circ K)\right)^{-1}$$
which is an operator of the type $1+F$, with $F\in \mathcal{L}^\infty_{G, s} (X)$, given that  $\mathcal{L}^\infty_{G, s} (X)$
is holomorphically closed.

\medskip
\noindent
It follows from the claim that 
$$(D+C)^{-1}= (B-B')+ (B-B')\circ F= B + (-B' +B\circ F - B'\circ F),\quad \text{with} \quad F\in  \mathcal{L}^\infty_{G, s} (X)\,.$$
This shows that $(D+C)^{-1}$ is the sum of an element in $\Psi^{-1}_{G,c}$ and of an element
in $ \mathcal{L}^\infty_{G, s} (X)$. This completes the proof of Proposition \ref{prop:inverse-perturbed}. 
\end{proof}

\noindent
Given Proposition \ref{prop:inverse-perturbed}, we now go back to the large time behaviour of $f(-t (D+C)^2)$.

\smallskip
\noindent
Recall that 
$(D+C)^2= D^2 + A$, with $A=C^2 + DC+  C D$, $A\in  \mathcal{L}^c_{G, s} (X)$. Notice that $A$ is self-adjoint.
We consider the associated resolvent $((D+C)^2 - \lambda)^{-1}$. In \cite{PP2}, for unperturbed Dirac laplacians,
this was analyzed first for $\lambda$ finite, see \cite[Proposition 2.2]{PP2}, and then for 
$\lambda$ large, using parameter-dependent pseudodifferential operators.
 Assume $|{\rm Im}\lambda| > 0$. We write
$$(D^2 +A -\lambda)= \left (1 + A (D^2-\lambda)^{-1} \right)(D^2-\lambda)$$
and get
$$(D^2 +A -\lambda)^{-1}= (D^2-\lambda)^{-1} \left (1 + A (D^2-\lambda)^{-1}\right) ^{-1}.$$
By  \cite[Proposition 2.2, 2.10]{PP2} we know\footnote{As already remarked, in \cite{PP2} the first two authors worked with a different smooth subalgebra, but the same argument extends to $\mathcal{L}^\infty_{G, s}(X)$.} that
$$(D^2-\lambda)^{-1} = B(\lambda) + C(\lambda)$$
with $B(\lambda)\in \Psi^{-2}_{G,c} (X)$ and of uniform $G$-compact support
and $C(\lambda)\in \mathcal{L}^\infty_{G, s} (X)$. Moreover $C(\lambda)\to 0$ in $\mathcal{L}^\infty_{G, s} (X)$
as $|\lambda|\to +\infty$. In fact, as already observed in \cite{PPST}
we can easily see that $C(\lambda)$ is rapidly decreasing in $\lambda$, with values
in $\mathcal{L}^\infty_{G, s} (X)$. Indeed, $(D^2-\lambda)^{-1} = B(\lambda)(1 + F(\lambda))$ with $1+F(\lambda)= (1+R(\lambda))^{-1}$
and $R(\lambda)$ the remainder of the symbolic parametrix $B(\lambda)$. We argued in \cite[Proposition 2.2, 2.10]{PP2}  that $F(\lambda)$ 
is going to 0 in  $\mathcal{L}^\infty_{G, s} (X)$. However, by the elementary identity
$$1- (1+R(\lambda))^{-1}= R(\lambda) [(1+R(\lambda))^{-1}]$$
we understand that $F(\lambda)$ is in fact rapidly decreasing with values in  $\mathcal{L}^\infty_{G, s} (X)$. 

Summarizing, 
\begin{equation}\label{improved}
(D^2-\lambda)^{-1}=B(\lambda) + C(\lambda)
\end{equation}
with 
\begin{itemize}
\item $B(\lambda)\in \Psi^{-2}_{G,c} (X)$  and of uniform $G$-compact support;
\item  $C(\lambda)$ rapidly decreasing in $\lambda$ with 
values in  $\mathcal{L}^\infty_{G, s} (X)$.
\end{itemize}
 Thus, we have the following equation, 
$$(D^2 +A -\lambda)^{-1}= \left(  B(\lambda) + C(\lambda) \right) \left (1 + A (B(\lambda) + C(\lambda)) \right) ^{-1}.$$
The term  $A (B(\lambda) + C(\lambda))$ is clearly in $\mathcal{L}^\infty_{G, s} (X)$, because we know, by 
\cite[Lemma 4.17]{PPST}, that for $\ell\leq 0$ composition 
defines a continuous map $\Psi^{\ell}_{G,c}(X)\times \mathcal{L}^\infty_{G, s} (X)\to \mathcal{L}^\infty_{G, s} (X)$. Moreover the $L^2$-operator norm of the term $A (B(\lambda) + C(\lambda))$ is small for  $|\lambda|$ large,
given that $C(\lambda)$ is rapidly decreasing and that $B(\lambda)$ is in $ \Psi^{-2}_{G,c} (X)$  and of uniform $G$-compact support.%

On the other hand, $\mathcal{L}^\infty_{G, s} (X)$
is holomorphically closed and so $(1 + A (B(\lambda) + C(\lambda))) ^{-1}$ is equal to $1+M(\lambda)$ with $M(\lambda)\in 
\mathcal{L}^\infty_{G, s} (X)$. We have therefore proved that for $\lambda$ finite, $|{\rm Im}\lambda | > 0$, 
\begin{equation}\label{structure}
(D^2 +A -\lambda)^{-1}= B(\lambda)+ N(\lambda)
\end{equation}
with $B(\lambda)\in \Psi^{-2}_{G,c} (X)$ and of uniform $G$-compact support and $N(\lambda)\in \mathcal{L}^\infty_{G, s} (X)$.
In our analysis we have assumed   $|{\rm Im}\lambda| > 0$; the case $|{\rm Im}\lambda | =0$ is treated using  Proposition \ref{prop:inverse-perturbed}.
Let us now look at  the large $\lambda$ behaviour.
As already remarked we know that  $C(\lambda)$ is rapidly decreasing  in $\lambda$ as an element in $\mathcal{L}^\infty_{G, s} (X)$; 
by the same reasoning given there we  have that the same property is true for $M(\lambda)$ and thus, using the fact that $\mathcal{L}^\infty_{G, s} (X)$
is a Fr\'echet algebra, we conclude that also for $N(\lambda)$ we can state that $N(\lambda)$ is rapidly decreasing  in $\lambda$ as an element in $\mathcal{L}^\infty_{G, s} (X)$.


\noindent
Similarly, we have the following equation,
$$(D^2 + A -\lambda)^{-k}=F(\lambda)+G(\lambda)$$
with
$F(\lambda)\in\Psi^{-2k}_{G,c}(X,\Lambda)$, of $G$-compact support uniformly in $\lambda$, and $G(\lambda)\in \mathcal{L}^\infty_{G, s} (X)$
and going to 0 rapidly in $\lambda$ as an element in $\mathcal{L}^\infty_{G,s} (X)$.
At this point we can proceed exactly as in \cite[Subsection 5.1]{PPST} and show that  if  $p$ is a seminorm  on $\mathcal{L}^\infty_{G,s} (X)$
then 
$$ p(\left( \int_{\gamma} g(-t\lambda) F(\lambda)d\lambda \right) \rightarrow 0  \quad \text{as} \quad t\to +\infty$$ 
and 
$$p(\left(\int_{\gamma} g(-t\lambda) G(\lambda)d\lambda  \right) \rightarrow 0  \quad \text{as} \quad t\to +\infty$$ 
so that $$p\left( \int_{\gamma} g(-t\lambda)(D^2-\lambda)^{-k}d\lambda \right) \rightarrow 0 \quad \text{as} \quad t\to +\infty\,.$$ 
The details are very similar to those given in \cite{PPST} for the unperturbed case, and therefore omitted.\\
We summarise our discussion in the following proposition:

\begin{proposition}(Proposition \ref{prop:large-no-boundary-unperturbed})
If  $C\in \mathcal{L}^c_{G} (X)$, $D+C$ is $L^2$-invertible and $f$ satisfies the assumption of Lemma \ref{lem:schwartzfunction}, then
\[
 f (-t(D+C)^2)\rightarrow 0 \quad\text{in}\quad   \mathcal{L}^\infty_{G, s} (X)\quad\text{as}\quad t\to +\infty.
 \]
In particular, the Connes-Moscovici projector $V(t(D+C))-e_1$ converges to 0 in $M_{2\times 2}\big(\mathcal{L}^\infty_{G, s} (X)\big)$ as $t\to +\infty$.
\end{proposition}

\subsection{Small time behaviour on manifolds without boundary}
We consider as above a self-adjoint element $C\in \mathcal{L}_G^c(X)$ such that $D+C$ is $L^2$-invertible. 
We want to prove Proposition \ref{prop:small-time-heat} which we state again for the benefit of the reader: 
\begin{proposition}
For any $K\in \mathcal{L}_G^c(X)$ we have that $K\exp(-\delta D^2)$ converges to $K$ in $\mathcal{L}_{G, s}^\infty(X)$ as $\delta\to 0$. Consequently, it directly follows from the Volterra series expansion of $\exp(-\delta(D+C)^2)$ that for any $C\in \mathcal{L}_G^c(X)$, 
\[
K\exp(-\delta(D+C)^2)\to K,\ \text{in}\ \mathcal{L}_{G, s}^\infty(X) \text{ as } \delta\downarrow 0.
\]
\end{proposition}


\begin{proof}
We can write 
\begin{equation}\label{eq:K-heat kernel}
\begin{split}
Ke^{-\delta D^2}=&K\frac{1}{2\pi i}\int_\gamma e^{-\delta\lambda} \big(D^2-\lambda \big)^{-1} {\rm d}\lambda. 
\end{split}
\end{equation}
We need to show that for any seminorm $p$ on $\mathcal{L}^\infty_{G,s}(X)$, 
\[
p(Ke^{-\delta D^2}-K)\to 0 \text{ as }  \delta\to 0. 
\]
We start the proof by observing that left multiplication by a smoothing operator $K\in \mathcal{L}_G^c(X)$ defines a continuous  
map on   $\mathcal{L}^\infty_{G,s}(X)$. A similar statement holds for $\Psi^{-2k}_{G,c} (X)$: multiplication by 
$K\in \mathcal{L}_G^c(X)$ is continuous with respect to the seminorms on $\Psi^{-2k}_{G,c} (X)$.
By standard properties of the Bochner integral under the application of a  bounded operator and using the integration by parts trick to replace $(D^2-\lambda)^{-1}$ by $(D^2-\lambda)^{-k}$, we have

\begin{align*}
K\frac{1}{2\pi i}\int_\gamma e^{-\delta \lambda}(D^2-\lambda)^{-1}{\rm d}\lambda&=t^{1-k}\frac{(k-1)!}{2\pi i}K\int_\gamma e^{-\delta \lambda} (D^2-\lambda)^{-k}{\rm d}\lambda\\
&=t^{1-k}\frac{(k-1)!}{2\pi i}K\left( \int_\gamma e^{-\delta \lambda} B(\lambda){\rm d}\lambda+
 \int_\gamma e^{-\delta \lambda} C(\lambda){\rm d}\lambda\right)\\
&=t^{1-k}\frac{(k-1)!}{2\pi i}\left( \int_\gamma e^{-\delta \lambda} K B(\lambda){\rm d}\lambda+
 \int_\gamma e^{-\delta \lambda} K C(\lambda){\rm d}\lambda\right)\\
 &= t^{1-k}\frac{(k-1)!}{2\pi i}\int_\gamma e^{-\delta \lambda} K(D^2-\lambda)^{-k}{\rm d}\lambda\\
 &=\frac{1}{2\pi i}\int_\gamma e^{-\delta \lambda} K(D^2-\lambda)^{-1}{\rm d}\lambda
\end{align*}
where we have expressed $(D^2-\lambda)^{-k}$ as $B(\lambda)+C(\lambda)$ 
with $B(\lambda)\in \Psi^{-2k}_{G,c} (X)$ and of uniform $G$-compact support and $C(\lambda)\in \mathcal{L}^\infty_{G, s}(X)$ 
and converging to zero rapidly as $|\lambda|\to \infty$.
%
Write now
\[
\big(D^2 -\lambda\big)^{-1}= F(\lambda)+ E(\lambda)
\]
with $F(\lambda)\in \Psi^{-2}_{G,c} (X)$ and of uniform $G$-compact support and $E(\lambda)\in \mathcal{L}^\infty_{G, s}(X)$ 
and converging to zero rapidly as $|\lambda|\to \infty$. We then obtain the following expression for $K\exp(-\delta D^2)$, 
\[
Ke^{-\delta D^2}=\frac{1}{2\pi i} \int_\gamma e^{-\delta \lambda}KF(\lambda){\rm d}\lambda+ \frac{1}{2\pi i}\int_\gamma e^{-\delta \lambda} KE(\lambda){\rm d}\lambda. 
\]
In the following, we study the above two integrals separately. \\

\noindent{Part I. $E(\lambda)$ integral.} We observe that $p(e^{-\delta \lambda}KE(\lambda))\leq p(KE(\lambda))$. As $E(\lambda)\in \mathcal{L}^\infty_{G, s}(X)$ converges to 0 rapidly, we have that for every seminorm $p$ on $\mathcal{L}_{G, s}^\infty(X)$, 
\[
\int_\gamma p\big(KE(\lambda)\big){\rm d}\lambda <\infty. 
\]
It follows from the dominated convergence theorem that 
\[
\lim _{\delta\to 0} p\left(\frac{1}{2\pi i}\int_\gamma e^{-\delta \lambda} KE(\lambda){\rm d}\lambda-\frac{1}{2\pi i}\int_\gamma KE(\lambda){\rm d}\lambda\right)=0 
\]
which is what we wanted to show.\\
%
\noindent{Part II. $F(\lambda)$ integral.} We need to use a $G$-equivariant version of the pseudodifferential calculus \cite[Section 1.3]{PPT} for the special case that the Lie groupoid is a connected real reductive Lie group $G$. More precisely, there is a $G$-invariant symbol function $b(\lambda, x, \xi)$ on $T^*X$ such that outside a $G$-compact neighborhood of the zero section in $T^*X$, 
\begin{equation}\label{eq:b-decay}
\vert\vert b(\lambda, x, \xi)\vert\vert \leq C (\vert \vert \xi\vert\vert ^2+\vert\vert \lambda \vert\vert^2)^{-1},\ 
\end{equation}
for sufficiently large $|\lambda|$ and constant $C$ independent of $x$, $\xi$, and $\lambda$.  

Fix $\nu>0$ sufficiently small, i.e. smaller than the injective radius of the riemannian metric on $X$.  Let $\chi(x,y)$ be a $G$-invariant smooth function on $X\times X$ such that $\chi(x,y)$ vanishes when the distance between $x$ and $y$ is larger than $\nu$ and $\chi(x,y)$ takes value 1 when the distance between $x$ and $y$ is less than $\nu/2$.

The operator $F(\lambda)$ has the following expression, for $f\in C^\infty_c(X)$, 
\[
F(f)(x)=\int_X \int_{T_x^*X}  \chi(x,y) e^{i\langle\xi, \exp^{-1}_x(y)\rangle}b(\lambda, x, \xi )f(y){\rm d}y{\rm d}\xi,
\]
where $\exp_x: T_xX\to X$ is the exponential map at $x$, which is bijective for $v\in T_x X$ with $||v||$ less than the injective radius of $X$. 

The operator $K$ has the following form 
\[
K(f)(x)=\int_X  k(x,y) f(y){\rm d}y,
\]
for a smooth $G$-equivariant function $k(x,y)$ on $X\times X$; moreover, because of the $G$-compact support condition there is a positive $\mu>0$ such that $k(x,y)$ vanishes when the distance between $x$ and $y$ is larger than $\mu$. 

The composition of $K$ and $F$ can be computed as follows,
\[
K\circ F(f)(x)=\int_X \int_{T_x^*X} \int_X k(x,y)\chi(y,z) e^{i\langle \xi, \exp^{-1}_y(z)\rangle} b(\lambda, y, \xi) f(z){\rm d}\xi{\rm d}y{\rm d}z. 
\] 

The proof of the following properties is by direct computation and left to the reader. 
\begin{enumerate}
\item As the integration over $y$ is on a compact set, the integration by parts trick on $y$ allows us to replace the symbol function $b(\lambda,y, \xi)$ by a new one $\tilde{b}(\lambda, y, \xi)$ such that 
\[
\vert \vert \tilde{b}(\lambda, y, \xi)\vert\vert\leq \frac{C}{(||\xi||^2+1)^{-2M}(||\xi||^2+||\lambda||^2)} 
\]
where $M$ can be made arbitrarily  large, e.g. larger than the dimension of $X$. 
It follows from the above decay property of $\tilde{b} (\lambda, y, \xi)$ that the following integral 
\[
l(\lambda, x, z):= \int_{T_x^*X} \int_X k(x,y)\chi(y,z) e^{i\langle \xi, \exp^{-1}_y(z)\rangle} b(\lambda, y, \xi){\rm d}\xi{\rm d}y
\]
converges absolutely.  
\item $\forall \lambda$, $l(\lambda, x,z)$ is a $G$-equivariant smooth function on $X\times X$ that vanishes when the distance between $x$ and $z$ is larger than $\mu+\nu$.
\item It follows from estimate (\ref{eq:b-decay}), that for $s,t\geq 0$, there is $C_{s,t}>0$ such that 
\[
\sup_{x,z\in X\times X}||(D_x^s D_z^tl)(\lambda, x, z)|| \leq \frac{C_{s,t}}{||\lambda||^2},
\]
where $D_x$ and $D_z$ are partial differentiations along the $x$ and $z$ directions. 
\end{enumerate}
By (1) and (2), we conclude that the composition $K\circ F$ belongs to $\mathcal{L}^c_G(X)\subset \mathcal{L}^\infty_{G, s}(X)$. Furthermore, by (3), for any seminorm $p$ on $\mathcal{L}^\infty_{G,s}(X)$, we have the bound $p(K F(\lambda))\leq C_p/||\lambda||^2$ and accordingly the convergent Bochner integral as an element in $\mathcal{L}^{\infty}_{G,s}(X)$, 
\[
\int_\gamma KF(\lambda){\rm d}\lambda.
\]
The dominated convergence theorem for the Bochner integral gives that 
\[
\lim _{\delta\to 0}  p\left(\frac{1}{2\pi i}\int_\gamma e^{-\delta \lambda} KF(\lambda){\rm d}\lambda-\frac{1}{2\pi i}\int_\gamma KF(\lambda){\rm d}\lambda\right)=0.  
\]
Summarizing Part I and II, we conclude that in $\mathcal{L}^\infty_{G, s}(X)$, 
\[
\begin{split}
\lim_{\delta\to 0} Ke^{-\delta D^2}&=\lim _{\delta\to 0}\left(\frac{1}{2\pi i} \int_\gamma e^{-\delta \lambda}KF(\lambda){\rm d}\lambda\right)+ \lim_{\delta\to 0}\left(\frac{1}{2\pi i}\int_\gamma e^{-\delta \lambda}KE(\lambda){\rm d}\lambda\right)\\
&=\frac{1}{2\pi i} \int_\gamma KF(\lambda){\rm d}\lambda+ \frac{1}{2\pi i}\int_\gamma K E(\lambda){\rm d}\lambda\\
&=\frac{1}{2\pi i}\int_\gamma K(F+E)(\lambda) {\rm d}\lambda.
\end{split}
\]
We conclude that $Ke^{-\delta D^2}$ is a convergent sequence in $\mathcal{L}^\infty_{G, s}(X)$;
since we have proved that  for the $L^2$-operator norm we do have
$$ \lim_{\delta\downarrow 0} \| Ke^{-\delta D^2} - K \| =  0$$
we conclude that $Ke^{-\delta D^2}$ converges to $K$   in $\mathcal{L}^\infty_{G, s}(X)$\end{proof}

\subsection{Large time behaviour on manifolds with boundary}
Next we tackle the large time behaviour of the heat kernel on $Y$ with boundary under the assumption that there exists $a>0$  
such that 
\[
{\rm Spec}_{L^2} (D+C)\cap [-2a,2a]=\emptyset.
\]
Our ultimate goal is to prove that under this hypothesis
$$ \exp (-t(D+C)^2)\rightarrow 0 \quad\text{in}\quad  {}^b \mathcal{L}^\infty_{G, s} (Y)\quad\text{as}\quad t\to +\infty$$
and more generally that for the Connes-Moscovici projector we have
 \begin{equation*} V(t(D+C))-e_1\to 0\quad\mbox{in}~M_{2\times 2}\Big({}^b\mathcal{L}^\infty_{G, s} (Y)\Big)\quad \mbox{as} ~t\to +\infty.
 \end{equation*} 

To this end we first need to investigate the structure of the inverse $(D+C)^{-1}$. Let us fix $\delta\in (0,a)$. We shall work with the
calculus with bounds ${}^b \Psi^{*}_{G,c}(Y)+ {}^b \mathcal{L}^{\infty,\delta}_{G} (Y)+  \mathcal{L}^{\infty,\delta}_{G, s} (Y)$ but  we shall not write the
$\delta$ in our notation.

\begin{proposition}\label{prop:structure-of-inverse-perturbed} Let $C\in  {}^b \mathcal{L}^c_{G,s} (Y)$ be  self adjoint  and assume that $D+C$ is $L^2$-invertible.
If $B\in {}^b \Psi^{-1}_{G,c}(Y) + {}^b \mathcal{L}^\infty_{G, s} (Y)$ is a full parametrix for $D+C$ then there exist $B'\in \mathcal{L}^\infty_{G, s} (Y)$ and $F \in \mathcal{L}^\infty_{G, s} (Y)$,
with $\mathcal{L}^\infty_{G, s} (Y)$ denoting as usual the residual algebra, such that 
$$(D+C)^{-1}= (B-B') + (B-B') \circ F\,.$$
Consequently,
\begin{equation}\label{inverse-perturbed}
(D+C)^{-1}=B + G\quad\text{with}\quad B\in {}^b \Psi^{-1}_{G,c}(Y) + {}^b \mathcal{L}^\infty_{G, s} (Y)\quad\text{and}\quad G\in \mathcal{L}^\infty_{G, s} (Y)\,.
\end{equation}
\end{proposition}

\begin{proof}
The proof is very similar to the proof of Proposition \ref{prop:inverse-perturbed}, once we use that 
\begin{itemize}
\item $\mathcal{L}^\infty_{G, s} (Y)$ is an ideal in ${}^b \mathcal{L}^\infty_{G, s} (Y)$ and  is holomorphically closed in $C^* (Y_0\subset Y)^G\subset \mathcal{B}(L^2)$;
\item by the analogue of \cite[Lemma 4.17]{PPST}  in the $b$-context (already observed in 
\cite[Proposition 5.22]{PP2}) there is a continuous  module structure
$${}^b \Psi^0_{G,c}(Y)\times {}^b \mathcal{L}^\infty_{G, s} (Y)\rightarrow {}^b \mathcal{L}^\infty_{G, s} (Y);$$ 
\item by an analogue of Lemma \ref{L2-behaviour-heat} for ${}^b \mathcal{L}^\infty_{G, s} (Y)$, we can see that $\exp (-t(D+C)^2)\circ K$ is well defined in $ {}^b \mathcal{L}^\infty_{G, s} (Y)$ and interpolates between
$K$, when $t\downarrow 0$,  and the 0-operator, when $t\uparrow +\infty$,  as a bounded operator on $L^2$.
\end{itemize}
\end{proof}

We now study the resolvent $((D^2+A)-\mu)^{-1}$, $\mu$ in the image of $\gamma$. The trick of writing 
$(D^2 +A -\mu)^{-1}= (D^2-\mu)^{-1} \left (1 + A (D^2-\mu)^{-1}\right) ^{-1}$ is not very useful on manifolds with boundary, because  $A (D^2-\mu)^{-1}$ does not belong
to a holomorphically closed subalgebra (indeed, this operator is  not residual).

 \noindent
 We proceed differently. For $\mu$ is a finite region, say $|\mu |<1$ and always $\mu$ in the straight-half-lines $\ell^\pm (m,b)$,
 see \eqref{half-lines}, 
we can use Proposition 
\ref{prop:structure-of-inverse-perturbed} since 
$$((D^2+A)-\mu)^{-1}= ((D+C)+\sqrt{\mu})^{-1}  ((D+C)-\sqrt{\mu})^{-1}$$
and the two operators on the right hand side have a structure like the one appearing in \eqref{inverse-perturbed}.  We want  now to study the behaviour of the resolvent for $| \mu |$ large,  $\mu\in \ell^\pm (m,b)$,  and prove a result similar to \eqref{improved}.

Consider a symbolic parametrix for $(D^2+A)-\mu)$;  this is equal to a symbolic parametrix for
$(D^2 -\mu)$, call it $B^{{\rm sym}} (\mu)$. Thus, if $(D^2 -\mu) B^{{\rm sym}} (\mu)= 1+ R^{{\rm sym}} (\mu)$, then 
$$(D^2+A)-\mu)B^{{\rm sym}} (\mu)= 1 + R^{{\rm sym}} (\mu)+ A  B^{{\rm sym}} (\mu)\,.$$

\noindent
We set 
$$ S(\mu):=R^{{\rm sym}} (\mu)+ A  B^{{\rm sym}} (\mu),$$
where $S$ stands for symbolic. We recall from 
\cite{PP2} and \cite{PPST} that 
 $R^{{\rm sym}} (\mu)$ (respectively $B^{{\rm sym}} (\mu)$) is a compactly supported $G$-equivariant $b$-pseudodifferential  operator 
with parameter of order $-\infty$ (respectively of order $-2$). Moreover, 
as $A\in {}^b \mathcal{L}^\infty_{G, s} (Y)$, we have that for each $\mu$, the operator 
$A  B^{{\rm sym}} (\mu)$ is in ${}^b \mathcal{L}^\infty_{G, s} (Y)$; indeed, we know from \cite{PP2} that
there is a continuous  module structure
$ {}^b \mathcal{L}^\infty_{G, s} (Y)\times {}^b \Psi^{\ell}_{G,c}(Y)\rightarrow {}^b \mathcal{L}^\infty_{G, s} (Y)$
for any $\ell\leq 0$. As already explained in \cite{PP2} we have that  $AB^{{\rm sym}} (\mu)$ is $O(1/|\mu|)$,
as a map with values in ${}^b \mathcal{L}^\infty_{G, s} (Y)$; indeed, $B^{{\rm sym}} (\mu)$ is $O(1/|\mu|)$
as a map in ${}^b \Psi^{0}_{G,c}$ and when we compose with $A$ we get the desired result. 
We now make a remark that will be useful
at the end of the proof. We use  the particular structure of $C$
and thus assume that $ I(A,\lambda) =\alpha(\lambda) A_\partial$,
with $A_\partial\in \mathcal{L}^c_{G, s} (\partial Y)$ and 
$\alpha (\lambda)$ an entire function which is rapidly decreasing in ${\rm Re}\lambda$
on any horizontal strip $|{\rm Im} \lambda|<N$, in particular in the strip  $\RR\times (-i\delta,i\delta)\subset \CC$, $\delta$ in which was introduced after \eqref{specrtum-perturbed}}. 
Under this additional assumption we   then see that 
 \begin{equation}\label{dlm}
 I(AB^{{\rm sym}} (\mu),\lambda)=\alpha (\lambda) A_\partial 
 I(B^{{\rm sym}} (\mu),\lambda)
 \;\; \text{with}\;\;A_\partial 
 I(B^{{\rm sym}} (\mu),\lambda)
\in  \mathcal{L}^\infty_{G, s} (\partial Y).
 \end{equation} 
We go back to the study of $((D^2+A)-\mu)^{-1}$. The true parametrix of $(D^2+A)-\mu)$ is 
$$B(\mu):= B^{{\rm sym}} (\mu) - \varphi  (I(D^2 + A -\mu)^{-1} I(S (\mu))).$$
We then have  $(D^2+A-\mu) B(\mu)=1- R(\mu)$
with $R(\mu)$, the remainder of the true parametrix, an element in the residual
algebra $\mathcal{L}^\infty_{G, s} (Y)$. Explicitly,
\begin{equation}\label{trueremainder}
R(\mu):= S (\mu)- (D^2+A-\mu) \varphi (I(D^2 + A -\mu)^{-1} I(S(\mu))).
\end{equation}
 We want to investigate the large $\mu$ behaviour 
of $R(\mu)$ and show first of all  that its $L^2$-operator norm is small for $|\mu|$ large. Next we shall use 
the structure of $B(\mu) (1-R(\mu))^{-1})$ in order to show that 
$\int e^{-\mu t} B(\mu) (1- R(\mu))^{-1} d\mu$, that is $\int e^{-\mu t} (D^2 + A -\mu)^{-1} d\mu$,
 goes  to zero in  ${}^b \mathcal{L}^\infty_{G, s} (Y)$ as $t\to +\infty$.\\
Let us analyze the norm of \eqref{trueremainder}. The first summand $S (\mu)$ has certainly small norm for $|\mu|$ large, thanks to the above remarks. We now tackle the second summand
in $R(\mu)$, that is
$(D^2+A-\mu) \varphi (I(D^2 + A -\mu)^{-1} I(S(\mu)))$. Recalling that $ S(\mu):=R^{{\rm sym}} (\mu)+ 
A  B^{{\rm sym}}(\mu)$ we rewrite this term as
\begin{equation}\label{extended-}
(D^2+A-\mu) \varphi (I(D^2 + A -\mu)^{-1} I(R^{{\rm sym}}(\mu)) + 
(D^2+A-\mu) \varphi (I(D^2 + A -\mu)^{-1} I(AB^{{\rm sym}}(\mu)).
\end{equation}
We proceed to analyze these two terms separately, starting with the first one. 
We thus consider the operator $I(D^2 + A -\mu)^{-1} I(R^{{\rm sym}}(\mu))$.
This operator has a Schwartz kernel given in projective coordinates by
\begin{equation}\label{inverse-mellin}
K (s,\cdot,\cdot)= \int_\RR s^{i\lambda} (D^2_\partial + \lambda^2 + I(A,\lambda) -\mu)^{-1} \circ I(R^{{\rm sym}} (\mu), \lambda) d\lambda.
\end{equation}
Observe that $ I(R^{{\rm sym}} (\mu), \lambda)$ defines a function of $\lambda\in\CC$ and $\mu\in\gamma$ with values in $\mathcal{L}^\infty_{G, s} (\partial Y)$;
we consider it as a $\mathcal{L}^\infty_{G, s} (\partial Y)$-valued  function on $(\RR\times (-i\delta,i\delta))\times \gamma$ and observe crucially  that is rapidly 
decreasing as $|{\rm Re} (\lambda) | + |\mu |\to +\infty$.
We can write 
$$(D^2_\partial + \lambda^2 + I(A,\lambda) -\mu)^{-1} = (D^2_\partial + \lambda^2  -\mu)^{-1} \circ (1+ I(A,\lambda) (D^2_\partial + \lambda^2 -\mu)^{-1} )^{-1}.$$
We now recall that $ I(A,\lambda) =\alpha(\lambda) A_\partial$, with $A_\partial\in \mathcal{L}^\infty_{G, s} (\partial Y)$
and 
$\alpha (\lambda)$ an entire function which is rapidly decreasing in ${\rm Re}\lambda$
on any horizontal strip $|{\rm Im} \lambda|<N$.
This means that the  kernel appearing under the sign of integral in the right hand side of  \eqref{inverse-mellin} can be written
as 
$$\left( (D^2_\partial + \lambda^2  -\mu)^{-1} \circ (1+  (\alpha(\lambda) A_\partial) (D^2_\partial + \lambda^2 -\mu)^{-1} )^{-1}\right) \circ  I(R^{{\rm sym}} (\mu), \lambda).$$
From the structure of $(D^2_\partial + \lambda^2  -\mu)^{-1} $, discussed in \cite{PP2}
and also below,
and the joint rapid decay of $I(R^{{\rm sym}} (\mu), \lambda)$
we understand  that this operator is rapidly decreasing on $(\RR\times (-i\delta,i\delta))\times \gamma$ as $|{\rm Re} \lambda | + | \mu |\to +\infty$.
This means that, as in the unperturbed case, $\mu\to \varphi  (I(D^2 + A -\mu)^{-1} I(R^{{\rm sym}} (\mu))$ is rapidly decreasing
with values in ${}^b \mathcal{L}_{G, s}^\infty ( Y)$. This property is unchanged by the composition with $(D^2+A)$
and by multiplication by $-\mu$.\\
Summarizing: the term $(D^2+A-\mu) \varphi (I(D^2 + A -\mu)^{-1} I(R^{{\rm sym}}(\mu))$ 
appearing in 
\eqref{extended-}
 has  norm $<1/4$ for $|\mu|$ large.\\
 We now tackle the second term in \eqref{extended-}, viz.
 $$(D^2+A-\mu) \varphi (I(D^2 + A -\mu)^{-1} I(AB^{{\rm sym}}(\mu))).$$
 Recall, see \eqref{dlm}, that 
$$
 I(AB^{{\rm sym}} (\mu),\lambda)=\alpha (\lambda) A_\partial 
 I(B^{{\rm sym}} (\mu),\lambda)
 \;\; \text{with}\;\;A_\partial 
 I(B^{{\rm sym}} (\mu),\lambda)
\in  \mathcal{L}^\infty_{G, s} (\partial Y)
$$
  We thus consider the operator $I(D^2 + A -\mu)^{-1} I(AB^{{\rm sym}}(\mu))$ with Schwartz kernel in projective coordinates
given by
\begin{equation}\label{inverse-mellin-bis}
K (s,\cdot,\cdot)= \int_\RR s^{i\lambda} (D^2_\partial + \lambda^2 + I(A,\lambda) -\mu)^{-1} \circ I(AB^{{\rm sym}}(\mu), \lambda) d\lambda.
\end{equation}
The operator under the integral sign can be written, proceeding as above, as
$$\left( (D^2_\partial + \lambda^2  -\mu)^{-1} \circ (1+  (\alpha(\lambda) A_\partial) (D^2_\partial + \lambda^2 -\mu)^{-1} )^{-1}\right) \circ  \alpha(\lambda) A_\partial 
 I(B^{{\rm sym}} (\mu),\lambda)
.$$
Recall now the discussion in \cite{PP2}, \cite{PPST} and the remark given here, after Proposition \eqref{prop:inverse-perturbed}  (all this inspired, in turn, from
\cite[Chapter 7]{Melrose-Book}, see in particular pp 284 to 288, and Lemma 7.35 there): 
 for the structure of $(D^2_\partial + \lambda^2  -\mu)^{-1}$
 we have
 $$(D^2_\partial + \lambda^2  -\mu)^{-1}=
  B^{{\rm sym}}_\partial (\lambda,\mu) +
 B^{{\rm sym}}_\partial (\lambda,\mu) \circ  L(\lambda,\mu)$$
 with $B^{{\rm sym}}_\partial (\lambda,\mu)$ of order $1/| \lambda^2  -\mu|$, as a function with values 
in $ \Psi^0_{G,c} (\partial Y)$ and $L(\lambda,\mu)$ an element 
in $\mathcal{L}^\infty_{G, s} (\partial Y)$ rapidly decreasing jointly in $\lambda$ and $\mu$.
We can thus write 
\begin{equation}\label{resolvent-l-m}
(D^2_\partial + \lambda^2  -\mu)^{-1}=
  B^{{\rm sym}}_\partial (\lambda,\mu) +
  G(\lambda,\mu)
  \end{equation}
  with $G(\lambda,\mu)$ rapidly decreasing jointly in $\lambda$ and $\nu$.
Thus 
$$\alpha(\lambda) A_\partial (D^2_\partial + \lambda^2 -\mu)^{-1}$$
can be written as 
$$\Sigma (\lambda,\mu) + \Omega (\lambda,\mu)$$
with $\Sigma (\lambda,\mu)$  a smoothing operator of $G$-compact support
for each $\lambda$ and $\mu$ which is  
$O(1/| \lambda^2  -\mu|)$ and of rapid decay in $\lambda$
and $\Omega (\lambda,\mu)$ rapidly decreasing as a function of $\lambda$ and $\mu$.
This implies, clearly, that 
$\Sigma (\lambda,\mu) + \Omega (\lambda,\mu)$ has small norm for $|\mu|$ large and thus that
$$1 + (\Sigma (\lambda,\mu) + \Omega (\lambda,\mu))$$
is invertible in the holomorphically closed algebra $\mathcal{L}^\infty_{G, s} (\partial Y)$ for $|\mu|$ 
large. Reasoning  as 
in the proof of \cite[Proposition 2.9]{PP2} (see at the end of page 733) 
we conclude that 
$$\alpha(\lambda) A_\partial (D^2_\partial + \lambda^2 -\mu)^{-1} )^{-1}=1+F(\lambda,\mu)$$ with $F(\lambda,\mu)$ 
going to zero for $(\lambda,\mu)$ large. 
Moreover, by applying the same argument we gave after Proposition \eqref{prop:inverse-perturbed} we see that
in fact $F(\lambda,\mu)$ has the same $(\lambda,\mu)$ behaviour as $\Sigma (\lambda,\mu) + \Omega (\lambda,\mu)$.
The operator under the integral sign of \eqref{inverse-mellin-bis} can thus be written as
$$ (D^2_\partial + \lambda^2  -\mu)^{-1} \circ (1+  F(\lambda,\mu)) \circ   \alpha(\lambda) A_\partial I(B^{{\rm sym}}(\mu), \lambda).$$
We rewrite this term as 
$$\left( B^{{\rm sym}}_\partial (\lambda,\mu) +
 G(\lambda,\mu) \right) \circ (1+  F(\lambda,\mu)) \circ  \alpha(\lambda) A_\partial I(B^{{\rm sym}}(\mu), \lambda).$$
In fact, we are interested in understanding the behaviour of this  term when
we  multiply it by $\mu$. Explicitly
\begin{align*}
& \mu B^{{\rm sym}}_\partial (\lambda,\mu)\circ \alpha(\lambda) A_\partial I(B^{{\rm sym}}(\mu), \lambda) +
\mu  G(\lambda,\mu)  \circ \alpha(\lambda) A_\partial I(B^{{\rm sym}}(\mu), \lambda) +\\
&  \mu B^{{\rm sym}}_\partial (\lambda,\mu)\circ F(\lambda,\mu) \circ  \alpha(\lambda) A_\partial I(B^{{\rm sym}}(\mu), \lambda) +
  \mu G(\lambda,\mu) \circ F(\lambda,\mu)\circ \alpha(\lambda) A_\partial I(B^{{\rm sym}}(\mu), \lambda) 
 \end{align*}
 which can be rewritten as 
 \begin{align*}
& (-\lambda^2 +\mu) B^{{\rm sym}}_\partial (\lambda,\mu)\circ \alpha(\lambda) A_\partial I(B^{{\rm sym}}(\mu), \lambda) +
  B^{{\rm sym}}_\partial (\lambda,\mu)\circ  \lambda^2 \alpha(\lambda) A_\partial I(B^{{\rm sym}}(\mu), \lambda)+\\
& \mu  G(\lambda,\mu)  \circ \alpha(\lambda) A_\partial I(B^{{\rm sym}}(\mu), \lambda) +\\
&  (-\lambda^2 + \mu ) B^{{\rm sym}}_\partial (\lambda,\mu)\circ F(\lambda,\mu) \circ  \alpha(\lambda) A_\partial I(B^{{\rm sym}}(\mu), \lambda) +
B^{{\rm sym}}_\partial (\lambda,\mu)\circ F(\lambda,\mu) \circ  \lambda^2  \alpha(\lambda) A_\partial I(B^{{\rm sym}}(\mu), \lambda)+\\
& \mu   G(\lambda,\mu) \circ F(\lambda,\mu)\circ \alpha(\lambda) A_\partial I(B^{{\rm sym}}(\mu), \lambda)
\end{align*}
Let us analyze the various terms. We consider the operator in the first summand of the first line; the operator
$ (-\lambda^2 +\mu )  B^{{\rm sym}}_\partial (\lambda,\mu)$ is uniformly bounded in $\lambda$ and $\mu$, whereas 
$\alpha(\lambda) A_\partial I(B^{{\rm sym}}(\mu), \lambda)$ is rapidly decreasing in $\lambda$ and it is $O(1/|\mu|)$
in $\mu$; both factors are considered as maps with values in $\mathcal{L}^c_G (\partial Y)$; thus overall the first summand is an operator which is rapidly decreasing in $\lambda$ and it is $O(1/|\mu|)$
in $\mu$ with values in  $\mathcal{L}^c_G (\partial Y)$.
The operator on the second summand of the first line 
is rapidly decreasing in $\lambda$ and $O(1/|\mu|^2)$ with values in $\mathcal{L}^c_G (\partial Y)$.
We pass to  the operator in the second  line;
$\mu  G(\lambda,\mu)  \circ \alpha(\lambda) A_\partial I(B^{{\rm sym}}(\mu), \lambda)$.
This  is rapidly decreasing jointly in $\lambda$ and $\mu$ with values in $\mathcal{L}^\infty_{G, s} (\partial Y)$.
The operator in the first summand of the third line is 
rapidly decreasing in $\lambda$ and it is $O(1/|\mu|)$
in $\mu$ with values in $\mathcal{L}^\infty_{G, s} (\partial Y)$; the operator on the second summand of the third line 
is rapidly decreasing in $\lambda$ and $O(1/|\mu|^2)$ with values in $\mathcal{L}^\infty_{G, s} (\partial Y)$. Finally the operator on the fourth line
is rapidly decreasing jointly in $\lambda$ and $\mu$ with values in $\mathcal{L}^\infty_{G, s} (\partial Y)$. \\
We now want to take the inverse Mellin transform of these kernels. When we perform the $\lambda$ integration involved in the inverse Mellin transform we are faced
with the following expressions:
 \begin{align*}
&\int_{\RR}s^{i\lambda} \left((-\lambda^2 +\mu) B^{{\rm sym}}_\partial (\lambda,\mu)\circ \alpha(\lambda) \right)A_\partial I(B^{{\rm sym}}(\mu), \lambda)d\lambda\\
 & \int_{\RR}s^{i\lambda} B^{{\rm sym}}_\partial (\lambda,\mu)\circ  \lambda^2 \alpha(\lambda) A_\partial I(B^{{\rm sym}}(\mu), \lambda)d\lambda\\
 &\int_{\RR}s^{i\lambda} \mu  G(\lambda,\mu)  \circ \alpha(\lambda) A_\partial I(B^{{\rm sym}}(\mu), \lambda)d\lambda\\
&   \int_{\RR}s^{i\lambda} (-\lambda^2 + \mu ) B^{{\rm sym}}_\partial (\lambda,\mu)\circ F(\lambda,\mu) \circ  \alpha(\lambda) A_\partial I(B^{{\rm sym}}(\mu), \lambda)d\lambda\\
&\int_{\RR} s^{i\lambda}B^{{\rm sym}}_\partial (\lambda,\mu)\circ F(\lambda,\mu) \circ  \lambda^2  \alpha(\lambda) A_\partial I(B^{{\rm sym}}(\mu), \lambda)d\lambda\\
& \int_{\RR}s^{i\lambda}\mu   G(\lambda,\mu) \circ F(\lambda,\mu)\circ \alpha(\lambda) A_\partial I(B^{{\rm sym}}(\mu), \lambda)d\lambda
\end{align*}
 Let us focus on the second one above, which is the less obvious.
 We want to show that when we perform the $\lambda$-integration 
 \begin{equation}\label{second-term}
 \int_{\RR}s^{i\lambda} B^{{\rm sym}}_\partial (\lambda,\mu)\circ  \lambda^2 \alpha(\lambda) A_\partial I(B^{{\rm sym}}(\mu), \lambda)d\lambda\,.\end{equation}
 We are left with a kernel that is still $O(1/|\mu|)$ in ${}^b \mathcal{L}^c_{G,s,\RR} ({\rm cyl}(\partial Y))$. We now proceed
 to show this property.\\
  In \eqref{second-term} the term $B^{{\rm sym}}_\partial (\lambda,\mu)$ is a symbolic parametrix for $(D^2_\partial +\lambda^2 -\mu)^{-1}$
 and thus it is the quantization of a $(-2)$-pseudodifferential symbol with parameter, where the parameter is 
  $\mu -\lambda^2$. The term $I(B^{{\rm sym}}(\mu), \lambda)$ is the indicial family of a symbolic parametrix 
  for $(D^2 -\mu)$ and so it is the indicial family of a $(-2)$-b-pseudodifferential operator with parameter $\mu$, obtained by quantization
  of a $(-2)$-symbol with parameter $\mu$.
  These terms are then composed with $A_\partial$, producing an element in $\mathcal{L}_{G,s} (\partial Y)$.
 The dependence of the operator $B^{{\rm sym}}_\partial (\lambda,\mu)\circ  \lambda^2 \alpha(\lambda) A_\partial $ is thus of the type 
 $$\frac{1}{|\mu-\lambda^2|} \cdot \lambda^2 \cdot \alpha(\lambda) $$
 with $\alpha (\lambda)$ rapidly decreasing in $\lambda$.\\
 We are left with $I(B^{{\rm sym}}(\mu), \lambda)$. The operator $B^{{\rm sym}}(\mu)$ is obtained by
 quantization of a symbol of order -2 with parameter $\mu$,  call it $b(r,y,\xi,\eta;\mu)$, near the boundary. We shall eventually apply these reasonings to  $(D^2+A-\mu)^{-k}$, $k$ large; for this last operator, with Schwartz kernel
 given by a $C^{j}$ section, $j=2k -n$, 
 we would then be looking at $I(B^{{\rm sym}}(\mu), \lambda)$, with 
 $B^{{\rm sym}}(\mu)$  obtained by
 quantization of a symbol of order -2k with parameter $\mu$,  call it $b(r,y,\xi,\eta;\mu)$, near the boundary.
 Its indicial family is the $C^j$-kernel locally expressed, near the boundary, as
 $$\int_{\RR^n} e^{i(y-y')\cdot\eta}b(0, y', \lambda,\eta;\mu)d\eta.\,$$ 
(See \cite{Melrose-Book} p. 115 (4.60), (4.61) and Lemma 5.4 and its proof, especially (5,17) and what comes after it.) Summarizing, we get a term that goes like 
 $$\frac{1}{|\mu-\lambda^2|}\cdot \lambda^2 \cdot \alpha(\lambda)\cdot \int_{\RR^n} e^{i(y-y')\cdot\eta}b(0, y', \lambda,\eta;\mu)d\eta$$
 with $\alpha (\lambda)$ rapidly decreasing in $\lambda$, $b$ a symbol-with-parameter $\mu$ of order  $(-2k)$ and symbolic dependance 
 of order  $(-2k)$ in $\lambda$. When we take the inverse Mellin transform
 $$\int_\RR s^{i\lambda} \frac{1}{|\mu-\lambda^2|}\cdot \lambda^2 \cdot \alpha(\lambda)\cdot \int_{\RR^n} e^{i(y-y')\cdot\eta}b(0, y', \lambda,\eta;\mu)d\eta\, d\lambda$$
 and we recall that $\mu$ runs over the straight half-lines $\ell^\pm (m,b):= \{z\in\CC\,|\, {\rm Im}z=\pm m {\rm Re}z + b\}$,
 we get something that is bounded by an expression that goes like $1/| \mu |$.\\
 
 \noindent
On the basis of the above analysis we are able to conclude that
$\mu\to \mu (I(D^2 + A -\mu)^{-k} I(A B^{{\rm sym}}(\mu))$ is $O(1/|\mu|)$,
with values in ${}^b \mathcal{L}^c_{G,s,\RR} ({\rm cyl}(\partial Y))$. This will imply that
$\mu\to 
 (D^2 + A -\mu)\sigma (I(D^2 + A -\mu)^{-k}I(A B^{{\rm sym}}(\mu)))$ is $O(1/|\mu|)$,
with values in ${}^b \mathcal{L}_{G,s} (Y)$.
This is precisely the other term we needed to estimate in \eqref{extended-}.\\
We can then conclude that on the right hand side of $(D^2+A-\mu)^k B(\mu)=1- R(\mu)$ we have  an operator $R(\mu)$,
\begin{equation*}
R(\mu)= S (\mu)- (D^2+A-\mu)^k \varphi (I(D^2 + A -\mu)^{-k} I(S(\mu))),
\end{equation*} 
that has $L^2$-norm smaller than $1/2$ for $|\mu|$ large and that has values in $ \mathcal{L}^\infty_{G, s} ( Y)$; 
since the latter  is holomorphically closed, we 
have that $(1-R(\mu))^{-1}$ is equal to $1 + L(\mu)$ and  $L(\mu)\in \mathcal{L}^\infty_{G, s} ( Y)$; moreover,   $L(\mu)$  is decreasing  to 0 in  Fr\'echet topology as $| \mu| \to +\infty$.\\
Summarizing, for $\mu\in\gamma$, $|\mu|$
large,
$$(D^2 + A -\mu)^{-k}= B(\mu) + B(\mu) \circ L(\mu) $$
with 
$B(\mu):= B^{{\rm sym}} (\mu) - \varphi  (I(D^2 + A -\mu)^{-k} I(S (\mu)))$,
$ S(\mu):=R^{{\rm sym}} (\mu)+ A  B^{{\rm sym}} (\mu)$
and $L(\mu)$ converging to $0$   as $| \mu|\to +\infty$.\\
Now we can write $\exp (-t (D+C)^2)\equiv \exp (-t (D^2+A))$ as
$$\exp (-t(D^2+A))=\frac{(k-1)!}{2\pi i} t^{1-k}\int_\gamma e^{-t\mu} ((D^2+A-\mu)^{-k}d\mu.$$
We thus need to consider
\begin{equation}
\begin{split}
\int_\gamma e^{-t\mu}  B^{{\rm sym}}(\mu) d\mu - \int_\gamma e^{-t\mu}  \varphi (I(D^2+A-\mu)^{-k} I(R^{{\rm sym}} (\mu)))d\mu\\ - \int_\gamma e^{-t\mu}  \varphi (I(D^2+A-\mu)^{-k} I( A  B^{{\rm sym}} (\mu)))d\mu
+ \int_\gamma e^{-t\mu}B^{{\rm sym}} (\mu)\circ L (\mu) d\mu-\\  \int_\gamma e^{-t\mu} 
\varphi (I(D^2+A-\mu)^{-k} I(R^{{\rm sym}} (\mu)))\circ L(\mu) d\mu
- \int_\gamma e^{-t\mu} 
\varphi (I(D^2+A-\mu)^{-k} I(A  B^{{\rm sym}} (\mu) (\mu)))\circ L(\mu) d\mu
\end{split}
\end{equation}
with $B^{{\rm sym}}(\mu)$ of order $-2k$.
Using the information we have just gathered about the operators involved in these 6 integrals and reasoning precisely as in the unperturbed case treated in \cite[Section 5.2]{PPST}  we can conclude that 
 the following proposition holds:

\begin{proposition}(Proposition \ref{prop:large-yes-boundary-perturbed})
Let $C\in {}^b\mathcal{L}^{c}_{G} (Y)$ be self-adjoint,
$C=\sigma (C_{\cyl})$, with
$C_{\cyl}$ obtained by inverse Mellin transform of $\alpha (\lambda)C_{\partial}$,
 with
$C_{\partial}\in \mathcal{L}^c_{G} (\partial Y)$ and 
$\alpha(\lambda)$ an entire function which is rapidly decreasing in ${\rm Re}\lambda$
on any horizontal strip $|{\rm Im} \lambda|<N$. 
 Let $D+C$ be $L^2$-invertible. Then
\[
 \exp (-t(D+C)^2)\rightarrow 0 \quad\text{in}\quad   {}^b \mathcal{L}^\infty_{G, s} (Y)\quad\text{as}\quad t\to +\infty.
 \]
 \end{proposition}

\noindent
Proceeding as in the closed case, considering thus a function $f$ as in Lemma
\ref{lem:schwartzfunction}, we can establish more generally the following. 

\begin{proposition}(Proposition \ref{prop:large-yes-boundary-perturbed-bis}) 
Let $C\in {}^b\mathcal{L}^{c}_{G} (Y)$ be self-adjoint
and 
of the type  explained in the statement of the previous 
Proposition. If  $D+C$ is $L^2$-invertible then
for the Connes-Moscovici projector it holds that
 \[
 V(t(D+C))-e_1\to 0\quad\mbox{in}~M_{2\times 2}\Big({}^b\mathcal{L}^\infty_{G, s} (Y)\Big)\quad \mbox{as} ~t\to +\infty.
 \]
\end{proposition}

\end{document}